\documentclass[10pt]{amsart}
\usepackage[margin=1.03in]{geometry}
\usepackage{amsmath,amsfonts,amssymb,amsthm,epsfig,color,tikz,hyperref,tkz-euclide}

\usepackage{fancyvrb}
\usepackage{MnSymbol} 
\usepackage{float}
\usepackage{graphicx}
\usepackage{verbatim}
\usepackage{amsrefs}
\newcommand{\bb}{\mathbb}

\newcommand{\h}{\bb H}

\newcommand{\R}{\bb R}

\newcommand{\p}{\bb P}

\newtheorem{Theorem}{Theorem}
\newtheorem{Cor}[Theorem]{Corollary}
\newtheorem{Prop}[Theorem]{Proposition}
\newtheorem{lemma}[Theorem]{Lemma}
\newtheorem*{lemma*}{Lemma}
\newtheorem*{theorem*}{Theorem}

\theoremstyle{definition} 
\newtheorem{Def}{Definition}

\newtheorem{remark}{Remark}

\newtheorem{Example}{Example} 
\newtheorem{Non-example}{Non-example}

\numberwithin{equation}{section}
\numberwithin{Def}{section}
\numberwithin{Theorem}{section}

\begin{document}

\title{Currents with corners and Counting Weighted Triangulations}
\author{Tarik Aougab}
\author{Jayadev S. Athreya}
\address{Department of Mathematics, Haveford College, 370 Lancaster Ave
Haverford, PA 19041, USA}
\email{taougab@haverford.edu}
\address{Department of Mathematics, University of Washington, Padelford Hall, Seattle, WA 98195, USA}
\email{jathreya@uw.edu}
\maketitle

\begin{abstract}
Let $\Sigma$ be a closed orientable hyperbolic surface. We introduce the notion of a \textit{geodesic current with corners} on $\Sigma$, which behaves like a geodesic current away from certain singularities (the ``corners''). We topologize the space of all currents with corners and study its properties. In several ways that we make precise, the space of geodesic currents is ``at infinity'' in the space of currents with corners, in the sense that they arise as limits in which the corners are pushed out to infinity. We also emphasize the following analogy: currents are to closed curves, as currents with corners are to \textit{graphs} on $\Sigma$. 

We prove that the space of currents with corners shares many properties with the space of geodesic currents, although crucially, there is no canonical action of the mapping class group nor is there a continuous intersection form. To circumvent these difficulties, we focus on those currents with corners arising from harmonic maps of graphs into $\Sigma$. This leads to the space of \textit{marked harmonic currents with corners}, which admits a natural Borel action by the mapping class group, and an analog of Bonahon's~\cite{Bonahon} compactness criterion for sub-level sets of the intersection form against a filling current. 

As an application, we consider an analog of a curve counting problem on $\Sigma$ for triangulations. Fixing an embedding $\phi$ of a weighted graph $\Gamma$ into $\Sigma$ whose image $\phi(\Gamma)$ is a triangulation of $\Sigma$, let $N_{\phi}(L)$ denote the number of mapping classes $f$ so that a weighted-length minimizing representative in the homotopy class determined by $f \circ \phi$ has length at most $L$. In analogy with theorems of Mirzakhani~\cite{Mirzakhani}, Erlandsson-Souto~\cite{ErlandssonSouto}, and Rafi-Souto~\cite{RafiSouto}, we prove that $N_{\phi}(L)$ grows polynomially of degree $6g-6$ and the limit 
\[ \lim_{L \rightarrow \infty} \frac{N_{\phi}(L)}{L^{6g-6}}\]
exists and has an explicit interpretation depending on the geometry of $\Sigma$, the vector of weights, and the combinatorics of $\phi$ and $\Gamma$.  
\end{abstract}

\section{Introduction} 

\paragraph*{\bf Counting geodesics} Let $\Sigma$ be a closed orientable surface of genus at least $2$ and equip $\Sigma$ with an arbitrary complete hyperbolic metric. Classical work of Huber~\cite{Huber} and Margulis~\cite{Margulis} yields an asymptotic for the number of primitive closed geodesics on $\Sigma$ under a given length, $L$. Indeed, letting $N_{\Sigma}(L)$ denote the number of primitive closed geodesics with length at most $L$, one has
\begin{equation} \label{Huber}
N_{\Sigma}(L) \sim \frac{e^{L}}{L}.
\end{equation}
In a sense that can be made precise, as $L \rightarrow \infty$, ``most'' curves with length at most $L$ are not embedded. It took the revolutionary work of Mirzakhani~\cite{Mirzakhani}, building on work of Birman-Series~\cite{BirmanSeries}, to fully understand the asymptotics of the much more slowly growing number $\mathcal{S}_{\Sigma}(L)$ of \textit{simple} closed geodesics. In simplest terms, Mirzakhani~\cite{Mirzakhani} recasts the counting problem as a convergence statement for a family of measures on the space $\mathcal{ML}(\Sigma)$ of measured geodesic laminations; the asymptotics for $\mathcal{S}_{\Sigma}(L)$ are encoded by properties of the limiting measure, and the argument uses subtle geometric properties of the mapping class group action on $\mathcal{ML}(\Sigma)$ to identify this limit. 

\paragraph*{\bf Intersection patterns} In subsequent work, Mirzakhani~\cite{Mirzakhani2} generalized these results to obtain asymptotics for the number of closed geodesics with length at most $L$ and with some prescribed number of self-intersections, and these sorts of questions were explored further by Sapir~\cites{Sapir1, Sapir2}. This required a suite of new ideas because once the curves in question have self-crossings, there is no obvious way to relate the desired count to a measure on $\mathcal{ML}(\Sigma)$. Recent and beautiful work of Erlandsson-Souto~\cites{ErlandssonSouto, ErlandssonSoutobook} reproves these counting results in ways that are inspired by Mirzakhani's original techniques; to do this, they replace $\mathcal{ML}(\Sigma)$ with the space of \textit{geodesic currents} on $\Sigma$, and they invent new organizational tools (such as the notion of a \textit{radalla}) for packaging together large collections of closed curves, analogous to the multi-curves carried by a given train track. See in particular the new monograph~\cite{ErlandssonSoutobook} for an excellent exposition.

\medskip


\medskip

\paragraph*{\bf Counting triangulations} We aim to count \textit{weighted triangulations} on $\Sigma$, and inspired by Erlandsson-Souto~\cites{ErlandssonSouto, ErlandssonSoutobook} and follow-up work of Rafi-Souto~\cite{RafiSouto}, we achieve this by building a space analogous to the space of currents but in which (embeddings of) graphs play the role of closed curves, and we then convert the counting problem into a statement about convergence of measures on this space of \textit{currents with corners}. To state our main application, let $\Gamma = (V,E)$ be a finite, connected, simple graph equipped with a positive edge weighting $w: E \rightarrow \mathbb{R}_{> 0}$. Given an embedding $\phi: \Gamma \rightarrow \Sigma$, there exists a map $\tilde{\phi}$ in the homotopy class of $\phi$ which minimizes the \textit{weighted length} $\ell_{\Sigma}(\Gamma, [\phi])$ of $\Gamma$, defined as 
\begin{equation} \label{weighted length}
\ell_{\Sigma}(\Gamma, [\phi]) = \sum_{e \in E} w(e) \cdot \ell_{\Sigma} \tilde{\phi}(e).
\end{equation}
Moreover, in our setting where we will assume that $\phi$ induces a triangulation of $\Sigma$, one can show that the map $\tilde{\phi}$ is unique. We refer to the pair $(\Gamma, [\phi])$ as a \textit{marked weighted graph}, and we let $\ell_{\Sigma}(\Gamma, [\phi])$ be the \textit{weighted length} of the marked weighted graph $(\Gamma, [\phi])$. Then given any mapping class $f \in \mathcal{MCG}(\Sigma)$, one can consider the new marked weighted graph $(\Gamma, [f \circ \phi])$ and its associated weighted length. We prove: 

\begin{Theorem} \label{main theorem!} Let $(\Gamma, [\phi])$ be a marked weighted triangulation of $\Sigma$. Then 
\[ \lim_{L \rightarrow \infty} \frac{ \#(f \in \mathcal{MCG}(\Sigma): \ell_{\Sigma}(\Gamma, [f \circ \phi]) \leq L)}{L^{6g-6}} = \frac{n_{(\Gamma, [\phi])} \cdot m_{\Sigma}}{\textbf{m}}, \]
where 
\[ n_{(\Gamma, [\phi])} := \mu_{Th}\left( \left\{\lambda \in \mathcal{ML}(\Sigma): (\Gamma, [\phi]) \hspace{1 mm} \mbox{admits a representative whose (weighted) intersection with} \hspace{1 mm} \lambda \hspace{1 mm} \mbox{is at most} \hspace{1 mm} 1 \right\}\right), \]
and
\[ m_{\Sigma} = \mu_{Th}(\left\{\lambda \in \mathcal{ML}(\Sigma) : \ell_{\Sigma}(\lambda) \leq 1 \right\}),\]
where $\mu_{Th}$ denotes the Thurston measure on the space of measured laminations (see, for example, Masur~\cite{Masur85} or Penner-Harer~\cite{PennerHarer}).

Finally, $\textbf{m}$ is the total integral of $m_{\Sigma}$ over the moduli space $\mathcal{M}$ with respect to the Weil-Petersson volume form (see, for example~\cite{FarbMargalit} for the definition): 
\[ \textbf{m} = \int_{\mathcal{M}} m_{\Sigma} dvol_{WP}. \]
\end{Theorem}


 \paragraph*{\bf Normalizations} Masur~\cite{Masur85} showed that the Thurston measure $\mu_{Th}$ is the unique mapping class group invariant measure on $\mathcal{ML}$  in the Lebesgue class, \emph{up to scaling}. For our results, the normalization of Thurston measure needs to be specified, since this normalization will affect the right hand side of Theorem~\ref{main theorem!} and Theorem~\ref{main theorem}, as any scaling factor would occur twice in the numerator and only once in the denominator. We will, following Rafi-Souto~\cite{RafiSouto}*{Remark, page 8}, use the measure defined by the symplectic structure on $\mathcal{ML}.$ For how this normalization is related to other natural normalizations, see the work of Arana-Herrera~\cite{AH}, or Monin-Telpukhovskiy~\cite{Monin}.

\paragraph*{\bf Analogies} Theorem \ref{main theorem!} resembles the analogous counts for closed curves and for geodesic currents: the limit converges to an expression involving a constant that depends on the Thurston measure of the set of measured laminations intersecting the initial choice of graph $(\Gamma, [\phi])$ at most once, the Thurston measure of the unit length ball in $\mathcal{ML}(\Sigma)$, and then a constant depending only on the topological type of $\Sigma$. Theorem \ref{main theorem!} is a direct analog of work of Rafi-Souto~\cite{RafiSouto}, which extends the theory of curve counting to \textit{current counting}-- counting in an orbit of a filling geodesic current. Indeed, our motivation for building the theory of currents with corners was to be able to implement their arguments, suitably modified, to prove Theorem \ref{main theorem!}.

\paragraph*{\bf Harmonic maps} A map $\phi$ of a weighted graph $\Gamma$ is said to be \textit{harmonic} if $\phi$ restricted to each edge is a constant speed geodesic segment, and for each vertex $v$ of $\Gamma$, one has the \emph{balance condition}
\begin{equation}\label{equation:balance} \sum_{e \sim v} w(e) \phi'_{e}(0) = 0, \end{equation}
where $\phi'_{e}(0)$ denotes the tangent vector in $T_{\phi(v)}\Sigma$ to the geodesic segment $\phi(e)$ at $v$. Completely analogous to the theory of harmonic maps between manifolds, harmonic graph maps are minimizers of the energy functional 
\[ \mathcal{E}(\phi) = \sum_{e \in E(\Gamma)} w(e) \cdot \int ||\phi'_{e}(t)||^{2} dt. \]
See, for example, Chen~\cite{Chen} for more applications of the study of harmonic maps whose domains are Riemannian simplicial complexes. In a homotopy class of maps such that the image of each component of $\Gamma$ is homotopically non-trivial and not homotopic to a closed curve, there exists a unique harmonic representative. Our methods allow us to count harmonic representatives with respect to length, and in fact we prove that the number of such graphs grows identically to the count in Theorem \ref{main theorem!}: 

\begin{Theorem} \label{harmonic count} Let $(\Gamma, [\phi])$ be as above, and let $\ell_{\Sigma}^{h}(\Gamma, [\phi])$ denote the weighted length of the unique harmonic representative in $[\phi]$. Then 
\[ \lim_{L \rightarrow \infty} \frac{ \#(f \in \mathcal{MCG}(\Sigma): \ell_{\Sigma}^{h}(\Gamma, [f \circ \phi]) \leq L)}{L^{6g-6}} = \frac{n_{(\Gamma, [\phi])} \cdot m_{\Sigma}}{\textbf{m}}. \]    
\end{Theorem}

\paragraph*{\bf Why triangulations?} The reason we consider triangulations, as opposed to general embeddings of graphs, is that at some point in our arguments we we will require the fact, due to de Verdiere~\cite{deVerdiere} that \emph{harmonic} triangulations are always embeddings. Much of the machinery we build applies to arbitrary graphs, and there is reason to believe the main theorems hold more generally for any graph that fills the surface. We plan to address this in future work.

\paragraph*{\bf Harmonic triangulations and approximate length minimizers} As a corollary of Theorems \ref{main theorem!} and \ref{harmonic count}, we prove that ``most'' harmonic triangulations are close to being length-minimizers: 

\begin{Cor} \label{length minimizers vs harmonic} For any $\epsilon>0$, 
\[ \lim_{L \rightarrow \infty} \frac{\#\left( f \in \mathcal{MCG}: \frac{\ell_{\Sigma}^{h}(\Gamma, [f \circ \phi])}{ \ell_{\Sigma}(\Gamma, [f \circ \phi])} \leq 1+ \epsilon, \ell_{\Sigma}(\Gamma, [f \circ \phi])\le L  \right) }{\#\left( f \in \mathcal{MCG}: \ell_{\Sigma}(\Gamma, [f \circ \phi])\le L    \right) } = 1. \]

\end{Cor}

\paragraph*{\bf Quasi-minimizers} The content of Corollary \ref{length minimizers vs harmonic} is that most harmonic triangulations are \emph{quasi-minimizers}, in the sense that the ratio of their length to the length-minimizing value in their homotopy class is nearly $1$, as opposed to simply saying that the ratio is uniformly bounded in terms of some constant depending only on $\Gamma$. Indeed, harmonic maps minimize a weighted $l^{2}$-norm. Since the weighted $l^{1}$-norm of an $n$-dimensional vector is bounded above and below by the corresponding $l^{2}$-norm up to a multiplicative error of $\sqrt{n}$, it is clear that a harmonic triangulation can not be too far from being a length minimizer, where ``too far'' is quantifiable as some increasing function of the total number of edges (see for example Lemma \ref{harmonic to length} in Section \ref{intersection} where this is made precise). 

\paragraph*{\bf Currents with corners} We approach Theorem \ref{main theorem!} by introducing a new space of \textit{currents with corners}, which we denote by $\mathcal{GCC}(\Sigma)$. A point $c \in \mathcal{GCC}(\Sigma)$ can be identified with a $\pi_{1}(\Sigma)$-invariant locally finite Radon measure on the space of geodesic \textit{segments} of the universal cover, $\mathbb{H}^{2}$. There is a variety of difficulties we must overcome in order to use $\mathcal{GCC}(\Sigma)$ in service of Theorem \ref{main theorem!}; principle among them is the fact that $\mathcal{GCC}(\Sigma)$ admits no obvious action by the mapping class group. Indeed, the action of $\mathcal{MCG}(\Sigma)$ on the space of classical geodesic currents comes from the fact that an isotopy class $[f]$ of homeomorphisms induces a well-defined homeomorphism $\tilde{f}$ of the boundary at infinity of the universal covering; since a current is a measure on the space of bi-infinite geodesics, the action of $[f]$ on a given current is then easy to define using the pushforward associated to $\tilde{f}$. On the other hand, the image under a mapping class of a finite length geodesic segment is not well-defined. 

\paragraph*{\bf Harmonic currents with corners} To circumvent this issue, we focus on a subspace $\mathcal{HC}(\Sigma) \subset \mathcal{GCC}(\Sigma)$ which we call the \textbf{space of harmonic currents with corners}. Loosely speaking, $\mathcal{HC}(\Sigma)$ consists of the currents with corners arising from harmonic maps of graphs into $\Sigma$, in the sense of Sunada~\cite{Sunada} and Kajigaya-Tanaka~\cite{KajigayaTanaka}. In fact, we specify even further to the space $\mathcal{HC}_{k}(\Sigma)$ consisting of such graphs with at most $k$ edges and we show that this space enjoys a variety of helpful topological properties; we summarize the most important such properties in the following theorem: 

\begin{Theorem} \label{properties of k-harmonic currents} The space $\mathcal{HC}_{k}(\Sigma)$ is closed in $\mathcal{GCC}(\Sigma)$. There is an intersection form $i( \cdot, \cdot )$ defined on the space of currents with corners, in analogy with Bonahon's~\cite{Bonahon} intersection form on the space of currents, and if $\mu \in \mathcal{HC}_{k}(\Sigma)$ is a filling harmonic current with corners, the set 
\[ C_{R}(\mu) := \left\{\nu \in \mathcal{HC}_{k} : i(\mu, \nu) \leq R \right\} \]
is pre-compact. Moreover, there is a dense subset $B \subset \mathcal{GC}(\Sigma)$ of the space of geodesic currents such that the function 
\[ i(\cdot, b) : \mathcal{HC}_{k} \rightarrow \mathbb{R}\]
is continuous for all $b \in B$. 
\end{Theorem}

\paragraph*{\bf Marked harmonic currents} By decorating each current in $\mathcal{HC}_{k}$ with a \textit{marking} that keeps track of how the underlying graph is mapped to $\Sigma$, we construct a new space $\widetilde{\mathcal{HC}}_{k}$ which we call the space of \textit{marked harmonic $k$-currents} -- see Section \ref{harmonic currents} for details. We show that this space admits a reasonably well-behaved action by the mapping class group: 

\begin{Theorem} \label{borel} There is a connected, open, and dense subset of $\widetilde{\mathcal{HC}}_{k}$ admitting an action of $\mathcal{MCG}(\Sigma)$ by homeomorphisms. 
\end{Theorem}

\paragraph*{\bf Technical difficulties} The theory of currents with corners enjoys much of the same structure that comes from classical currents, but it also lacks some crucial elements, for instance:
\begin{itemize}
    \item The intersection form $i( \cdot, \cdot)$ lifts to a form on $\widetilde{\mathcal{HC}}_{k}$, but it is neither continuous nor invariant under the $\mathcal{MCG}$-action (see Example \ref{intersection not mcg invariant}). 
    \item The $\mathcal{MCG}$-action may not be continuous. 
\end{itemize} 

Our intent was to build a space of currents with corners whose behavior captured all features of the classical space of currents upon which the work of Rafi-Souto~\cite{RafiSouto} relies. However, because of these key differences, there are several places -- see for instance the end of Section \ref{intersection} and Section \ref{establishing invariance}-- where we have to work harder to achieve the main counting result. 

\subsection{Connections to other work} Recently, Erlandsson-Souto~\cite{ESTripods} have shown how to obtain counting asymptotics (as length increases) for immersed trivalent graphs in a hyperbolic surface where all angles at vertices are $2\pi/3$.  The following analogy describes how their result relates to ours (at least in spirit): counting all curves up to length $L$ is to counting the curves in a given mapping class group orbit up to length $L$, as this result is to our work.

\subsection{Outline}\label{sec:outline}

In Section \ref{currents with corners}, we introduce the space $\mathcal{GCC}(\Sigma)$ of currents with corners, topologize it, and explore some basic examples. We also introduce the intersection form mentioned in Theorem \ref{properties of k-harmonic currents} and study some of its properties. In Section \ref{harmonic currents}, we define the space $\mathcal{HC}(\Sigma)$ of \textit{harmonic} currents with corners and the subspace $\mathcal{HC}_{k}(\Sigma)$. We then prove Theorem \ref{properties of k-harmonic currents}, setting the stage for Theorem \ref{main theorem!}. In Finally, in Section \ref{main theorem}, we prove Theorem \ref{main theorem!}, following the outline of the main theorem of Rafi-Souto for counting in an orbit of a filling geodesic current.

\paragraph*{\bf Acknowledgements} T.A was supported by NSF Grant 1939936. J.S.A.
was partially supported by NSF grant DMS
2003528; the Pacific Institute for the Mathematical Sciences; the Royalty Research Fund
and the Victor Klee fund at the University of Washington; and this work was concluded
during his term as the Chaire Jean Morlet at the Centre International de Recherches
Mathematique-Luminy. The authors thank Ryokichi Tanaka for many helpful conversations and guidance through the literature on harmonic maps of graphs.


\section{Preliminaries} \label{prelim}

In all that proceeds, when we are discussing asymptotics, we will say\textit{``$f(n)$ is roughly $Kg(n)$''} to mean that $\lim_{n \rightarrow \infty} f(n)/g(n) = K$.

\subsection{Topologies on spaces of measures} Given a space $X$, a \textbf{Radon measure} (see, for example, Folland~\cite{Folland}*{page 212}) $\mu$ on $X$ is a Borel measure that is \textit{regular}, meaning that any open set $U$ can be inner-approximated by compact subsets, and any Borel set $A$ can be outer-approxiated by open sets: 

\[ \mu(U)= \sup \left( \mu(C): C \subset X, C \hspace{1 mm} \mbox{compact} \right); \]
\[ \mu(A) = \inf \left( \mu(U): A \subset U, U \hspace{1 mm} \mbox{open} \right). \]
Some authors also require Radon measures to be locally finite, meaning that every compact set has finite measure. 

\paragraph*{\bf Narrow topology} Let $\mathcal{R}(X)$ denote the set of all Radon measures on $X$. Then the \textbf{narrow topology} on $\mathcal{R}(X)$ is characterized by the convergence rule that $\lim_{n \rightarrow\infty} \mu_{n} \rightarrow \mu$ if and only if 
\begin{equation} \label{narrow convergence}
\lim_{n \rightarrow \infty} \int f d\mu_{n}  = \int f d\mu, 
\end{equation}
for all bounded real-valued continuous functions $f$ on $X$. 

\paragraph*{\bf Wide topology} The \textbf{wide topology} on $\mathcal{R}(X)$ is an \textit{a priori} weaker topology on $\mathcal{R}(X)$ in which $\mu_{n} \rightarrow \mu$ if and only if Equation \ref{narrow convergence} holds for all \textit{compactly supported} continuous functions $f$. We will refer to this as \textbf{weak-}$^{\ast}$ convergence. Let $\mathcal{P}(X) \subset \mathcal{R}(X)$ be the subspace of probability measures. Under certain assumptions on $X$, the weak and narrow topologies coincide for probability measures: 

\begin{lemma} \label{weak and narrow}
The weak and narrow topologies coincide on $\mathcal{P}(X)$ in the event that $X$ is locally compact Hausdorff. 
\end{lemma}

\paragraph*{\bf Portmanteau theorem} We emphasize a point of caution: in either the wide or narrow topologies, it is \textit{not} necessarily the case that the limit of measures $\mu_{n}(A)$ of some Borel set $A$ coincides with $\mu(A)$ when $\mu_{n} \rightarrow \mu$. However, in the narrow topology, this holds under an additional topological assumption on $A$ which we will describe in the following proposition-- using Lemma \ref{weak and narrow}, we state it in a context where it is also true for the wide topology -- this is the so-called \textit{Portmanteau theorem} (see, for example, Billingsley~\cite{Billingsley}*{\S 1.2}: 

\begin{Prop} \label{Portmanteau} If $(\nu_{n})_{n=1}^{\infty}$ is a sequence of probability measures on a locally compact Hausdorff space $X$ converging weak-$^{\ast}$ to $\nu$ and $A \subset X$ is any Borel set so that $\nu(\partial A)= 0$, then 
\[ \lim_{n \rightarrow \infty} \nu_{n}(A) = \nu(A).\]
Here, $\partial A$ denotes the \textit{topological boundary} of $A$: those points $x \in A$ for which any open neighborhood of $x$ intersects $A^{c}$. Moreover, 
\[ \liminf_{n \rightarrow \infty} \nu_{n}(U) \geq \nu(U)\]
for any open set $U$. Conversely, if $\lim_{n \rightarrow \infty} \nu_{n}(A) = \nu(A)$ for all Borel sets $A$ with $\nu(\partial(A)) = 0$, then $(\nu_{n})$ weak-$^{\ast}$ converges to $\nu$. 
\end{Prop}

\subsection{Geodesic currents}\label{sec:currents}

Let $\h^{2}$ denote the hyperbolic plane. The space of \textit{oriented} bi-infinite geodesics in $\h^2$ -- which we denote by $\tilde{G}(\mathbb{H}^{2})$-- is topologically identified with 
\[ (S^{1} \times S^{1} \setminus \Delta), \]
where $\Delta$ is the diagonal in $S^1 \times S^1$. Indeed, any oriented geodesic is given by a pair of distinct points on the circle $S^1$ at infinity. The space of \textit{unoriented} bi-infinite geodesic segments-- denoted $G(\mathbb{H}^{2})$ -- is obtained from $\tilde{G}(\mathbb{H}^{2})$ after quotienting by the $\mathbb{Z}_{2}$ action coming from flipping the orientation. If $\gamma \in \tilde{G}(\mathbb{H}^{2})$, let $\overline{\gamma}$ denote its unoriented class in $G(\mathbb{H}^{2})$. 

\paragraph*{\bf Universal covers} Given a closed, oriented hyperbolic surface $\Sigma$, the universal cover $\tilde{\Sigma}$ is identified with $\mathbb{H}^{2}$ and we thus refer to $G(\tilde{\Sigma})$ as $G(\mathbb{H}^{2})$ equipped with the isometric action of $\pi_{1}(\Sigma)$ coming from deck transformations. Note that $G(\tilde{\Sigma})$ depends only on the algebra of $\pi_{1}(\Sigma)$ and not on particular geometric features of the metric.

\begin{Def}\label{def:current} (\textit{measures on the space of geodesics}) A \textbf{geodesic current} on $\Sigma$ is a locally finite Radon $\pi_{1}(\Sigma)$-invariant measure on $G(\tilde{\Sigma})$.
\end{Def}

\paragraph*{\bf Transverse measures} We can also define geodesic currents in a way that parallels the idea of a measured geodesic lamination, by considering transverse measures on the projectivized tangent bundle $\mathbb{P}T(\Sigma)$, the quotient of the unit tangent bundle by the antipodal map. The projectivized tangent bundle comes equipped with the \textit{geodesic foliation} $\mathcal{F}$, the leaves of which are trajectories under the geodesic flow on $T\Sigma$.

\begin{Def}  \label{bundle def} (\textit{transverse measures on the projectivized tangent bundle}) A \textbf{geodesic current} is a family of measures 
\[ \left\{ \mu_{\tau}: \tau \hspace{1 mm} \mbox{is a codimension-1 submanifold of } \mathbb{P}T\Sigma \hspace{1 mm} \mbox{transverse to the geodesic foliation} \right\},\]
satisfying the following invariance property. Given submanifolds $\tau_{1}, \tau_{2}$, points $x_{1} \in \tau_{1}, x_{2} \in \tau_{2}$ contained on the same leaf of $\mathcal{F}$, and a holonomy diffeomorphism $\phi: U_{1} \rightarrow U_{2}$ from a neighborhood $U_{1}$ of $x_{1}$ to $U_{2}$ of $x_{2}$, one has that $\phi_{\ast} \mu_{\tau_{1}} = \mu_{\tau_{2}}$. 
\end{Def}

\paragraph*{\bf The space of currents} We will denote the space of all geodesic currents by $\mathcal{GC}(\Sigma)$; since they arise as Radon measures on a space, $\mathcal{GC}$ comes equipped with a weak-$^{\ast}$ topology-- a sequence $(\mu_{n})$ converges to $\mu \in \mathcal{GC}(\Sigma)$ if for any compactly supported continuous function $f: G(\mathbb{H}^{2}) \rightarrow \mathbb{R}$, one has 
\[ \lim_{n \rightarrow \infty} \int f d\mu_{n} = \int f d\mu.\]

\paragraph*{\bf Oriented currents} There is a natural notion of an \textbf{oriented geodesic current}-- a $\pi_{1}(\Sigma)$-invariant locally finite Radon measure on $\tilde{G}(\h^{2})$, or alternatively, a transverse invariant measure to the geodesic foliation on $T^{1}\Sigma$.

\section{Geodesic Currents with corners} \label{currents with corners}

\paragraph*{\bf The space of segments} To define the space of currents with corners, we will first need to define-- and then topologize-- the space of \textit{geodesic segments} of $\mathbb{H}^{2}$. Because we will want the space of currents to live as a subspace of the space of currents with corners, it follows that we should allow geodesic segments to be bi-infinite. Identifying $\mathbb{H}^{2}$ with the open unit disk $\mathbb D \subset \mathbb{C}$ via the Poincare disk model, an oriented unparameterized geodesic segment is described uniquely by its endpoints $(x,y) \in \mathbb D \times \mathbb D$; we can parameterize semi- and bi-infinite geodesics by choosing pairs in $\overline{\mathbb D} \times \overline{\mathbb D}$. We therefore identify the \textit{space of unoriented geodesic segments} -- denoted $GS(\mathbb{H}^{2})$ with 
\[ (\overline{\mathbb D} \times \overline{\mathbb D} \setminus \Delta )/\mathbb{Z}_{2}, \]
where as in the previous section, the $\mathbb{Z}_{2}$ action reverses the orientation on a given segment. We will let $\widetilde{GS}(\h^{2})$ denote the \textit{space of oriented geodesic segments}. 

\paragraph*{\bf Plane bundles} It will at times be helpful to conceive of $GS(\mathbb{H}^{2})$ concretely as a plane bundle over $G(\mathbb{H}^{2})$, and so we briefly describe a more hands-on way of visualizing $GS(\mathbb{H}^{2})$ in this way. To do this, we will need a way of referring to a given segment along some bi-infinite geodesic $\gamma$; this is easy to do once one has chosen a parameterization of $\gamma$, but we will need a choice of parameterization that varies continuously from geodesic to geodesic. After again identifying $\mathbb{H}^{2}$ with $\mathbb D$ via the Poincare model, each $\gamma \in G(\mathbb{H}^{2})$ is either a diameter of the unit disk or a circular arc perpendicular to $\partial \mathbb D = S^{1} \subset \mathbb{C}$.

\begin{Def} The \textbf{center point} $c(\gamma)$ of $\gamma \in G(\mathbb{H}^{2})$ is the point along $\gamma$ subdividing it into two pieces of equal Euclidean arclength. 
\end{Def}

\paragraph*{\bf Fixing a parameterization} Note that the assignment sending a geodesic to its center point is continuous. Moreover, $c(\gamma)$ induces a preferred choice of parameterization for $\gamma$ once one also chooses one of two orientations: we simply parameterize a given oriented bi-infinite geodesic by hyperbolic arclength, so that time $t=0$ corresponds to $c(\gamma)$. This allows us to topologize the space of \textit{oriented} geodesic segments of $\mathbb{H}^{2}$ as 
\[ \tilde{G}(\mathbb{H}^{2}) \times \overline{\mathbb{R}}^{2}, \]
where $\overline{\mathbb{R}}$ denotes the two-point compactification of $\mathbb{R}$:
\[ \overline{\mathbb{R}} = [-\infty, \infty]. \]

\paragraph*{\bf Topologizing $GS(\mathbb H^2)$} This yields the following topological description for $GS(\mathbb{H}^{2})$:

\[GS(\mathbb{H}^{2}) \cong  \left(\tilde{G}(\mathbb{H}^{2}) \times \left\{ (a,b) : -\infty \leq a < b \leq \infty \right\} \right)/\sim, \]
where we identify $$(\gamma, (a,b)) \sim (\gamma', (c,d))$$ if $\overline{\gamma} = \overline{\gamma'}$ and $(c,d) = (-b,-a)$.  Indeed, $\rho \in GS(\mathbb{H}^{2})$ specifies a pair of unordered points $\left\{ x,y \right\}$ in $\overline{\mathbb D} \times \overline{\mathbb D}$; the unique unoriented bi-infinite geodesic containing both $x$ and $y$ specifies a point in $\gamma \in G(\mathbb{H}^{2})$, and then $c(\gamma)$ gives rise to an identification of $\rho$ with an equivalence class (containing two elements) of parameterized sub-intervals of $[-\infty, \infty]$.

\paragraph*{\bf A $\pi_1$-action} The natural isometric action of $\pi_{1}(\Sigma)$ on $\tilde{\Sigma} \cong \mathbb{H}^{2}$ by deck transformation induces an action of $\pi_{1}(\Sigma)$ on $GS(\mathbb{H}^{2})$; indeed, given $\rho \in \pi_{1}(\Sigma)$, the geodesic segment with endpoints $\left\{x, y \right\}$ is sent to the geodesic segment with endpoints $\left\{ \rho(x), \rho(y) \right\}$. We may thus refer to the space $GS(\tilde{\Sigma})$-- the space of unoriented geodesic segments in the universal cover of $\Sigma$, equipped with the natural $\pi_{1}(\Sigma)$ action described above. By $GS(\Sigma)$, we will mean the space obtained by quotienting $GS(\tilde{\Sigma})$ by the $\pi_{1}(\Sigma)$-action. 

\paragraph*{\bf Compactifying the tangent bundle} Let $T(\tilde{\Sigma})$ denote the tangent bundle to the universal cover of $\Sigma$. Let $T^{\infty}(\tilde{\Sigma})$ denote the space obtained from $T(\tilde{\Sigma})$ by compactifying each tangent space with a circle at infinity.  Let $T^{\infty}_{\Delta}\Sigma$ be the following ``diagonal subspace'' of the Whitney sum $T^{\infty}\Sigma \oplus T^{\infty}\Sigma$:
\[ T^{\infty}_{\Delta}\Sigma = \left\{ (x, v_{1}, v_{2}) : v_{1}, v_{2} \hspace{1 mm} \mbox{point in opposing directions} \right\}.  \] 
By convention, if either $v_{1}= 0$ or $v_{2}= 0$, we declare $v_{1}, v_{2}$ to be pointing in opposite directions. 

\paragraph*{\bf A projection map} There is then a continuous map 
\[ \rho: T^{\infty}_{\Delta}(\tilde{\Sigma})\setminus \textbf{0} \rightarrow GS(\tilde{\Sigma}), \]
where \textbf{0} denotes the $0$-section, defined as follows. A point $p \in T^{\infty}_{\Delta}(\tilde{\Sigma}) \setminus \textbf{0}$ consists of a triple $(x, v_{1}, v_{2})$ where $x \in \tilde{\Sigma}$, $v_{i}$ is a potentially infinite tangent vector, and $v_{1}, v_{2}$ are not both $0$. Then $\rho(p)$ will be the geodesic segment obtained by concatenating two geodesics that both start at $x$: one tangent to $v_{1}$ and with length $||v_{1}||$, and the other tangent to $v_{2}$ with length $||v_{2}||$. Then $\rho$ surjects onto $GS(\tilde{\Sigma})$. Note also that the $\pi_{1}(\Sigma)$ action on $\tilde{\Sigma}$ induces an action on $T^{\infty}_{\Delta}(\tilde{\Sigma})$, and that $\rho$ descends to a continuous surjection 
\[ \overline{\rho}: (T^{\infty}_{\Delta}(\Sigma)=:) T^{\infty}_{\Delta}(\tilde{\Sigma})/\pi_{1}(\Sigma) \rightarrow GS(\Sigma). \]
If we then let $T^{\infty}_{\Delta}(\Sigma)_{K}$ denote the set of all $(x, v_{1}, v_{2})$ with $x \in \Sigma$ and $||v_{1} - v_{2}|| \ge K$, $\overline{\rho}$ restricts to a map onto $GS(\Sigma)_{K}$, the set of geodesic segments of length at least $K$. Since $T^{\infty}_{\Delta}(\Sigma)_{K}$ is compact, it follows that $GS(\Sigma)_{K}$ is as well. 

\paragraph*{\bf Centered geodesic segments} One can also consider the space $\mathbb{P}T^{\infty}\tilde{\Sigma}$, obtained as a quotient of $T^{\infty}\Sigma$ by identifying points $(x,v)$ and $(x,-v)$. The presence of half-infinite geodesic rays in $GS(\h^{2})$ will sometimes unnecessarily complicate the theory, and so it will therefore be helpful to consider the space of \textit{centered} geodesic segments, defined as follows: 

\begin{Def} \label{centered segments} The space of \textbf{centered geodesic segments}, denoted $CS(\h^{2})$, is identified with $\mathbb{P}T^{\infty}\tilde{\Sigma}$. The point $(x,[v])$ corresponds to the geodesic segment centered at $x$ and of length $2||v||$.      
\end{Def}

\subsection*{Currents with corners}

As above, let $\Sigma$ be a closed orientable surface with genus at least $2$. We are now in a position to define currents with corners: 

\begin{Def} A \textbf{geodesic current with corners} on $\Sigma$ is a locally finite Radon $\pi_{1}(\Sigma)$-invariant measure on the space of geodesic segments $GS(\tilde{\Sigma})$. An \textbf{oriented geodesic current with corners} on $\Sigma$ is a locally finite Radon $\pi_{1}$-invariant measure on the space of oriented geodesic segments. A \textbf{centered geodesic current with corners} is a locally finite Radon $\pi_{1}$-invariant measure on $CS(\h^{2})$. 
\end{Def}

\paragraph*{\bf Examples} Let $\mathcal{GCC}(\Sigma)$ denote the set of all geodesic currents with corners on $\Sigma$; let $\mathcal{CGCC}(\Sigma)$ be the set of centered geodesic currents with corners. Before topologizing $\mathcal{GCC}(\Sigma)$, we first discuss some motivating examples.  

\begin{Example} \label{current is a current with corners} A(n oriented) geodesic current is both a(n oriented) current with corners and a centered current with corners, which assigns a measure of $0$ to every geodesic segment that is not bi-infinite. 
\end{Example}

\begin{Example} \label{corners on a curve} Let $\gamma$ be a geodesic in $\tilde{\Sigma}$ whose projection to $\Sigma$ is dense. There is then no clear way to attach a geodesic current to $\gamma$ in the same way one associates a counting measure to a geodesic whose projection is closed on $\Sigma$. However, one could consider a counting measure on the full lift of a short segment of the projection of $\gamma$ to $\Sigma$. The lift will be a $\pi_{1}(\Sigma)$-invariant collection of geodesic segments in $\h^{2}$, and by choosing the segment sufficiently small, one can guarantee that the measure is locally finite. The support of this current with corners will be a set of short segments in the lift of $\gamma$, based at a collection of $\pi_{1}(\Sigma)$ orbit points. By interpreting each segment as centered, one also gets a centered current with corners in this way. 
\end{Example} 

\begin{Example} \label{graph} Let $\Gamma$ be a finite weighted graph, and let $\phi: \Gamma \rightarrow \Sigma$ be an embedding of $\Gamma$ so that edges are sent to geodesic segments. The embedding lifts to a map $\tilde{\phi}: \Gamma \rightarrow \h^{2}$, and the image of this map can be naturally associated to a current with corners (or a centered current with corners) by considering the sum of counting measures on each edge, weighted by the respective edge weight. 
\end{Example}

\paragraph*{\bf Graph currents} By a \textit{graph current}, we will mean any (centered) current with corners arising from a graph on $\Sigma$ as in the previous example.

\begin{Def}\label{def:mass} The \textbf{mass} of a graph current is the sum, taken over all edges $e$ of $\Gamma$ of the weight assigned to $e$. 
\end{Def}

\begin{Non-example} Let $\Gamma$ now be an infinite weighted graph where each edge receives a positive weighting, and so that the set of weights has a positive greatest lower bound. Let $\phi: \Gamma \rightarrow \Sigma$ be an embedding of $\Gamma$ so that edges are sent to geodesic segments and so that the set of lengths of edges has a non-zero greatest lower bound and a finite least upper bound. One can then construct a $\pi_{1}(\Sigma)$-invariant measure on $GS(\tilde{\Sigma})$ as in Example \ref{graph} which will have as its support the set of all lifts of the edges in $\phi(\Gamma)$. However, this measure will not be locally finite. Indeed, since every edge in the image of $\phi$ has length at most some finite $K < \infty$ (and at least some $k >0$, there must be a (non-zero) geodesic segment $\gamma$ on $\Sigma$ which is a Gromov-Hausdorff accumulation point of the edges in $\phi(\Gamma)$ (see Tuzhilin~\cite{Tuzhilin} for the definition and an interesting discussion of the history of the Gromov-Hausdorff distance). It follows that any neighborhood of a lift of $\gamma$ to $\tilde{\Sigma}$ will be assigned infinite mass. 
\end{Non-example} 

\begin{Example} Let $\Gamma$ be a countably infinite star graph with one central vertex $v$ and edges of the form $(v, v_{i})$ for some countably infinite set of vertices $\left\{v_{1}, v_{2},... \right\}$. Fix some point $x \in \Sigma$, and choose an enumeration of the elements in $\pi_{1}(\Sigma, x)$. Let $\phi: \Gamma \rightarrow \Sigma$ be a map sending $v$ to $x$ and sending the edge $(v, v_{i})$ to the shortest loop in the $i^{th}$ class of $\pi_{1}(\Sigma, x)$. Consider then the measure on $GS(\tilde{\Sigma})$ which, given some $U \subset GS(\tilde{\Sigma})$, counts the number of lifts in $U$ of any edge in $\phi(\Gamma)$. Proper discontinuity of the $\pi_{1}(\Sigma)$ action implies that any such $U$ with compact closure receives finite mass. Note that the key difference between this and the Non-example above is that there is no finite least upper bound for the lengths of edges in $\Gamma$. 
\end{Example}

\subsection{The tangent bundle viewpoint}

As described in Section \ref{prelim}, a classical geodesic current can be interpreted as a transverse invariant measure to the geodesic foliation on $\mathbb{P}T\Sigma$. As we describe below, one can give an analogous formulation of geodesic currents with corners by replacing $\mathbb{P}T\Sigma$ with $T^{\infty}_{\Delta}\Sigma$. 

\paragraph*{\bf Positive subsets} Given $p= (x,v_{1}, v_{2}) \in T^{\infty}_{\Delta}\Sigma$, let the \textbf{length} of $p$ be the (potentially infinite) pair of magnitudes $\left( ||v_{1}||,  ||v_{2}|| \right)$. Given $A \subset T^{\infty}_{\Delta}\Sigma$, let $A_{>0} \subset A$ denote the subset of $A$ consisting of points with non-zero length. In the event that $A = A_{>0}$ we say that $A$ is \textbf{positive}.


\paragraph*{\bf Forgetful map} There is then a projection $\Pi$ from $A_{>0}$ to the projective tangent bundle $\mathbb{P}T\Sigma$ given by forgetting the length-- this is well defined because $v_{1}, v_{2}$ lie on the same line through $x$. We will also refer to a point $p \in T^{\infty}_{\Delta}\Sigma$ as being positive if its length vector is positive.

\paragraph*{\bf Field of geodesic segments} Then $T^{\infty}\Sigma \setminus \textbf{0}$ (where \textbf{0} denotes the $0$-section of the Whitney sum $T^{\infty}\Sigma \oplus T^{\infty}\Sigma$ ) comes equipped with a \textit{field of geodesic segments} $\mathcal{F}'$, defined to be the collection of maximal paths in $T^{\infty}_{\Delta}\Sigma$ with positive interior and whose projection to $\mathbb{P}T\Sigma$ lies in a leaf of the geodesic foliation $\mathcal{F}$. We will call each such path a \textit{leaf} of $\mathcal{F}'$.

\paragraph*{\bf Positive transversals} A codimension-1 submanifold $M$ of $T^{\infty}_{\Delta}\Sigma$ is said to be \textbf{transverse} to $\mathcal{F}'$ if $\Pi(M_{>0})$ is transverse to the geodesic foliation $\mathcal{F}$. The manifold $M$ is \textbf{positively transverse} to $\mathcal{F}'$ if $M$ is a  positive subset of $T^{\infty}_{\Delta} \Sigma$ and transverse to $\mathcal{F}'$ in the usual sense of the word, viewing $M$ as a submanifold of $T^{\infty}_{\Delta}(\Sigma) \setminus \textbf{0}$ and $\mathcal{F}'$ as a foliation on $T^{\infty}_{\Delta}(\Sigma) \setminus \textbf{0}$.

\begin{Def} \label{tangent bundle version} (compare with Definition \ref{bundle def}) A \textbf{geodesic current with corners} is a family of Radon measures 
\[ \left\{ \mu_{\tau}: \tau \subset T^{\infty}_{\Delta}\Sigma \hspace{1 mm} \mbox{is a codimension-1 submanifold transverse to } \mathcal{F}' \right\}, \]
satisfying the following invariance property: if $x,y$ are positive points on two submanifolds $\tau_{x}, \tau_{y}$ such that $\Pi(x),\Pi(y)$ are contained on the same leaf of $\mathcal{F}'$, and if $\psi: U_{x} \rightarrow U_{y}$ is a holonomy diffeomorphism from a positive neighborhood of $x$ to a positive neighborhood of $y$ defined by following the leaves of $\mathcal{F}'$ \textit{such that the trajectory of $U_{x}$ remains positively transverse to $\mathcal{F}'$}, then $\psi_{\ast}\tau_{x} = \tau_{y}$. 
\end{Def}


\paragraph*{\bf Classical currents} We can see that classical geodesic currents are currents with corners from this perspective as well. Note that the $\infty$-section of $T^{\infty}_{\Delta}\Sigma$ is naturally homeomorphic to $\mathbb{P}T\Sigma$. Thus if $\tau$ is a codimension-$1$ submanifold of $T^{\infty}_{\Delta}\Sigma$ transverse to $\mathcal{F}'$ and its intersection with the $\infty$-section is non-empty, it can be identified with a codimension-$1$ submanifold of $\mathbb{P}T\Sigma$ transverse to $\mathcal{F}$. In this way, any classical geodesic current is a geodesic current with corners.

\paragraph*{\bf Reconciling the perspectives} To see that these two definitions of currents with corners are indeed the same, consider the map 
\[ P: T_{\Delta}^{\infty}\tilde{\Sigma} \setminus \textbf{0}  \rightarrow GS(\tilde{\Sigma}) \times G(\h^{2}),\]
 sending $(x, v_{1}, v_{2})$ to $\rho(x, v_{1}, v_{2})$ in the first coordinate, and to the unique bi-infinite geodesic orthogonal to this segment and intersecting it at $x$ in the second coordinate.  Then $P$ is a homeomorphism onto its image, and the map from $T^{\infty}_{\Delta}\tilde{\Sigma}$ to $GS(\tilde{\Sigma})$ obtained by composing $P$ with the projection onto the first factor is a submersion whose fibers are-- after passing to closures-- precisely the leaves of $\mathcal{F}'$. It follows that transverse invariant Radon measures to (the lift of) $\mathcal{F}'$ that are invariant under $\pi_{1}(\Sigma)$ are in one-to-one correspondence with $\pi_{1}(\Sigma)$-invariant Radon measures on $GS(\tilde{\Sigma})$. 

\paragraph*{\bf Compactness issues} One key difference between this formulation and the analogous one in the setting of geodesic currents is that $T^{\infty}_{\Delta}\Sigma \setminus \textbf{0}$ is not compact while $\mathbb{P}T\Sigma$ is. This stems from the fact that points are not segments, and so we have to remove the $0$-section from $T^{\infty}_{\Delta}\Sigma$ to get a well-defined foliation $\mathcal{F}'$ everywhere. This issue manifests in Section \ref{intersection}, particularly in the second half of the proof of Lemma \ref{intersection function continuous}. This also parallels the fact that $GS(\Sigma)_{K}$ is compact for any $K >0$, but $GS(\Sigma)$ is not.

\paragraph*{\bf Centered currents} Finally, we note that the formulation for \textit{centered} currents with corners given above already utilizes a tangent bundle viewpoint because the space of centered geodesic segments is identified with $\mathbb{P}T^{\infty}\tilde{\Sigma}$.

\subsection{The topology on the space of currents with corners}

We use Bonahon's~\cite{Bonahon} original idea for constructing a topology on the space of currents, and we follow the exposition of Aramayona-Leininger~\cite{AL}. 

\paragraph*{\bf Segmented flow boxes} We define a \textbf{segmented flow box} to be the set of geodesic segments transverse to a given close-by pair of small geodesic arcs and with lengths varying in some open subset of $\mathbb{R}$. Formally, given $\epsilon_{1}, \epsilon_{2}>0$ and $x_{1}, x_{2} \in \overline{\mathbb{H}^{2}}$ distinct, a segmented flow box $B= B_{((x_{i}), \epsilon_{i})} \subset GS(\tilde{\Sigma})$ is a subset of $GS(\tilde{\Sigma})$ of the form 

\[ B = \left\{ \gamma \in GS(\tilde{\Sigma}) : \gamma \hspace{1 mm} \mbox{has one endpoint in} \hspace{1 mm} D_{\epsilon_{1}}(x_{1}) \hspace{1 mm} \mbox{and another in} \hspace{1 mm} D_{\epsilon_{2}}(x_{2}) \right\}, \]
where $D_{\epsilon_{1}}(x_{1})$ represents the open disk of \textit{Euclidean} radius $\epsilon_{1}$ about $x_{1}$ (note that this makes sense even in the event that $x_{1}$ sits in the boundary of $\mathbb{H}^{2}$) and where we assume that these disks have disjoint closures. The \textit{height} of such a segmented flow box $B$ is defined to be the (potentially infinite) hyperbolic distance between $D_{\epsilon_{1}}(x_{1})$ and $D_{\epsilon_{2}}(x_{2})$. 

\paragraph*{\bf Admissible flow boxes} A segmented flow box $B$ is said to be $\mu$-\textit{admissible} if $\mu$ assigns measure $0$ to $\partial(B)$, where 
\[\partial(B) = \left\{ \gamma \in B: \gamma \hspace{1 mm} \mbox{has at least one endpoint on either} \hspace{1 mm} \partial D_{\epsilon_{1}}(x_{1}) \hspace{1 mm} \mbox{or on} \hspace{1 mm} \partial D_{\epsilon_{2}}(x_{2}) \right\}. \] 

\paragraph*{\bf Width and height variation} Note that in the event that both $x_{1}, x_{2}$ are interior points of $\mathbb{H}^{2}$, $\partial(B)$ is topologically a $2$-dimensional torus; if exactly one point is in the interior, it's a cylinder, and otherwise it's topologically a disk. The \textbf{width} of a segmented flow box is the maximum of $\epsilon_{1}$ and $\epsilon_{2}$, and the \textbf{height variation} of a segmented flow box is the (potentially infinite) diameter of the infimal open subset in $\mathbb{R}$ in which the lengths vary.

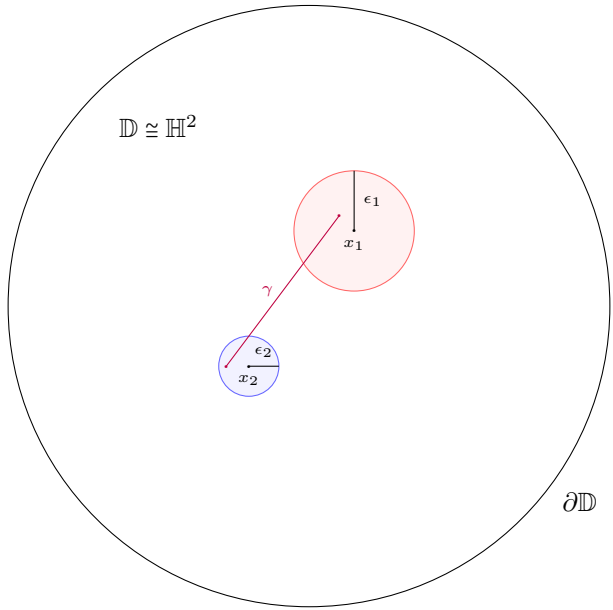
\begin{figure}\caption{An admissible flow box $B = B_{((x_{i}), \epsilon_{i})},$ with $x_i \in \mathbb H^2$, consists of the collection of \emph{hyperbolic} geodesics $\gamma$ connecting the $\epsilon_i$ neighborhoods of $x_i$, viewed in the disk model $\mathbb D$.  \medskip}\label{fig:flowbox}

\begin{tikzpicture}[scale=2.0]

\filldraw[color=black, fill=white] (0,0) circle (2);

\node at (-1, 1.2){$\mathbb D \cong \mathbb H^2$};

\node at (1.8, -1.3){$\partial \mathbb D$};

\filldraw[color=red!60, fill=red!5](0.3, 0.5) circle (0.4);

\filldraw[color=blue!60, fill=blue!5] (-0.4, -0.4) circle (0.2);

\draw (-0.4, -0.4)--(-0.2, -0.4);

\draw (0.3, 0.5)--(0.3, 0.9);

\node at (0.3, 0.5){.};
\node at (-0.4, -0.4){.};

\tkzLabelPoint[below](0.3, 0.5){\tiny $x_1$}
\tkzLabelPoint[left, below](-0.4, -0.4){\tiny $x_2$}

\tkzLabelPoint[above](-0.3, -0.4){\tiny $\epsilon_2$}
\tkzLabelPoint[right](0.3, 0.7){\tiny $\epsilon_1$}

\draw[purple](-0.55, -0.4)--(0.2, 0.6);

\node at (-0.175, 0.1)[purple, left]{\tiny $\gamma$};

\node[purple] at (-0.55, -0.4){.};
\node[purple] at (0.2, 0.6){.};

\end{tikzpicture}

\end{figure}


\paragraph*{\bf A neighborhood basis} Then the collection of sets of the form 

\[ U(\mu, \epsilon) = \left\{ \nu \in \mathcal{GCC}(\Sigma): \forall B \in \mathcal{B}, |\mu(B) - \nu(B)| < \epsilon \right\} \]
where $\epsilon$ ranges over all positive real numbers and $\mathcal{B}$ ranges over all finite collections of $\mu$-admissible segmented flow boxes, is a basis of neighborhoods for $\mu$. The importance of $\mu$-admissibility relates to the Portmanteau theorem for weak-$^{\ast}$ convergence (Proposition \ref{Portmanteau}): a sequence of currents $(\mu_{n})$ converging to $\mu$ satisfies
\begin{equation} \label{point of admissibility}
\lim_{n \rightarrow \infty} \mu_{n}(U_{\mu, \epsilon}) =  \mu(U_{\mu, \epsilon})
\end{equation}
when $U_{\mu, \epsilon}$ is $\mu$-admissible, in the event that 
\[ \sup_{n} \mu_{n}(U_{\mu, \epsilon}) < \infty. \]

\begin{remark} \label{weak star} Another way to topologize the space of currents with corners is via the weak$^{\ast}$ topology: a sequence $(\nu_{n})$ of currents converges to a current $\nu$ if and only if 
\[ \lim_{n \rightarrow \infty} \int f d\nu_{n} = \int f d\nu\]
for every compactly supported continuous function $f$ on the space of geodesic segments $GS(\mathbb{H}^{2})$. The following lemma establishes that this is equivalent to the topology generated by sets of the form $U(\mu, \epsilon)$ as described above:


\begin{lemma} \label{equivalence of topologies} Let $(\nu_{n})_{n=1}^{\infty}$ be a sequence of currents with corners. Then 
\[ \lim_{n \rightarrow \infty} \nu_{n}(A) = \nu(A)\]
for every Borel set $A$ with $\nu(\partial A) = 0$ if and only if $(\nu_{n})$ weak$^{\ast}$ converges to $\nu$. 
\end{lemma}

\begin{proof} Let $f$ be a compactly supported continuous function and suppose that for every finite collection of $\nu$-admissible segmented flow boxes $U_{1},..., U_{k}$ and every $\epsilon>0$, there is $N$ so that for all $n>N$, $|\nu_{n}(U_{j}) - \nu(U_{j})| < \epsilon$. Assume first that $\mbox{support}(f)$ is $\nu$-admissible. Then we have that
\begin{equation} \label{volumes match}
 \lim_{n \rightarrow \infty} \nu_{n}(\mbox{support}(f)) = \nu(\mbox{support}(f)). 
 \end{equation}
Note that each $\nu_{n}$, as well as $\nu$, assigns finite measure to $\mbox{support}(f)$ because it's compact and because each measure is locally finite. Then let $\nu'_{n}$ (respectively $\nu'$) denote the measure obtained from $\nu_{n}$ (respectively $\nu$) by dividing by the $\nu_{n}$-volume (resp. the $\nu$-volume) of $\mbox{support}(f)$. Then by Proposition \ref{Portmanteau}, we have that 
\[ \lim_{n \rightarrow \infty} \int f d\nu'_{n} = \int f d\nu'.\]
Moreover, the left hand side can be rewritten as 
\[ \lim_{n \rightarrow \infty} \frac{1}{\nu_{n}(\mbox{support}(f))} \cdot \int f d \nu_{n}, \]
and the right hand side as 
\[ \frac{1}{\nu(\mbox{support}(f))} \cdot \int f d\nu.\]
So, the desired result follows from \ref{volumes match}. Now, in the event that $\mbox{support}(f)$ is not $\nu$-admissible, we note that it can be approximated from within by $\nu$-admissible sets. Indeed, by compactness, we can assume that $\mbox{support}(f)$ consists of (the closure of) a single segmented flow box $B$. Let $B'$ be obtained from $B$ by slightly decreasing the value of $\epsilon_{1}$, the radius of the first disk about one of the two defining points. Note that if $B'$ and $B''$ are two segmented flow boxes obtained from $B$ in this fashion, but associated to decreasing $\epsilon_{1}$ by different amounts, their topological boundaries are disjoint. Then Since $\nu(B)$ is finite, there can be at most countably many $B'$ obtained in this way which are not $\nu$-admissible, and thus there exists an exhaustion of $B$ by segmented flow boxes obtained by this process which are all $\nu$-admissible.

\begin{remark} \label{same from the outside} We note that the above argument applies just as well to approximating $B$ from above by a sequence of nested segmented flow boxes that are all $\nu$-admissible. For this reason, it follows that the $\nu_{n}$-mass of $B$ is uniformly bounded. 
\end{remark}

\noindent Using the above remark, we can then apply the above argument to each such segmented flow box to obtain that 
\[ \lim_{n \rightarrow \infty} \int_{B_{k}} f d \nu_{n} = \int_{B_{k}} f d \nu.\]
Then let $f_{k}$ be the function that equals $f$ on $B_{k}$ and which is supported on $B_{k}$. Since we can assume that $f$ vanishes on $\partial B$ (indeed, $B$ represents the support of $f$), $\{f_{k}\}_{k \geq 1}$ converges uniformly to $f$ on the closure $\overline{B}$ of $B$. Thus, 
\begin{align*} \lim_{n \rightarrow \infty}\int_{B} f d\nu_{n} &=  \lim_{n \rightarrow \infty} \lim_{k \rightarrow \infty} \int_{B_{k}} f d \nu_{n} \\
& = \lim_{n \rightarrow \infty} \lim_{k \rightarrow \infty} \int_{B} f_{k} d\nu_{n} \\
&= \lim_{k \rightarrow \infty} \lim_{n \rightarrow \infty} \int_{B} f_{k} d\nu_{n} \\
&= \lim_{k \rightarrow \infty} \int_{B} f_{k} d\nu \\ &= \int_{B} f d\nu, \end{align*}
where the last equality follows by Lebesgue dominated convergence (or alternatively, we can equate the left hand side with $\lim_{k \rightarrow \infty}\int_{B_{k}} f d\nu$ and apply the inner regularity property for Radon measures). To interchange the two limits in the above calculation, we are using Remark \ref{same from the outside} and uniform convergence of $(f_{k})$ to argue that if we write \[ a_{n,k} = \int f_{k} d\nu_{n}, \]
then there are constants $c_n$ so that \[\lim_{k \rightarrow \infty} a_{n,k} = c_n,\] where this convergence occurs \emph{uniformly} in $n$; from here, we apply the Moore-Osgood theorem, see, for example, Taylor~\cite{Taylor}*{pp. 139-40}.

\noindent\textbf{The converse.} In the other direction, assume that $$\lim_{n \rightarrow \infty} \int f d\nu_{n} = \int f d\nu$$ for all compactly supported continuous functions and let $B$ be a $\nu$-admissible segmented flow box. We need to show that given $\epsilon >0$, there is $N$ so that for all $n>N$, $$|\nu_{n}(B) - \nu(B)| < \epsilon.$$ To do this, approximate $B$ from above by a nested sequence of segmented flow boxes $B_{1} \supset B_{2} \supset...$ and define $f_{i}$ to be the compactly supported continuous function that is identically $1$ on $B_{i}$, which tapers down to $0$ within $B_{i+1}$, and which is identically $0$ outside of $B_{i+1}$. Then 
\[ \lim_{k \rightarrow \infty} \nu_{n} (B_{k}) = \nu_{n}(B),\]
by the Radon property. Moreover, $$\int f_{k} d\nu_{n} \le \nu_{n}(B_{k}) \le \int f_{k-1} d\nu_{n}.$$  Thus, the squeeze theorem gives us that 
\[ \lim_{n \rightarrow \infty} \nu_{n}(B_{k}) = \nu(B_{k}).\]
When $k$ is large, $\nu_{n}(B_{k})$ is very close to $\nu_{n}(B)$ and the right hand side is very close to $\nu(B)$ (this is again using the Radon property for both $\nu_{n}$ and $\nu$).

\end{proof}

\end{remark}

\paragraph*{\bf Entourages} Given a finite collection $\mathcal{C}= \left\{f_{1},..., f_{n} \right\}$ of continuous and compactly supported real-valued functions $f$ on $\mathcal{GCC}$, let $U_{\mathcal{C}, \epsilon} \subset \mathcal{GCC} \times \mathcal{GCC}$ be the collection 
\[ \left\{(\mu_{1},\mu_{2}): \left|\int f_{i} d\mu_{1} - \int f_{i} d\mu_{2}\right| <\epsilon, i=1,..., n \right\}.  \]
Then the set of all $U_{\mathcal{C}, \epsilon}$ as $\mathcal{C}$ ranges over all finite collections of segmented flow boxes and $\epsilon >0$ gives a \emph{basis of entourages} for a uniform structure on $\mathcal{GCC}$. In fact, for the same functional analytic reasons why the analogous statement holds for $\mathcal{GC}$ we have: 


\begin{Prop} \label{complete} The space $\mathcal{GCC}(\Sigma)$ is complete as a uniform space. 
\end{Prop}

\paragraph*{\bf A topological embedding} To construct an analogous basis for a topology on $\mathcal{CGCC}(\Sigma)$, we proceed as follows: given $U \subset \overline{\mathbb{R}}$ open and $\epsilon>0$, consider \textit{centered} the segmented flow boxes $B_{x, U, \epsilon}$ defined by 
\[ B_{x, U, \epsilon} = \left\{ g \in CS(\tilde{\Sigma}): \mbox{the center of }g \hspace{1 mm} \mbox{lies within }\epsilon \hspace{1 mm} \mbox{of } x, \ell(g) \in U \right\}.  \]

\noindent As mentioned in Example \ref{current is a current with corners}, any geodesic current is naturally a current with corners. One therefore has the inclusion $\mathcal{GC}(\Sigma) \hookrightarrow \mathcal{GCC}(\Sigma)$; the following proposition states that this is a topological embedding: 

\begin{Prop} \label{currents embed} The inclusion map $\mathcal{GC}(\Sigma) \hookrightarrow \mathcal{GCC}(\Sigma)$ is a topological embedding, and the image of $\mathcal{GC}(\Sigma)$ is a closed subset of the codomain. The same holds for $\mathcal{GC}(\Sigma) \hookrightarrow \mathcal{CGCC}(\Sigma)$. 
\end{Prop}

\begin{proof} The fact that the maps are topological embeddings follows immediately from the fact that the intersection of a (centered) segmented flow box with $G(\tilde{\Sigma})$ is a classical flow box. To show that $\mathcal{GC}(\Sigma)$ is closed in the space of currents with corners, consider some $\mu \in \mathcal{GCC}(\Sigma)$ which is not a geodesic current. It follows that there is some non-bi-infinite geodesic $\gamma \in GS(\tilde{\Sigma})$ which is in the support of $\mu$. Then $\mu$ will assign positive measure to any segmented flow box containing $\gamma$; on the other hand any geodesic current must assign $0$ weight to such a segmented flow box. 
\end{proof}

\paragraph*{\bf Convergence examples} We next explore some examples to help the reader get a feel for convergence in $\mathcal{GCC}(\Sigma)$.


\begin{Example} \label{Dehn twist} Fix $c \in \mathcal{GC}(\Sigma)$ a simple closed geodesic, and let $(\phi_{n}: \Gamma_{n} \rightarrow \Sigma)_{n=0}^{\infty}$ be a sequence of $\pi_{1}$-injective maps of a fixed underlying connected graph $\Gamma$ such that for each $n$ the weighted graph $\Gamma_n$ has underlying graph $\Gamma$ and each edge has weight $\frac{1}{n}$, and the maps $\phi_n$ into $\Sigma$ send edges to geodesics so that: 
\begin{enumerate}
\item there exists a single edge $e$ of $\Gamma$ so that $\phi_{0}(e)$ intersects $c$ transversely exactly once;
\item $\phi_{n} = T_{c} \circ \phi_{n-1}$ up to homotopy, where $T_{c}$ denotes the right Dehn twist about the curve $c$. Assume that $\phi_{n}(e') = \phi_{0}(e')$ for all $n$ and for any $e' \neq e$.
\end{enumerate}
Then if we let $\mu_{n} \in \mathcal{GCC}(\Sigma)$ denote the graph current associated to $\phi_{n}(\Gamma)$, we have 
\[ \lim_{n \rightarrow \infty} \frac{1}{n} \cdot \mu_{n} =  c. \]    

\paragraph*{\it Essential intersections} Indeed, the assumptions imply that $\phi_{0}(\Gamma)$ intersects $c$ essentially, and so therefore $\phi_{n}(\Gamma)$ intersects $c$ essentially and $\phi_{n}(e)$ winds around $c$ roughly $n$ times. Let $B$ be a segmented flow box containing a single bi-infinite lift $\rho$ of $c$, and let $B'$ be another segmented flow box containing no lifts of $c$. We want to show: 
\begin{enumerate}
    \item $\lim_{n\rightarrow \infty} \frac{1}{n} \cdot \mu_{n}(B) = 1$;
    \item  $\lim_{n \rightarrow \infty} \frac{1}{n} \cdot \mu_{n}(B') = 0$.
\end{enumerate}

\paragraph*{\it Length ratios} The length of $\phi_{n}(e)$ is roughly $\ell(c) \cdot n$; formally, one has 
\[ \lim_{n \rightarrow \infty} \frac{\ell(\phi_{n}(e))}{n} = \ell(c). \]


\paragraph*{\it Fellow-traveling} There is a lift of $\phi_{n}(e)$ that fellow-travels $\rho$ and remains in a bounded neighborhood of $\rho$ -- denote this lift by $\phi_{n}(e)_{\rho}$.  Abusing notation slightly, let $\rho \in \pi_{1}(\Sigma)$ be the hyperbolic element that translates by $\ell(c)$ along $\rho$. Then when $n$ is very large, it follows that 
\[ \rho^{k}(\phi_{n}(e)) \in B, \]
where $k$ ranges in a set of consecutive integers with size roughly $n$. (See figure) Since $\rho$ is the only lift of $c$ contained in $B$, we can assume that every other element in $\pi_{1}(\Sigma)$ moves $\phi_{n}(e)$ out of $B$. Therefore, the $\mu_{n}$-mass of $B$ is roughly $n/\ell(c)$, and so the $\mu_{n}$-mass of $B$ is roughly $1$, as desired. This verifies $(1)$. 

\paragraph*{\it Lifts of edges} For $(2)$, since $B'$ contains no lifts of $c$, the only potential source for $\mu_{n}$-mass comes from lifts of other edges $e' \neq e$. Thus the $\mu_{n}$-mass of $B'$ is bounded above by 
\[  w(\Gamma) \cdot |E(\Gamma)| \cdot \max_{e' \neq e} \#(\mbox{lifts of } e' \in B' ), \]
where $w(\Gamma)$ is the $\ell^{1}$-norm of the weight vector for $\Gamma$. Since each $e' \neq e$ is fixed by each $\phi_{n}$, the number of lifts of $e'$ contained in $B'$ is bounded independently of $n$, and therefore dividing this quantity by $n$ yields a quotient that goes to $0$ as $n \rightarrow \infty$, as desired.

\end{Example}

\begin{Example} \label{not converging to a classical current} Let $\gamma$ be a closed curve on $\Sigma$ that is geodesic everywhere except at one point, $p$. There is then a free homotopy from $\gamma$ to the geodesic representative in its homotopy class, and the homotopy can be chosen so that the curve $\gamma_{t}$ at time $t$ is geodesic everywhere except at the image $p_{t}$ of the point $p$, and the geodesic curvature at $p_{t}$ goes to $0$ as $t \rightarrow 1$. 
\paragraph*{\it Curves and currents} For each $t$, the curve $\gamma_{t}$ determines a current with corners $\Gamma_{t}$, defined as follows. Consider the full pre-image $\Lambda_{t}$ of $p_{t}$ under the universal covering in $\h^{2}$. The full pre-image of $\gamma_{t}$ is comprised of a $\pi_{1}(\Sigma)$-invariant collection $C_{t} \subset GS(\tilde{\Sigma})$ of geodesic segments running between pairs of points in $\Lambda_{t}$. Then $\Gamma_{t}$ is the counting measure on $C_{t}$. 

\paragraph*{\it Distinct limits} Let $\gamma_{1}$ denote the geodesic representative in the homotopy class of $\gamma$, and by a slight abuse of notation, let $\gamma_{1}$ also refer to the classical geodesic current associated to this closed geodesic. We observe that $$\lim_{t \rightarrow 1}\Gamma_{t} = \Gamma_1 \neq \gamma_1.$$ That is, it is \textbf{not} equal to $\gamma_{1}$, and is in fact equal to $\Gamma_{1}$. In particular, $\gamma_{1}$ and $\Gamma_{1}$ are very different as points in the space of currents with corners. Indeed, the support of the former consists entirely of bi-infinite geodesics, whereas the support of the latter consists of geodesic segments with finite length. 
\end{Example}

\paragraph*{\bf Convergence to a classical current} In Example \ref{not converging to a classical current}, we observed that a sequence of currents with corners did not converge to a classical current, even though when viewed as closed curves on $\Sigma$, the sequence $\{\gamma_{t}\}$ Gromov-Hausdorff converged to to the closed geodesic $\gamma_{1}$. The purpose of the next example is to demonstrate when a current with corners converges to a classical current. 

\begin{Example} \label{converging to a classical current} As in Example \ref{not converging to a classical current}, let $\gamma$ be a closed geodesic on $\Sigma$ and let $p \in \Sigma$ be a point on $\gamma$; abusing notation slightly, we identify $\gamma$ with an element of $\pi_{1}(\Sigma, p)$. Fix $i \in \mathbb{N}$ and let $P \subset \h^{2}$ denote the full pre-image of $p$. Let $\Gamma_{i}$ denote the current with corners that assigns $1$ to each path-lift (there is one for each $\tilde{p} \in P$) of $\gamma^{i}$. Then 
\[ \lim_{i \rightarrow \infty} \Gamma_{i} = \gamma, \]
as currents with corners. Indeed, letting $\tilde{\gamma} \in G(\h^{2})$ denote a bi-infinite geodesic in the pre-image of $\gamma$, a segmented flow box containing $\gamma$ consists of all sufficiently long geodesic segments of bi-infinite geodesics in a flow box containing $\gamma$. Thus, fixing any such segmented flow box $B$, for sufficiently large $i$, $\Gamma_{i}$ will assign mass to $B$. 

\paragraph*{\it Path lifts} Note that this construction did not rely on $p$ being on $\gamma$; we could have instead chosen $p \in \Sigma$ arbitrarily and considered path lifts of $\rho^{i}$ with $\rho \in \pi_{1}(S,p)$ freely homotopic to $\gamma$. As $i$ gets larger and larger, an infinite cyclic pre-image of $\rho^{i}$ is a quasi-geodesic of higher and higher quality (that is, closer and closer to an actual geodesic) that fellow-travels an axis for $\gamma$. 

\end{Example}


\subsection{Grading the space of currents} \label{kcorners}

In the following section, we will work in $\mathcal{CGCC}(\Sigma)$ and when relevant, we will explain why this choice is suitable. It will be convenient to pick out some manageable subspaces of $\mathcal{CGCC}(\Sigma)$; in practice, we will be concerned with an orbit of some finite graph under a group action by homeomorphisms, so it is natural to focus on smaller subspaces preserved by such an action, and whose topology is perhaps more well-behaved than that of the full space of currents with corners. 

\paragraph*{\bf The space of $k$-currents} With this in mind, fix some $k\in \mathbb{N}$; the \textbf{space of currents with $k$ corners} (or the \textbf{space of $k$-currents}), denoted $\mathcal{CGCC}_{k}(\Sigma)$, is the subspace of $\mathcal{CGCC}$ consisting of those currents expressible as the sum of a classical geodesic current and at most $k$ atoms (up to the action of $\pi_{1}(\Sigma)$) associated to finite length geodesic segments. We next show that $\mathcal{CGCC}_{k}(\Sigma)$ is a closed set: 

\begin{lemma} \label{k-currents are closed} The set of $k$-currents is closed in $\mathcal{CGCC}(\Sigma)$.
\end{lemma}

\begin{proof} Fix some $\mu \in\mathcal{CGCC}(\Sigma)$ in the complement of $\mathcal{CGCC}_{k}(\Sigma)$. Our goal is to exhibit an open neighborhood of $\mu$ separating it from everything in $\mathcal{CGCC}_{k}$. Let $\iota(\mu)$ denote the minimum integer $t$ such that $\mu$ is expressible as a sum of a classical geodesic current and $t$ $\pi_{1}$-inequivalent atoms associated to finite length geodesic segments; in the event that no such natural number exists, we set $\iota(\mu) = \infty$. 

Since $\mu \notin \mathcal{CGCC}_{k}(\Sigma)$, $\iota(\mu)$ must be strictly larger than $k$. Therefore, there must exist $(k+1)$ $\pi_{1}$-inequivalent segmented flow boxes $B_{1},..., B_{k+1}$ which are assigned positive mass by $\mu$ and which contain no bi-infinite geodesics. Set $\epsilon > 0$ equal to the minimum mass $\mu$ assigns to any of $B_{1},..., B_{k+1}$. Then by definition of the topology on $\mathcal{CGCC}$, the set of all currents assigning mass within $\epsilon/2$ of the $\mu$-mass of each $B_{i}$, constitutes an open set. On the other hand, any $\rho \in \mathcal{CGCC}_{k}$ must assign mass $0$ to at least one such $B_{i}$. 

\end{proof}

 \begin{remark} The definition of $k$-currents and the argument in Lemma \ref{k-currents are closed} is simpler because we chose to work with $\mathcal{CGCC}$ instead of $\mathcal{GCC}$. The prototypical example of a $k$-current will be a graph current coming from the piecewise geodesic map of a graph with at most $k$ edges into $\Sigma$. We would like the closure of the set of such currents to consist exactly of the union of those currents with the space of classical geodesic currents. However, in $\mathcal{GCC}(\Sigma)$, it is conceivable that a current with half-infinite geodesic rays in its support arises as the limit of graph currents with at most $k$ finite length edges.     
 \end{remark}


\section{Harmonic and marked currents} \label{harmonic currents} 

\paragraph*{\bf Mapping class group action} As we have seen in Section \ref{currents with corners}, the space $\mathcal{GCC}(\Sigma)$ is defined in a way that is fairly analogous to the classical space of currents, $\mathcal{GC}(\Sigma)$. However, there is one major draw-back to the space of currents with corners: there is no canonical action by the mapping class group. Indeed, any mapping class automatically induces a homeomorphism of the boundary at infinity of the universal cover $\h^{2}$, and therefore acts in a well-defined way on the space of bi-infinite geodesics $G(\h^{2})$. On the other hand, there is no obvious way to define an action of a mapping class on the space of geodesic \textit{segments}. To resolve this issue, we will focus on the subspace of $\mathcal{GCC}(\Sigma)$ consisting of graph currents that arise from maps of graphs into $\Sigma$ that minimize some notion of energy, and which are therefore uniquely determined in a given homotopy class of maps.

\begin{Def}  Let $\Gamma$ be a graph with unit length edges and non-negative weights $w_{e}$ associated to each edge $e$. An embedding $\phi : \Gamma \rightarrow \Sigma$ is called \textit{harmonic} if: 
\begin{enumerate}
\item For each edge $e$ of $\Gamma$, the image $\phi(e)$ of $e$ is a geodesic segment; 
\item The map $\phi_{e} : e \rightarrow \Sigma$ has constant speed equal to the length of the geodesic $\phi(e)$; 
\item At each vertex $v$ of $\Gamma$, one has 
\[ \sum_{e \sim v} w_{e} \phi_{e}'(0) = 0, \]
where $e \sim v$ means that $e$ begins at $v$. 
\end{enumerate}
\end{Def}

\paragraph*{\bf Existence} Kajigaya-Tanaka~\cite{KajigayaTanaka} prove that a harmonic embedding exists in each homotopy class of maps of $\Gamma$ into $\Sigma$; that it is unique so long as the image of the embedding is not contractible or freely homotopic to a geodesic; and that it minimizes the \textit{energy functional} $\mathcal{E}$, defined as 
\[ \mathcal{E}(\phi) = \sum_{e \in \Gamma} w_{e} \int ||\phi'_{e}(t)||^{2} dt. \]

\begin{Def} A \textbf{harmonic current with corners} is any current with corners arising as a graph current where the associated embedding into $\Sigma$ is a harmonic embedding. By convention, any geodesic current is also a harmonic current. 
\end{Def}

\begin{remark} A harmonic embedding that is not a geodesic current comes with an embedding, but a current with corners is not naturally associated to any map of a graph into the surface. However, one can characterize the property of being a harmonic current with corners in terms only of the \textit{image} of a harmonic embedding, and so the above definition is well-defined. Indeed, by a harmonic current with corners, we will mean any graph current associated to a possibly disconnected graph on $\Sigma$ with the following property: at any vertex $v$, let $u_{e_{1}}, ..., u_{e_{k}}$ denote the unit tangent vectors in $T_{v}\Sigma$ determined by the edges $e_{1},..., e_{k}$ incident to $v$. Then one has the balance condition

\[ \sum_{i= 1}^{k} w_{e_{i}} \ell(e_{i})  = 0. \]

\end{remark}

\paragraph*{\bf Harmonic k-currents} As a shorthand we will simply say $\textit{harmonic current}$ when referring to a harmonic current with corners.  Let $\mathcal{HC}$ denote the subset of the space of currents with corners consisting of harmonic currents. We will use $\mathcal{HC}$ to build a space on which $\mathcal{MCG}(\Sigma)$ acts. It will also be useful to consider the intersection $\mathcal{GCC}_{k} \cap \mathcal{HC}$, which we call the space of \textbf{harmonic k-currents} and which we denote by $\mathcal{HC}_{k}$. Note that the notion of a harmonic ($k$-)current applies equally well in the setting of $\mathcal{CGCC}(\Sigma)$.

\subsection{Marked currents} Since a harmonic map of a graph is unique up to homotopy, one might hope that $\mathcal{HC}$ admits a mapping class group action. However, a priori the data of a harmonic current doesn't uniquely specify a given map of a graph to $\Sigma$, so there isn't even a clear way to hone in on a fixed homotopy class. Said differently, there may be multiple maps of different graphs to $\Sigma$ that result in the same harmonic current, and so to develop a well-defined mapping class group action, we are led to the notion of a \textit{marked current}-- one that comes equipped with a choice of graph and a map of that graph to $\Sigma$. 

\begin{Def} A \textbf{simple marked current with corners} is a triple $(\Gamma, \phi: \Gamma \rightarrow \Sigma, \phi(\Gamma))$ of a simple weighted connected graph $\Gamma$, a map $\phi$ of $\Gamma$ into $\Sigma$ sending edges to geodesics, and the centered current with corners associated to the image $\phi(\Gamma)$, as in Example \ref{graph}. 
\end{Def}

\paragraph*{\bf Marking graphs and marking maps} The graph $\Gamma$ will be referred to as the \textbf{marking graph} (or \textbf{domain graph}) of the marked current, and $\phi$ is the \textbf{marking map}. We will denote the space of marked centered currents with corners by $\widetilde{\mathcal{CGCC}}(\Sigma)$; the subspace $\widetilde{\mathcal{CGCC}}_{k}(\Sigma)$ consists of those marked currents whose marking graph has at most $k$ edges.

\subsection{Topology on marked currents with corners}

In this section, we equip $\widetilde{\mathcal{CGCC}}_{k}(\Sigma)$ with a topology. The full space of marked currents with corners admits the filtration 
\[ \widetilde{\mathcal{CGCC}}_{1} \subset \widetilde{\mathcal{CGCC}}_{2} \subset... \]
While in practice we will only ever need to work with $\widetilde{\mathcal{CGCC}}_{k}(\Sigma)$ for some finite $k$, in principle one could equip the full space $\widetilde{\mathcal{CGCC}}$ with the associated direct limit topology. Fixing some $k \in \mathbb{N}$, let $n_{k}$ denote the smallest natural number so that $\binom{n_{k}}{2} \ge k$. Let $\Delta_{n_{k}}$ denote the complete graph on $n_{k}$ vertices and where each edge is identified with a unit length segment, and where we have arbitrarily assigned an orientation to each edge. Let $\mathcal{F}(\Delta_{n_{k}}, \Sigma)$ denote the space of maps from all subgraphs of $\Delta_{n_{k}}$ into $\Sigma$ sending edges to geodesic segments. Given such a map $\phi: \Gamma \rightarrow \Sigma$ of a subgraph $\Gamma$ of $\Delta_{n_{k}}$, let $\mathcal{U}_{\epsilon}(\phi)$ denote the set of maps $\psi$ of $\Gamma$ into $\Sigma$ such that for each edge $e$ of $\Gamma$, the Gromov-Hausdorff distance between $\phi(e)$ and $\psi(e)$ is less than $\epsilon$. Then the collection 
\[ \left\{\mathcal{U}_{\epsilon}(\phi): \phi \in \mathcal{F}(\Delta_{n_{k}}, \Sigma), \epsilon>0 \right\} \]
forms a basis for a (disconnected) topology on $\mathcal{F}(\Delta_{n_{k}}, \Sigma)$. 

\paragraph*{\bf A compactification} Since we will be interested in situations where a sequence of graphs converges to a classical geodesic current, we will want some compactification of $\mathcal{F}(\Delta_{n_{k}}, \Sigma)$ that allows for edges to be sent to infinitely long geodesics. To this end, let $\overline{\mathcal{F}}(\Delta_{n_{k}}, \Sigma)$ (or $\overline{\mathcal{F}}(\Delta_{n_{k}})$ for short) denote the set of maps $\phi$ of subgraphs of $\Delta_{n_{k}}$ to $\Sigma$ such that for each edge $e$, exactly one of the following options holds:
\begin{itemize}
    \item $e$ is sent to a geodesic segment and $\phi$ has speed $\ell(\phi(e))$ at its endpoints.
    \item $\phi$ sends the interior of $e$ to a bi-infinite geodesic $g$ such that if $e'$ is incident to the terminal vertex $v$ of $e$, and $v$ is also the terminal (respectively, initial) vertex of $e'$, then the interior of $e'$ is sent to a bi-infinite geodesic $g'$ which is forward (respectively, backwards) asymptotic to $g$. The same holds when $v$ is replaced with the initial vertex of $e$. 
\end{itemize}


\paragraph*{\bf Extending the compactification} We now describe how to extend the topology on $\mathcal{F}(\Delta_{n_{k}})$ to one on $\overline{\mathcal{F}}(\Delta_{n_{k}})$. Given $\phi \in \overline{\mathcal{F}}(\Delta_{n_{k}})$, fix points $x_{i}$ in the interior of $e_{i}$, where $e_{i}$ ranges over each edge of the domain graph $\Gamma$ of $\phi$, and fix also a finite $L \in \mathbb{R}_{>0}$ at most the length of the image of $\phi(e)$, and some $\epsilon>0$. Then let $\mathcal{V}_{\epsilon,L,(x_{i})}(\phi)$ denote the set of maps $\psi$ of $\Gamma$ such that the geodesic segment of length $L$ centered at $\psi(x_{i})$ and in the direction of $\psi(e_{i})$ is within Gromov-Hausdorff distance $\epsilon$ to the geodesic segment of length $L$ centered at $\phi(x_{i})$ and in the direction of $\phi(e_{i})$, for each $i$. The topology associated to the basis 
\[ \left\{\mathcal{V}_{\epsilon, L, (x_{i})})(\phi): \epsilon, 0< L \le \ell(\phi(e_{i})) , x_{i} \in e_{i} \right\} \]
is equivalent to the topology induced by the basis $\left\{\mathcal{U}_{\epsilon}(\phi) \right\}$ on $\mathcal{F}(\Delta_{n_{k}}, \Sigma)$, and it extends this topology to one defined on all of $\overline{\mathcal{F}}(\Delta_{n_{k}})$. Consider the space obtained from $\overline{\mathcal{F}}(\Delta_{n_{k}})$ by identifying all maps whose image agrees for each edge; abusing notation slightly, refer also to this space as $\overline{\mathcal{F}}(\Delta_{n_{k}})$. A sequence of maps $(\phi_{n})$ of a graph $\Gamma$ converges to a map $\phi$ of $\Gamma$ if for each edge $e$ of $\Gamma$, the following holds: given $\epsilon>0$, $x \in \phi(e)$, and some finite $L \le \ell(\phi(e))$, for all sufficiently large $n$ there is $y_{n} \in \phi_{n}(e)$ so that the Gromov-Hausdorff distance between the subsegments of $\phi(e)$ and $\phi_{n}(e)$ centered at $x$ and $y_{n}$ respectively is at most $\epsilon$.

\paragraph*{\bf Maps from subgraphs} Alternatively, we can view $\mathcal{F}(\Delta_{n_{k}})$ as the closed subspace of the space of continuous maps from any subgraph of $\Delta_{n_{k}}$ into $\h^{2}$ sending edges to geodesic segments, modulo the equivalence that two maps are considered the same when, for each edge, their images coincide up to the action of $\pi_{1}(\Sigma)$. Then $\overline{\mathcal{F}}(\Delta_{n_{k}}, \Sigma)$ is the compactification of this disjoint union of spaces of maps (all with the same target space being $\h^{2}$) arising by adding $\partial_{\infty}\h^{2}$ to the target space. We next consider the product space 
\[ \mathcal{P}_{k} = \mathbb{R}^{n_{k}}_{\ge 0} \times \mathcal{F}(\Delta_{n_{k}}, \Sigma) \times \mathcal{CGCC}_{k}(\Sigma). \]

\noindent Then let $\mathcal{Q}_{k}$ be the ``diagonal'' in $\mathcal{P}_{k}$ consisting of triples in which the third component is the image under the second component with respect to the choice of weights: 

\[ \mathcal{Q}_{k} = \left\{(w, \phi, \mu):  \mu \hspace{1 mm} \mbox{is the centered current associated to } \phi, \hspace{1 mm} \mbox{weighted by } w \right\}  \]
Note that $\mathcal{Q}_{k}$ contains the set of simple marked $k$-currents. Then the \textbf{space of marked $k$-currents}, denoted $\widetilde{\mathcal{CGCC}}_{k}(\Sigma)$, will be defined as the quotient of the closure $\overline{\mathcal{Q}_{k}}$ in \[\mathbb{R}^{n_{k}}_{\ge 0} \times \overline{\mathcal{F}}(\Delta_{n_{k}}, \Sigma) \times \mathcal{CGCC}_{k}(\Sigma)\]
obtained by identifying triples when the third component is the same geodesic current. 

\subsection{Observations on the topology}\label{sec:observations} We record a few observations about this topology: 

\medskip

\paragraph*{\bf Continuity at simple marked $k$-currents} The function sending a marked current $(w, \phi, \phi(\Delta_{n_{k}}))$ to the total (unweighted) hyperbolic length of $\phi$ is continuous at simple marked $k$-currents. Indeed, the set of simple marked currents is open in $\widetilde{\mathcal{CGCC}_{k}}$, so a small open neighborhood about a simple marked current $(w, \phi, \mu)$ in $\widetilde{\mathcal{CGCC}}_{k}$ consists entirely of other simple marked $k$-currents. Thus being nearby to $(w, \phi, \mu)$ requires in particular that the marking map of a given graph current is very nearby to $\phi$ in $\mathcal{F}(\Delta_{n_{k}}, \Sigma)$.  We note that the length function is \textit{not} continuous at a point of $\mathcal{GC}(\Sigma)$, if we define the length of a classical current to mean the intersection against the Liouville current (see Bonahon~\cite{Bonahon}*{Proposition 14}) $\mathcal{L}_{\Sigma}$. This is because-- as seen in Example \ref{converging to a classical current}-- a sequence of graph currents converging to a classical current will have the property that hyperbolic length goes to infinity.  

\paragraph*{\bf Ghost edges} Our choice of topology gives rise to a minor pathology that we will sometimes need to pay attention to: an element $(w, \phi, \mu)$ of $\widetilde{\mathcal{CGCC}}_{k}$ can have what we will call \textit{ghost edges}-- edges whose image under the marking map do not give rise to segments in the support of $\mu$. For example, the marking graph $\Gamma$ may have an edge $e$ which is sent to a non-zero finite length geodesic segment on $\Sigma$ by $\phi$ but so that the associated weight is $0$. One can also consider the possibility of bi-infinite ghost edges-- see the end of Remark \ref{weights to 0} below for more on this. Ghost edges contribute to hyperbolic length, but they have no impact on the projection of the marked current to $\mathcal{CGCC}$. Abusing notation slightly, henceforth by a ``simple marked $k$-current'' we will mean a marked graph current where we allow for the possibility of ghost edges.

\paragraph*{\bf Continuity of weighted length} The function sending a marked $k$-current to the $\ell^{1}$ norm of its weight vector is continuous at any simple marked $k$-current. This is not to be confused with the function sending a marked $k$-current to the total mass of its third component; indeed, this latter function is not continuous, owing to the possibility of edges that shrink to vertices in the limit. This latter function is however bounded over sequences $(\mu_{n})$ converging to some $\mu$, which will allow us to apply Proposition \ref{Portmanteau}. On the other hand, \textit{weighted length} -- the dot product of length and weight vectors-- is continuous along sequences where the total mass is bounded. That is, if $(\mu_{n}) \subset \widetilde{\mathcal{CGCC}}_{k}$ is a sequence of simple marked currents such that the total mass of $\mu_{n}$ is at most $M$ for each $n$ and $\lim_{n \rightarrow \infty} \mu_{n} = \mu$ for $\mu$ some simple marked current, then the limit of weighted lengths is the weighted length of $\mu$. 

\medskip
\noindent Since we will need to refer back to it, we promote the final observation to a remark:

\begin{remark} \label{weights to 0}
If $(x_{n}) = ((w_{n}, \phi_{n}, \mu_{n}))$ is a sequence in $\mathcal{Q}_{k}$ converging to a geodesic current, let $\overline{w}_{n}$ denote the projection of $w_{n}$ to the subspace corresponding to edges whose hyperbolic lengths do not converge to $0$ in the limit. Then we claim that $\lim_{n \rightarrow \infty} \overline{w}_{n} = 0$. Said differently, for each edge in the domain graph, either that edge shrinks to a point in the limit, or the associated sequence of weights converges to $0$. Indeed, let $g \in G(\tilde{\Sigma})$ be in the support of $\mu$. Then there must be some edge $e$ in the domain graph and a sequence of lifts $\tilde{e}_{1}, \tilde{e}_{2},...$ (where $\tilde{e}_{n}$ is a lift of the image of $e_{n}$ under the marking graph for $x_{n}$) so that 
\[ g= \lim_{n \rightarrow \infty} \tilde{e}_{n}. \]
Conversely, any edge in the domain graph which does not shrink to a point in the limit either (in the limit) gives rise to geodesics in the support of $\mu$, or gives rise to an infinite length ghost edge. In the latter case, of course the corresponding weights must go to $0$ so we will assume there are no such ghost edges. 

Now, if the corresponding sequence of weights do not converge to $0$, then $g$ must be an atom of $\mu$. Then the same argument used in Example \ref{converging to a classical current} contradicts local finiteness of $\mu$: the number of lifts of edges in the support of $\mu_{n}$ which land in a given segmented flow box neighborhood of an atom in the support of $\mu$ goes to infinity as $n$ does, and so if the weights do not go to $0$, any neighborhood of the atom has infinite $\mu$-mass.

This line of argument also gives information about the behavior of ghost edges with infinite hyperbolic length: if $(x_{n})= (w_{n}, \phi_{n}, \mu_{n})_{n}$ is a sequence of marked  currents converging to $x= (w, \phi, \mu)$ such that $x$ has an infinite length ghost edge $e$, then 
\[ \lim_{n \rightarrow \infty} w_{n}(e) \ell(\phi_{n}(e)) = 0,\]
where $w_{n}(e)$ denotes the weight assigned to $e$ by $w_{n}$. Indeed, by assumption, there is a $\mu$-admissible segmented flow box neighborhood $B$ of $\phi(e)$ (where, abusing notation slightly, we have identified $\phi(e)$ with one of its bi-infinite lifts to $\h^{2}$) which has null $\mu$-measure. Now, assuming that $e$ is the only edge that maps to its image under $\phi$, the $\mu_{n}$-mass of $B$ is equal to the product of $w_{n}$ and the number of lifts of $\phi_{n}(e)$ lying in $B$ (in the event that $\phi$ sends additional edges to $\phi(e)$, we will just apply the following argument to each edge individually and use the fact that there are only at most $k$ edges of the domain graph). Therefore, it must be that 
\[ \lim_{ n \rightarrow \infty} w_{n}(e) \cdot \#(\mbox{lifts of } \phi_{n}(e) \hspace{1 mm} \mbox{in } B) = 0.\]
Finally, the number of lifts of $\phi_{n}(e)$ lying in $B$ is at least on the order of $\ell(\phi_{n}(e))$; indeed, take some closed geodesic $\gamma$ admitting a lift to $\h^{2}$ lying in $B$. Then if $\rho \in B$ is some lift of $\phi_{n}(e)$, then by the same arguments used in Example \ref{converging to a classical current}, there will be on the order $\ell(\phi_{n}(e))/\ell(\gamma)$ lifts of $\phi_{n}(e)$ lying in $B$. 
\end{remark}

\subsection{Marked harmonic currents} Sitting inside of $\widetilde{\mathcal{CGCC}}_{k}$ is  $\widetilde{\mathcal{HC}}_{k}$, the \textbf{space of marked harmonic $k$-currents}, defined as the closure of the set of marked currents in $\mathcal{Q}_{k}$ for which the third component is harmonic. The next lemma characterizes $\mathcal{H}_{k}$ as a set of harmonic graph currents and geodesic currents:

\begin{lemma} \label{harmonic closed} The space $\widetilde{\mathcal{HC}}_{k}$ consists of points $x=(w, \phi, \mu)$ where either $x$ is a simple marked harmonic $k$-current, or $\mu \in \mathcal{GC}(\Sigma)$.  
\end{lemma}

\begin{proof} First, since the balance equations are a closed condition, $\mathcal{HC}_{k}$ is closed in $\mathcal{CGCC}_{k}$. Indeed, if $\mu \in \mathcal{CGCC}_{k}$ is not harmonic, it in particular can not be a geodesic current, and so it has a corner around which the associated balance equation does not hold. There are then segmented flow boxes-- one for each segment incident to this corner-- such that any $k$-current in the open set in $\mathcal{CGCC}_{k}$ corresponding to perturbing the mass slightly in each flow box still does not satisfy the balance equation at the corner. The lemma now follows from Example \ref{any curve} immediately below which shows that any geodesic current arises in the closure of the points in $\mathcal{Q}_{k}$ whose third component is harmonic. 
\end{proof}

%

\paragraph*{\bf A schematic} We use the following figure to summarize the conclusions reached regarding the structure of $\widetilde{\mathcal{HC}}_{k}$: the space is organized schematically like a star-shaped set with $\mathcal{GC}(\Sigma)$ in the center and with one branch for each subgraph of $\Delta_{n_{k}}$ representing the simple harmonic marked $k$-currents with that marking graph. As one approaches $\mathcal{GC}(\Sigma)$ from within a given branch, the weights of all edges that survive in the limit go to $0$; the total mass of the limiting geodesic current will depend on the specific weights occurring along the sequence and the rate at which they decay to $0$.

\paragraph*{\bf A mapping class group action} Let $\mathcal{N}_{k} \subset \widetilde{\mathcal{CGCC}}_{k}$ denote the set of all $k$-currents $(w, \phi, \mu)$ such that the domain graph has a component $C$ so that the image of the complement of all ghost edges under $\phi$ of $C$ is homotopic to a closed curve. Any $k$-current in the complement of $\mathcal{N}_{k}$ has the property that it is the unique harmonic current specified by its domain graph, weight vector, and marking map. We have therefore honed in on a space with a natural $\mathcal{MCG}$-action:

\begin{lemma} \label{mcg action} The mapping class group $\mathcal{MCG}(\Sigma)$ acts on $\widetilde{\mathcal{HC}}_{k} \setminus \mathcal{N}_{k}$ by composing the marking map with a mapping class, or by its standard action on the space of geodesic currents. 
\end{lemma}

\begin{proof} The action is well-defined because harmonic maps are unique in a given homotopy class of embedding. 
\end{proof}

\noindent We also point out that $\widetilde{\mathcal{CGCC}}_{k}$ (and therefore $\widetilde{\mathcal{HC}}_{k}$) is Hausdorff: 

\begin{lemma} \label{Hausdorff} The space $\widetilde{\mathcal{CGCC}}_{k}$ is Hausdorff. 
\end{lemma}

\begin{proof} The space $\widetilde{\mathcal{CGCC}}_{k}$ is obtained from $\overline{Q_{k}}$ by identifying triples when the third component is the same geodesic current. Let $[x], [y] \in \widetilde{\mathcal{CGCC}}_{k}$ be distinct. Then the third components of $[x], [y]$ are well-defined and without loss of generality correspond either to distinct graph currents or to distinct geodesic currents $\mu_{x}, \mu_{y}$. Indeed, if $\mu_{x}= \mu_{y}$, then necessarily $\mu_{x}$ is a graph current, but then $[x], [y]$ must disagree either in their weight vectors or as maps of graphs into $\Sigma$; in either case, they can be separated by open neighborhoods in $\widetilde{\mathcal{CGCC}}_{k}$. 

Let $U_{x}, U_{y} \subset \mathcal{CGCC}_{k}$ be disjoint open neighborhoods about $\mu_{x}, \mu_{y}$ respectively. Since the map sending a point $x \in \mathcal{Q}_{k}$ to its third component is continuous, so is the induced map from $\widetilde{\mathcal{CGCC}}_{k}$ to $\mathcal{CGCC}_{k}$, and therefore we can pull $U_{x}, U_{y}$ back to disjoint open neighborhoods about $[x],[y]$ respectively. 

\end{proof}

\paragraph*{\bf Convergence to a classical current} The following example shows how a sequence of marked harmonic graph currents can converge to a classical geodesic current: 

\begin{Example} \label{Dehn twist 2} As in Example \ref{Dehn twist}, fix a weighted graph $\Gamma$ and an embedding $\phi_{0}: \Gamma \rightarrow \Sigma$ sending edges to geodesic segments. We then have a marked harmonic current $\mu_{0}$ associated to this data. 

Assume that there is exactly one edge $e$ so that $\phi_{0}(e)$ intersects some simple closed geodesic $c$ on $\Sigma$ exactly once. Let $\mu_{n}$ be the marked harmonic current associated to the same weighted graph but marked with a map $\phi_{n}$ which, up to homotopy, is $T^{n}_{c} \circ \phi_{0}$. Just as in Example \ref{Dehn twist}, we have that 
\[ \lim_{n \rightarrow \infty} \frac{1}{n} \mu_{n} = c/\ell(c).\]

To see this, let $\gamma_{n}$ denote the harmonic current associated to $\mu_{n}$; that is, $\gamma_{n}$ is the object obtained from $\mu_{n}$ after forgetting the marking data. Then the length $\gamma_{n}$ goes to infinity linearly with $n$, and at least one edge of $\gamma_{n}$ fellow-travels $c$ and winds around it roughly $n$ times. On the other hand, unlike in Example \ref{Dehn twist}, it is impossible that only one edge is getting longer; indeed, the balance equations preclude this. However, it is easy to see that if the length of an edge goes to infinity, that edge must fellow-travel $c$ for almost all of its length. From here, the same arguments used in Example \ref{Dehn twist} prove the limiting statement.

\end{Example}

\begin{Example} \label{any curve} Let $\gamma \in \mathcal{GC}(\Sigma)$ be any closed curve and fix $\phi_{0}: \Gamma \rightarrow \Sigma$ such that some edge $e$ of a weighted graph $\Gamma$ is sent by $\phi_{0}$ to a loop homotopic to $\gamma$. Then let $\mu_{n}$ be the marked harmonic current associated to the map $\phi_{n}$ which, up to homotopy, equals $\phi_{0}$ on the complement of $e$ and which sends $e$ to a loop that (up to homotopy) traverses $\gamma$ $n$ times. Then arguing similarly to Example \ref{Dehn twist 2}, one has
\[ \lim_{n \rightarrow \infty} \frac{1}{n} \mu_{n} = \gamma/\ell(\gamma). \]
Multiplying each $\mu_{n}$ by an arbitrary scalar produces an arbitrary multiple of $\gamma$ in the limit.

 Since curve multiples are dense in $\mathcal{GC}(\Sigma)$, we deduce that every point of $\mathcal{GC}(\Sigma)$ is an accumulation point of each component of $\mathcal{Q}_{k} \cap \widetilde{\mathcal{HC}}_{k}$. In particular, $\mathcal{GC}(\Sigma)$ is contained in $\widetilde{\mathcal{HC}}_{k}$, and the choice of topology on $\widetilde{\mathcal{HC}}_{k}$ together with Proposition \ref{currents embed} implies that $\mathcal{GC}(\Sigma) \hookrightarrow \widetilde{\mathcal{HC}}_{k}$ is a topological embedding.
\end{Example}

There is a projection map 
\[ \Pi: \widetilde{\mathcal{CGCC}}_{k} \setminus \mathcal{N}_{k} \rightarrow \widetilde{\mathcal{HC}_{k}} \] sending a marked current $\mu$ with at most $k$-corners to the harmonic representative in its homotopy class. In the event that $\mu$ has ghost edges, the third component of $\Pi(\mu)$ is the harmonic graph current associated to the marking map restricted to those edges of the domain graph that receive positive weights. Unfortunately, this map is not continuous, as the following examples demonstrate. 

\begin{Example} \label{snapping back a univalent edge} 
Suppose that $\Gamma$ is a length one path consisting of one edge and two vertices. Fix some closed geodesic $\gamma$ on $\Sigma$. There is then a sequence of graph currents $(\Gamma_{i})_{i=1}^{\infty}$ so that the support of $\Gamma_{i}$ is a proper geodesic subsegment of $\gamma$, and such that the union of the supports over all $\Gamma_{i}$ is the whole of $\Gamma$. Then the harmonic representative of each $\Gamma_{i}$ is the empty current; on the other hand, $\lim_{i \rightarrow \infty} \Gamma_{i}$ is the graph current with one edge and one vertex such that the loop edge is the geodesic representative for  $\gamma$. 
\end{Example}

\begin{Example} \label{being far from harmonic} Let $G_{1}$ be an abstract graph, and let $\Gamma_{1}$ represent the marked harmonic current associated to $G_{1}$. For concreteness, suppose some vertex $v$ of $G_{1}$ has degree $6$ and that the star of $v$ embeds in $\Sigma$ under the harmonic map. Let $G_{2}$ be the abstract graph obtained from $G_{1}$ by ``cutting the star at $v$ in half'' with respect to some cyclic ordering on the edges incident to $v$-- see Figure~\ref{fig:splitting}. One can then define $\Gamma_{2}$ to be a graph current in such a way that it is very close in the topology of $\mathcal{GCC}_{k}$ to the harmonic current $\Gamma_{1}$. However, the one balance equation at $v$ in $\Gamma_{1}$ now corresponds to two separate balance equations associated to two distinct vertices in $\Gamma_{2}$; the fact that $\Gamma_{2}$ is nearby $\Gamma_{1}$ only guarantees that the sum of both balance equations is very nearly $0$, but each equation may be far from satisfied on its own-- see Figure~\ref{fig:splitting}. 

\end{Example}


\begin{figure}[h!]
    \centering
    \begin{tikzpicture}
        \draw(0,0)--(0:1);
        \draw(0,0)--(60:1);
          \draw(0,0)--(120:1);
            \draw(0,0)--(180:1);
            \draw(0,0)--(240:1);
            \draw(0,0)--(300:1);
            \node at (0,0){$\circ$};
\draw[red, ->](2,0)--(4,0);
    \tikzset{shift={(6,0)}}
\draw(0,0)--(120:1);
            \draw(0,0)--(180:1);
            \draw(0,0)--(240:1);
\node at (0,0){$\circ$};
\draw(0,0)--(2,0);
    \tikzset{shift={(2,0)}}
    \node at (0,0){$\circ$};
\draw(0,0)--(0:1);
        \draw(0,0)--(60:1);
          \draw(0,0)--(300:1);

    \end{tikzpicture}
    \caption{Splitting can change the balance equations.}
    \label{fig:splitting}
\end{figure}
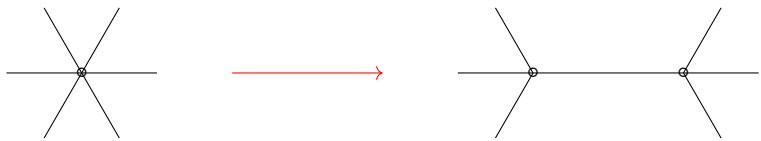

What both Examples \ref{snapping back a univalent edge} and \ref{being far from harmonic} demonstrate is that discontinuities of the projection map $\Pi$ can arise due to the existence of graph currents that are very far from being harmonic. The content of Lemma \ref{continuous projection} below is that $\Pi$ satisfies a weak form of continuity along sequences of $k$-currents that converge to a harmonic current. To state it, we first need to introduce the notion of \textit{harmonic deviation} which quantifies how far away a geodesic graph is from being harmonic: 

\begin{Def} \label{harmonic deviation} Let $\mu = (\phi: G \rightarrow \Sigma, \phi(G))$ be a marked graph current on $\Sigma$ for $G = (V = \left\{v_{1},..., v_{n} \right\},E)$ some abstract graph. Then the \textbf{harmonic deviation} of $\mu$ is the $\ell^{1}$ norm of the vector whose $i^{th}$ component is 
\[ \sum_{e \sim v_{i}} w_{e} \phi'_{e}(v_{i}), \]
where $\phi'_{e}(v_{i})$ denotes the vector tangent to $\phi_{e}$ at the vertex $v_{i}$, pointing away from $v_{i}$. Given $v_{i} \in V$, by the $v_{i}$-\textbf{balance value} (or simply the \textbf{balance value} when $v_{i}$ is implicit), we will mean the value of the above expression. The vector whose components equal the balance values will be referred to as the \textbf{balance vector}. 
\end{Def}
\paragraph*{\bf Almost harmonic currents} Note that the harmonic deviation of $\mu$ vanishes precisely when it is harmonic. Given some $\epsilon \ge 0$, we will say that $\mu$ is $\epsilon$-\textbf{almost harmonic} when its harmonic deviation is at most $\epsilon$. We will let $\widetilde{\mathcal{CGCC}^{\epsilon}}_{k}(\Sigma) \subset \widetilde{\mathcal{CGCC}_{k}}$ denote the set of $\epsilon$-almost harmonic $k$-currents. We observe that the set of $\epsilon$-almost harmonic currents is closed: 

\begin{lemma} \label{almost harmonic currents are closed} The set $\widetilde{\mathcal{CGCC}^{\epsilon}}_{k}(\Sigma)$ is closed in $\widetilde{\mathcal{CGCC}_{k}}$.
\end{lemma}

\begin{proof}  Suppose $\mu \in \widetilde{\mathcal{CGCC}_{k}}$ is not $\epsilon$-almost harmonic. Then one can select segmented flow boxes $B_{1},..., B_{j}$ (for $j \leq k$) so that each $B_{i}$ contains exactly one finite length atom of $\mu$. There is then some $\delta$ sufficiently small so that the set of (marked) $k$-currents which assign a mass of within $\delta$ of the $\mu$-mass of each $B_{i}$ has the property that its restriction to the union of the $B_{i}$'s is also not $\epsilon$-almost harmonic- denote by $V$ the open set of all such $k$-currents. We can then intersect $V$ with a sufficiently small open neighborhood $U$ around $\mu$ such that every $k$-current in $U$ has weighted length extremely close to that of $\mu$, and such that:
\begin{itemize}
\item a $k$-current in $U \cap V$ can not possibly have fewer edges than $\mu$;
\item if a $k$-current $\rho \in U \cap V$ has more edges than $\mu$, the additional edges contribute so negligibly to the weighted length that they can not possibly alter the balance vector enough to push the $\ell^{1}$ norm below $\epsilon$.
\end{itemize}
\end{proof}

\paragraph*{\bf The continuity of $\Pi$} We are now ready to state the key lemma regarding continuity of $\Pi$:

\begin{lemma} \label{continuous projection} Let $(\mu_{i})_{i=1}^{\infty}$ be a sequence of marked $k$-currents converging to some marked harmonic simple $k$-current, $\mu$. Let $\epsilon_{i}$ denote the harmonic deviation of $\mu_{i}$, and suppose further that $\lim_{i \rightarrow \infty} \epsilon_{i} = 0$. Then 
\[ \lim_{i \rightarrow \infty} \Pi(\mu_{i}) = \mu.\]
\end{lemma}

Before formally proving Lemma \ref{continuous projection}, we first give a brief sketch of the idea to aid in exposition: 
\begin{enumerate}
\item For a fixed graph $G$, the space of geodesic maps $\mathcal M(G, \Sigma, L)$ of $G$ into $\Sigma$ with weighted length at most some $L>0$, is compact.
\item We will use the compactness of $\mathcal M(G, \Sigma, L)$ to show, in several steps, that when the harmonic deviation of a marked current is small, it must be nearby (as a map) to the harmonic representative in its homotopy class. 
\item We will repeatedly use the fact that for a fixed positive weight vector $w$ and homotopy class $[\phi]$ that is non-trivial on each component of the domain graph, there is a unique harmonic representative. 
\item Some bookkeeping will arise to deal with the possibility of ghost edges. They need to be treated somewhat separately since uniqueness of harmonic maps can fail when some of the weights are $0$.  
\end{enumerate}

\begin{proof} 

Assume first that the limit $k$-current $\mu$ has no ghost edges. We can also assume the existence of some $M>0$ so that the total weight of each $\mu_{i}$ is at most $M$, and that $\mu_{i}$ has no ghost edges. Moreover, by using Proposition \ref{currents embed}, we can assume that no $\mu_{i}$ is a geodesic current and therefore that the total hyperbolic length of each $\mu_{i}$ is at most some $L>0$. Finally, we can assume that each $\mu_{i}$ has the same marking graph, which is equal to the marking graph of $\mu$ and which we will denote by $\Gamma$.

Fix a non-trivial homotopy class of an embedding $[\phi]: \Gamma \rightarrow \Sigma$, and for $L >0$ let $T(\Sigma, [\phi], L)$ denote the set of maps $\psi \in [\phi]$ in this homotopy class satisfying the following properties:

\begin{enumerate}
\item Each edge is sent to a geodesic segment.
\item Both the total hyperbolic length of the image graph and the $\ell^{1}$-norm of the weight vector is at most $L$.
\end{enumerate}

Then by our assumptions above, there is an $L>0$ so that the marking map of $\mu_{i}$ is in $ T(\Sigma, [\phi], L)$. Note that this relies on the assumption of no ghost edges; indeed, if $\mu$ has an infinite length ghost edge, there will be no uniform upper bound on the hyperbolic length of the image of each marking map. Since $\Sigma$ and $[\phi]$ will be constant during the course of the argument, we will denote $T(\Sigma, [\phi], L)$ by $T(L)$ for short.

We use the Gromov-Hausdorff distance arising from the ambient metric on $\Sigma$ to equip $T(L)$ with a topology: two maps are close when the sum of Gromov-Hausdorff distances, taken over all edges of $\Gamma$, is small, and the distance between weight vectors is small. With respect to this topology, $T(L)$ is compact. 

Now, fix a positive weight vector $w \in \mathbb{R_+}^{n_{k}}$ with $\ell^{1}$-norm at most $L$. We claim that given $\delta>0$, there is an $\epsilon>0$ so that for any map $\phi$ in $T(L)$ with weight vector $w$ that is $\epsilon$-almost harmonic, then for each edge $e$ of $G$, $\phi(e)$ is within Gromov-Hausdorff distance $\delta$ from $h(e)$, where $h$ denotes the harmonic representative in the homotopy class of $\phi$ relative to the weight vector $w$. To prove this, we argue by contradiction: if not, then there is some $\delta$ and a sequence of maps $\{\psi_{i}\}_{i=1}^{\infty} \subset T(L)$ of graphs which remain at a Gromov-Hausdorff distance of at least $\delta$ from the harmonic representative (all with respect to $w$), but so that the image of $\psi_{i}$ is $\epsilon_{i}$-almost harmonic with $\lim_{i \rightarrow \infty} \epsilon_{i} = 0$.

 Let $h: \Delta_{n_{k}} \rightarrow \Sigma$ be the unique harmonic representative in the homotopy class of $\phi$, and (passing to a subsequence if necessary) let $e$ be an edge of the domain graph $\Gamma$ so that $\psi_{i}(e)$ remains at least Gromov-Hausdorff distance $\delta$ from $h(e)$. Since the image of $e$ is determined by $\phi$ and the image of its endpoints $v_{1}, v_{2}$, there must be some $\delta'$ so that (passing once more to a subsequence if necessary) $\psi_{i}(v_{1})$ is at least distance $\delta'$ from $h(v_{1})$. 

Consider the closed subspace $\tau$ of maps in $T(L)$ sending $v_{1}$ outside of the open ball of radius $\delta'$ about $h(v_{1})$. By compactness, there must exist a limiting map $\psi: \Gamma \rightarrow \Sigma$ which is in $\tau$, and furthermore, Lemma \ref{almost harmonic currents are closed} implies that this map is harmonic. This contradicts the uniqueness of harmonic maps in a given homotopy class. Note that this relies on the positivity of $w$: if some weights in $w$ are $0$, there may be multiple harmonic maps in a fixed homotopy class. 

We point out an important consequence of this claim: let $D_{L,M} \subset [0, L]^{n_{k}}$ be the set of positive weight vectors $w$ so that the corresponding harmonic map $H$ lies in $T(M)$. Then the map sending a weight vector $w \in D_{L,M}$ to the corresponding harmonic map in $T(M)$ is continuous. Indeed, if not, there is some sequence $(w_{i})$ of weight vectors converging to $w$, but such that for some vertex $v$ of the domain graph, $H(w_{i})(v)$ stays some definite distance away from $H(w)(v)$. However, we claim that the map $H(w_{i})$ is necessarily $\epsilon_{i}$-almost harmonic with respect to the weight vector $w$, for some sequence $(\epsilon_{i})$ going to $0$ as $i \rightarrow \infty$. If so, this is a contradiction by what we have just shown above. To see the claim, note that $H(w_{i})$ has harmonic deviation $0$ with respect to the $w_{i}$ weights; since the lengths of all edges of $H(w_{i})$ lie in the compact set $[0,M]$, the harmonic deviation of $H(w_{i})$ with respect to the $w$ weights gets arbitrarily small as $i \rightarrow \infty$, as desired.

Now, let 
\[ \kappa_{\delta}: (0,L]^{n_{k}} \rightarrow \mathbb{R} \]
be the map sending a given weight vector $w \in \mathbb{R}^{n_{k}}$ to the supremal $\epsilon$ so that any map in $T(L)$ that is $\epsilon$-almost harmonic with weight vector $w$ has the property that the image of each edge of the domain graph lies within Gromov-Hausdorff distance $\delta$ from the image of that edge under the harmonic map given by $[\phi]$ and $w$. Then given any $\delta'>0$, the restriction of $\kappa_{\delta}$ to $[\delta',L]^{n_{k}}$ has a non-zero infimum, for if not, there is a sequence of maps $(\psi_{i})$ in $T(L)$ so that 
\begin{itemize}
\item $\psi_{i}$ is $\epsilon_{i}$-almost harmonic and $\lim_{i \rightarrow \infty} \epsilon_{i} = 0$;
\item Each component of the weight vector $w_{i}$ is at least $\delta'$;
\item for some vertex $v$, $\psi_{i}(v)$ is distance at least $\delta$ from $h_{i}(v)$, where $h_{i}$ is the harmonic map relative to the weight vector $w_{i}$ for $\psi_{i}$. 
\end{itemize}
Then again by compactness, after passing to a subsequence if necessary there is some limiting map $\psi \in T(L)$ that is harmonic, and each $w_{i}$ (and of course also $w$) lies in $D_{L',L}$ for some fixed $L'$ depending only on $L$ and the number of edges. 

Since the map sending a weight vector in $D_{L',L}$ to the harmonic representative is continuous, it follows that $\psi$ has some definite distance from the corresponding harmonic map specified by its weight vector, which again contradicts uniqueness of harmonic maps. 

We deduce that given $\delta_{1}, \delta_{2} >0$, there is an $\epsilon>0$ so that for any map $\phi$ in $T(\Sigma, L)$ that is $\epsilon$-almost harmonic and such that each component of the weight vector is at least $\delta_{1}$, then for each edge $e$ of the domain graph, $\phi(e)$ is within Gromov-Hausdorff distance $\delta_{2}$ from $h(e)$, where $h$ is the corresponding harmonic map. This concludes the proof in the case that $\mu$ is a marked graph current with no ghost edges, because it implies that $\lim_{i \rightarrow \infty} \Pi(\mu_{i})$ and $\lim_{i \rightarrow \infty} \mu_{i}$ coincide. 

It remains to consider the possibility of ghost edges. Let $\Gamma$ denote the marking graph of $\mu$, and $\phi$ the marking map. Then if $\mu$ has ghost edges, let $\Gamma' \subset \Gamma$ be the subgraph consisting of the complement of all ghost edges, and let $\phi'$ be the restriction of $\phi$ to $\Gamma'$. Then using Remark \ref{weights to 0}, the marked currents $\mu_{i}'= (w'_{i}, \phi'_{i}, c_{i}')$ (where $w'_{i}$ denotes the weight subvector of $w_{i}$ with components corresponding to the edges in $\Gamma'$ and where $c_{i}'$ is the graph current so that $(w_{i}', \phi_{i}',c_{i}') \in \mathcal{Q}_{k}$) still satisfy the main hypothesis of the lemma: as $n \rightarrow \infty$, the harmonic deviations go to $0$. This is obvious when the ghost edges are finite length, and follows from the second part of Remark \ref{weights to 0} when there are infinite length ghost edges. 

It then follows-- exactly by the argument used above under the assumption of no ghost edges since there is now some uniform $L$ so that the hyperbolic length of each $(w_{i}', \phi_{i}', c_{i}')$ is at most $L$-- that 
\[ \lim_{i \rightarrow \infty} \Pi(\mu_{i}') = \mu',\]
where $\mu'$ is the marked harmonic current with domain graph $\Gamma'$ and marking map $\phi'$. 

Let $\mu_{i}''$ be the marked current with domain graph $\Gamma'$ and third component equal to the current with corners whose support consists of the image of all edges in $\Gamma'$ \textit{under the harmonic map of the full graph} $\Gamma$ (in the homotopy class of $\phi_{n}$). In other words, the third component of $\mu_{i}''$ is the graph current obtained by removing the atoms from the harmonic graph current $\Pi(\mu_{i})$ corresponding to edges that get longer and longer as $i$ goes to infinity. 

Then, again by the second part of Remark \ref{weights to 0}, the harmonic deviation of $\mu_{i}''$ goes to $0$ with $i$, and so again by our arguments above, given any $\delta>0$ there is $N$ so that each edge of $\mu_{i}''$ is within Gromov-Hausdorff distance at most $\delta$ from the corresponding edge of its harmonic representative, which is $\Pi(\mu_{i}')$ (for all $i>N$). It follows that 
\[ \lim_{i \rightarrow \infty} \mu''_{i} = \lim_{i \rightarrow \infty} \Pi(\mu_{i}'') = \mu'.\]
Since when $i$ is large, the full mass of $\Pi(\mu_{i})$ is very nearly the mass of $\mu''_{i}$, this implies the desired result, namely that $\lim_{i \rightarrow \infty} \Pi(\mu_{i}) = \mu$.

\end{proof}


\section{Borel action of the mapping class group on the space of (marked) harmonic currents}

\paragraph*{\bf Bounded weights} Given $M >0$, let $\widetilde{\mathcal{HC}}_{k}(\Sigma, M) \subset \widetilde{\mathcal{HC}}_{k}(\Sigma) \setminus \mathcal{N}_{k}$ be the subspace consisting of all geodesic currents and all harmonic graph $k$-currents such that each edge has mass at most $M$. Note that the mapping class group action preserves each $\widetilde{\mathcal{HC}}_{k}(\Sigma, M)$, and in this section, we will prove that induced action on this subspace is Borel:

\begin{Theorem} \label{Borel action} The mapping class group $\mathcal{MCG}(\Sigma)$ acts by Borel transformations on $\widetilde{\mathcal{HC}}_{k}(\Sigma, M)$. 
\end{Theorem}

Since the full space $\widetilde{\mathcal{HC}}$ is topologized with the final topology with respect to the directed system of inclusions $\widetilde{\mathcal{HC}}_{0} \subset \widetilde{\mathcal{HC}}_{1} \subset...$, it follows from the definition of the direct limit topology that the $\mathcal{MCG}$-action on $\widetilde{\mathcal{HC}}(\Sigma, M)$-- the space of harmonic marked currents with any number of edges and where each edge has mass at most $M$ --is Borel as well. This observation won't be necessary for our counting theorems.

\begin{proof}  Let $U \subset \widetilde{\mathcal{HC}}_{k}(\Sigma, M)$ be open; by replacing $f \in \mathcal{MCG}(\Sigma)$ with its inverse, it suffices to prove that $f(U)$ is a Borel subset. Decompose $U$ as $U = (U \cap \mathcal{GC}(\Sigma)) \sqcup (U \cap \mathcal{GC}(\Sigma)^{c})$. Since the space of currents is closed, both of these subsets are Borel. By the definition of the topology on $\widetilde{\mathcal{HC}}_{k}$ and Proposition \ref{currents embed}, the first of these is an open set in $\mathcal{GC}(\Sigma)$, and the second is open in the full space $\widetilde{\mathcal{HC}}_{k}(\Sigma, M)$. Since the mapping class group action preserves the space of currents, one trivially has 
\[ f(U) = f(U \cap \mathcal{GC}(\Sigma)) \sqcup f(U \cap \mathcal{GC}(\Sigma)^{c}).\]
Since the mapping class group acts by homeomorphisms on $\mathcal{GC}(\Sigma)$, it suffices to prove that the restriction of $f$ to the complement of $\mathcal{GC}(\Sigma)$ is continuous. In fact, what we will show in the course of this proof is that each mapping class $f$ acts as a homeomorphism on $\mathcal{GC}(\Sigma)$ and $\mathcal{GC}(\Sigma)^c$, but we can have issues with length functions if we have a sequence of graph currents approximating a geodesic current.

To this end, fix some $\mu \in \widetilde{\mathcal{HC}}_{k}(\Sigma, M)$ \textit{not} a geodesic current; that is, $\mu$ is a marked harmonic current associated to some graph $\Gamma$ and some harmonic map $\phi: \Gamma \rightarrow \Sigma$ so that $\Gamma$ has at most $k$ edges, each of which has mass at most $M$. Given $f \in \mathcal{MCG}(\Sigma)$, let $\psi: \Sigma \rightarrow \Sigma$ be a diffeomorphism representing $f$, and which sends $\mu$ to $f(\mu)$. 


\begin{remark} \label{approximate diffeo} A priori, it is possible that no such $\psi$ exists. For example, the harmonic map in the homotopy class corresponding to $f \circ \phi$ might identify edges of $\phi(\Gamma)$. In this case, we can find a diffeomorphism $\psi$ that approximates the map sending $\mu$ to $f(\mu)$ with arbitrary precision, for example with respect to the topology on $T_{\Gamma}(\Sigma, f \circ \phi, L)$ where $L$ is chosen to be sufficiently larger than the weighted length of $f(\mu)$. 

Another potential issue comes from Reidemeister moves (see, for example~\cite{FarbMargalit}*{\S 5.3}): perhaps the harmonic representative in $[f\circ \phi]$ is obtained by first sliding an edge past a vertex and then applying a homeomorphism. Since our argument will not depend on the original current with corners $\mu$ being harmonic, in the event that the harmonic representative in $[f \circ \phi]$ is related not to (the image of) $\mu$ by a homeomorphism but to some other graph $\mu'$ obtained from $\mu$ by Reidemeister moves, we begin the argument by replacing $\mu$ with (the non-harmonic) marked current with corners $\mu'$. The Reidemeister move can be executed simultaneously for all harmonic currents in a small neighborhood $U$ of $\mu$ to obtain a small neighborhood $U'$ of $\mu'$. 
\end{remark}

Given any marked harmonic current $\mu'$, by $\mu'_{\psi}$ we will mean the geodesic current with corners obtained from $\psi(\mu')$ by straightening the image of each edge to a geodesic segment. We claim that given $\epsilon >0$, there exists an open neighborhood $V$ about $\mu$ so that for any $\mu' \in V$, $\mu'_{\psi}$ is $\epsilon$-almost harmonic. Assuming this claim and an additional assumption that will arise momentarily and that we will justify, we finish the proof: take any sequence of harmonic $k$-currents $\mu^{1}, \mu^{2},...$ converging to $\mu$. The claim implies that the sequence $\mu^{1}_{\psi}, \mu^{2}_{\psi}...$ has the property that the corresponding sequence of harmonic deviations $\epsilon_{1}, \epsilon_{2}...$ converges to $0$. Assume also that 
\[ \lim_{i \rightarrow \infty} \mu^{i}_{\psi} = \mu_{\psi} = f(\mu). \]
Then Lemma~\ref{continuous projection} applies and gives us that 
\[ \lim_{i \rightarrow \infty} \Pi(\mu^{i}_{\psi}) = \lim_{i \rightarrow \infty} f(\mu^{i}) = f(\mu), \]
as desired. 

It therefore suffices to prove:
\begin{itemize}
    \item $\mu^{i}_{\psi}$ is $\epsilon_{i}$-almost harmonic, where $\lim_{i \rightarrow \infty} \epsilon_{i} = 0$;
    \item $\lim_{i \rightarrow \infty} \mu^{i}_{\psi} = \mu_{\psi}$. 
\end{itemize}

Since weighted length varies continuously on $\widetilde{\mathcal{HC}}_{k}$, we can start by choosing a neighborhood $W$ of $\mu$ so that any $\mu' \in W$ has weighted length within some small $\delta$ of $L$, the weighted length of $\mu$. In fact, using Remark \ref{weights to 0}, we can choose $W$ sufficiently small so that for any such $\mu'$, there is a subset of its edges $\mathcal{E}$ accounting for all but $\delta$ of the total weighted length of $\mu'$ having the property that each $e \in \mathcal{E}$ lies in the $\delta$-neighborhood of some edge of $\mu$ (with respect to the hyperbolic metric on $\Sigma$). 

Given $\mu'$, each edge $e$ of its marking graph corresponds to a map $\mu'_{e}: [0,1] \rightarrow \Sigma$ of a unit interval into $\Sigma$ whose image is geodesic. One then gets a new map $\psi \circ \mu'_{e}: [0,1] \rightarrow \Sigma$ by composing $\mu'_{e}$ with $\psi$. 
Using smoothness of $\psi$ and the compactness of the space of smooth maps from the unit interval to $\Sigma$ and with derivative everywhere uniformly bounded above, we can in fact choose $W$ so that the integral of geodesic curvature $k_{g}$ is at most $\delta$: 
\[ \int_{0}^{1} k_{g}(\psi \circ \mu'_{e}(t)) dt < \delta, \forall \mu' \in W. \]
For any vertex $v$ of some $\mu'$, note that the $\psi$-pushforwards of the (appropriately weighted) tangent vectors to edges incident to $v$ are very close to balancing, because they balance for $\psi(\mu)$. However, $\psi(\mu')$ need not even be a geodesic current with corners because $\psi \circ \mu'_{e}$ needn't be geodesic. On the other hand, by selecting $\mu' \in W$, we guarantee that upon straightening each of the $\psi$-images of edges to geodesic segments, the change in angles between edges incident to any vertex is bounded above by a function of $\delta$ that goes to $0$ as $\delta$ does. Therefore, given $\epsilon>0$, by choosing $\delta$ sufficiently small, we guarantee that $\mu'_{\psi}$ is $\epsilon$-almost harmonic. We also guarantee that $\psi(\mu^{i})$ is $\delta_{i}$-Gromov-Hausdorff close to the geodesic segment obtained by straightening each edge, where $\lim_{i \rightarrow \infty} \delta_{i} = 0$. Since, by continuity of $\psi$, the image of an edge  in $\mu^{i}$ under $\psi$ is necessarily Gromov-Hausdorff close to an edge of $\mu$ when $i$ is large, it follows that $\mu^{i}_{\psi}$ is nearby $\mu_{\psi}$, as desired.

\end{proof}


\section{An intersection form for currents with corners} \label{intersection}

\subsection{The intersection form} \label{subsec:intersection}

A key tool in the study of geodesic currents is the bilinear continuous intersection form introduced by Bonahon~\cite{Bonahon}, and which generalizes the notion of geometric intersection number for pairs of curves (or for pairs of measured geodesic laminations). We introduce a sort of intersection form on the space $\mathcal{GCC}(\Sigma)$ by naively following Bonahon's~\cite{Bonahon} construction-- we will see that it enjoys several analogous properties to the classical form on the space of currents, but not all of them. 

\begin{Def} Let $\widetilde{\mathcal{D}} \subset GS(\h^{2}) \times GS(\h^{2})$ denote the set of pairs of geodesic segments which intersect transversely. By convention, we consider a pair of geodesic segments to intersect transversely if they share an endpoint in $\h^{2}$. A geodesic segment is also said to intersect a bi-infinite geodesic $g$ transversely when one of it endpoints lies on $g$. Since the $\pi_{1}(\Sigma)$ action sends transversely intersecting segments to transversely intersecting segments, it preserves $\widetilde{\mathcal{D}}$ and we can then consider the quotient
\begin{equation}\label{eq:intersection}\mathcal{D} = \widetilde{\mathcal{D}}/\pi_{1}(\Sigma). \end{equation}

\paragraph*{\bf Invariance} The $\pi_{1}(\Sigma)$-invariance of currents with corners $\mu, \nu$ implies that $\mu \times \nu$ descends to a measure on $\mathcal{D}$. We then define $i(\mu, \nu)$ as 
\[ i(\mu, \nu) = \int_{\mathcal{D}} d\mu \times d\nu. \]
\end{Def}

\noindent We observe that this definition makes sense both in $\mathcal{CGCC}$ and in $\widetilde{\mathcal{CGCC}}$; in the latter, the marking simply has no role in the computation of the intersection number (although it will come into play later in this section when we introduce a related intersection form on the space of marked currents with corners).

\begin{remark} \label{self-intersect} We point out a key subtlety with this definition that distinguishes it from the classical construction of an intersection form on $\mathcal{GC}(\Sigma)$. Fix a simple closed geodesic, $c$, on $\Sigma$ and place a vertex $v$ somewhere along it. Then the set of lifts of $c$ to $\h^{2}$ decorated by the set of lifts of $c$ is a current with corners $\mu$ as in Example \ref{corners on a curve}: its support consists of all path-lifts of $c$ starting and ending at lifts of $v$. Let $\tilde{c}_{1}, \tilde{c}_{2}$ be a pair of lifts of $c$ sharing an endpoint (which is necessarily in the lift of $v$); it follows that $\tilde{c}_{1}, \tilde{c}_{2}$ are related by a generator in the infinite cyclic subgroup of $\pi_{1}(\Sigma)$, generated by some element in the conjugacy class corresponding to $c$. On the other hand, by convention, $(\tilde{c}_{1}, \tilde{c}_{2}) \in \widetilde{\mathcal{D}}$ because, as we emphasize again, \emph{ sharing an endpoint counts as a transverse intersection}. Since these lifts lie over the same geodesic segment on $\Sigma$, it follows that $i(\mu, \mu) \neq 0$, even though the support of $\mu$ descends to a (pointed) simple closed geodesic.     
\end{remark}

\paragraph*{\bf Intersecting graph and geodesic currents} The following example explores intersection between a graph current and a classical geodesic current:

\begin{Example} \label{not simple} Fix some graph current $\mu$ associated to some embedding of a weighted graph $\Gamma$, and let $\mathcal{L}_{\Sigma} \in \mathcal{GC}(\Sigma)$ denote the Liouville current for $\Sigma$. Then 
\[ i(\mu, \mathcal{L}_{\Sigma}) = \sum_{e \in E(\Gamma)} w(e) \cdot \ell(e). \]
This follows from the description of the Liouville measure $\mathcal{L}$ on $G(\h^{2})$ found in Section 4 of Aramayona-Leininger~\cite{AL}*{\S 4}. Indeed, the set of bi-infinite geodesics  $E_{\delta}$ transverse to a given unit-speed geodesic segment $\delta: (-\epsilon, \epsilon) \rightarrow \h^{2}$ can be parameterized by the angle $\theta$ of intersection and the time of intersection. With respect to these coordinates, 
\begin{equation} \label{Liouville formula}
\mathcal{L}|_{E_{\delta}}= \frac{1}{2} \sin(\theta) d\theta dt. 
\end{equation}

It then follows that $\mathcal{L}(E_{\delta})= \ell(\delta)$. And by definition $i(\mu, \mathcal{L}_{\Sigma})$ is the sum-- over each edge $e$ of $\Gamma$-- of the product of $\mathcal{L}(\tilde{e})$ (where $\tilde{e}$ is some lift of $e$ to $\h^{2}$) and $w(e)$. 
    
\end{Example}

\paragraph*{\bf Discontinuity} Note that, unlike in the classical setting, this intersection form is not continuous. Even its restrictions to a $1$-variable function where we keep one argument fixed need not be continuous, as the following example demonstrates: 

\begin{Example} \label{not continuous current}
Let $\mu \in \mathcal{HC}_{k}$ be a graph current and consider the function 
\[ i(\cdot, \mu): \mathcal{GC}(\Sigma) \rightarrow \mathbb{R}. \]
 Let $c \in \mathcal{GC}(\Sigma)$ and let $p \in \h^{2}$ be a corner of $\mu$ such that, in the cyclic order of edges incident to $p$, there exists a pair of consecutive edges of $\mu$ that meet $p$ with angle $\pi$. Suppose further that $c$ has an atom passing through $p$. There may then be a sequence of geodesic currents $\left\{c_{i} \right\}$ converging to $c$ so that the corresponding sequence of atoms do not intersect any edge of $\mu$ transversely. In this case, we have that 

 $$i(\mu, c) - i(\mu, c_i) > \epsilon >0,$$ that is $i(\mu, c)$ will be a definite amount larger than $i(\mu, c_{i})$ for all $i$.

\end{Example}

Example \ref{not continuous current} could be generalized to involve a current $c$ with no atom through $p$, but which assigns positive measure to the set of bi-infinite geodesics passing through $p$. In more detail, define $\mathcal{J}(p) \subset G(\h^{2})$ by
 \[ \mathcal{J}(p) = \left\{\gamma \in G(\h^{2}): \gamma \hspace{1 mm} \mbox{passes through } p \right\}. \]
 Note that $\mathcal{J}(p)$ is compact; this follows from the fact that $\mathcal{J}(p_{1})$ is homeomorphic to $\mathcal{J}(p_{2})$ for any $p_{1}, p_{2} \in \h^{2}$ and that $\mathcal{J}_{0}$ (where $0$ denotes the origin of the unit Disk) is clearly a copy of $S^{1}$. 

Motivated by Example \ref{not continuous current}, we say that a current $c \in \mathcal{GC}(\Sigma)$ is \textbf{without bottlenecks} if it assigns measure $0$ to each $\mathcal{J}(p)$; the terminology comes from visualizing $\mathcal{J}(p)$ as a very ``narrow'' collection of geodesics that pass through a given point. If $c$ is not without bottlenecks, we will refer to $p \in \h^{2}$ as a \textbf{bottleneck point} for $c$ if $c(\mathcal{J}(p)) \neq 0$. In Example \ref{not continuous current}, we could just as well have considered the function 
 \[ i(c, \cdot): \mathcal{HC}_{k} \rightarrow \mathbb{R} \] 
 for some $c \in \mathcal{GC}(\Sigma)$. Then one can imagine a sequence of graph currents $(\mu_{i})$ converging to a graph current $\mu$ with a vertex $p$ so that $c(\mathcal{J}(p))\neq 0$, and so that the vertices of each $\mu_{i}$ are never bottleneck points of $c$. 

Note also that Equation \ref{Liouville formula} implies that $\mathcal{L}_{\Sigma}$ is without bottlenecks, and that being without bottlenecks implies having no atoms. The converse holds for measured laminations. Example \ref{not continuous current} motivates the idea that if we restrict to currents without bottlenecks, the intersection form may behave continuously-- the following lemma formalizes this: 

\begin{lemma} \label{intersection function continuous} Suppose $\mu \in \mathcal{GC}(\Sigma)$ is without bottlenecks. Suppose also that $(\nu_{n})$ is a sequence of marked graph $k$-currents converging to $\nu$, so that the supremum, taken over all $n$, of the total mass of $\nu_{n}$ is finite. Then  $$\lim_{n \rightarrow \infty} i(\nu_{n}, \mu) = i(\nu, \mu).$$ 
\end{lemma}

\begin{proof} Assume first that $\nu \in \widetilde{\mathcal{CGCC}}_{k}$ is a graph current. Let $(\nu_{n})$ be a sequence of graph currents converging to $\nu$. The intersection $i(\nu, \mu)$ can be concretely computed as follows: let $\tilde{e}_{1},\ldots, \tilde{e}_{j}$ (for some j $\le k$) be lifts of the edges $e_{1},\ldots, e_{j}$ of the graph on $\Sigma$ to $\h^{2}$ associated to $\nu$; then 
\[ i(\nu, \mu) = \sum_{m=1}^{j} \nu(\tilde{e}_{m}) \cdot \mu (\mathcal I (\tilde{e}_m))\] where, by slight abuse of notation,  \[ \mathcal I (\tilde{e}_m)= \left\{\gamma \in G(\h^{2}): \gamma \hspace{1 mm} \mbox{intersects } \tilde{e}_{m} \right\} = \bigcup_{p \in \tilde{e}_m} \mathcal{I}(p). \]

Let $\left\{B^{(n)}_{m}\right\}_{n=1}^{\infty}$ be a nested sequence of segmented flow boxes intersecting in $\tilde{e}_{m}$ and let $\{\epsilon_{n}\}_{n=1}^{\infty}$ be a sequence of positive real numbers converging to $0$. By passing to a subsequence and relabeling as necessary, we can assume that 
\[ \left|\nu_{n}\left(B^{(n)}_{m}\right)- \nu(\tilde{e}_{m})\right| < \epsilon_{n}, m=1,\ldots,j. \]
Now, let $T^{(n)}_{m} \subset G(\h^{2})$ denote the set of bi-infinite geodesics that transversely intersect either $\tilde{e}_{m}$, or at least one edge of $\nu_{n}$ in $B^{(n)}_{m}$, but not both. It suffices to prove that
\begin{enumerate}
\item \begin{equation} \label{limit to 0}
 \lim_{n \rightarrow \infty} \mu\left(T^{(n)}_{m}\right) = 0 , m=1,\ldots, j. 
 \end{equation}
\item Let $\left\{E_{n} \right\}$ be a sequence of edges with $E_{n}$ in the support of $\nu_{n}$ such that the hyperbolic lengths of $E_{n}$ converge to $0$. Then so does the contribution to $i(\nu_{n}, \mu)$ coming from $E_{n}$. 
\end{enumerate}
Indeed, the difference between $\lim_{n \rightarrow \infty}i(\nu_{n}, \mu)$ and $i(\nu, \mu)$ comes either from geodesics in $\bigcap_{n=1}^{\infty}T^{(n)}_{m}$ in the support of $\mu$, or from geodesics in the support of $\mu$ that intersect edges in each $\nu_{n}$ which collapse to points in the limiting current $\nu$. 
The intersection $\bigcap_{n=1}^{\infty}T^{(n)}_{m}$ is contained in the set
\[ A^{m}:=\mathcal{J}(p^{m}_{1}) \cap \mathcal{J}(p^{m}_{2}),\]
where $p^{m}_{1}, p^{m}_{2}$ are the endpoints of $\tilde{e}_{m}$.

Now, by way of contradiction, assume that \ref{limit to 0} does not hold; then there is some $\epsilon>0$ and subsets $S^{(n)}_{m} \subset T^{(n)}_{m}$ (after possibly passing to a subsequence and relabeling) so that
\begin{itemize}
\item $S^{(n)}_{m} \cap T^{(n+1)}_{m} = \emptyset$ (and so in particular, $S^{(n)}_{m} \cap S^{(n+1)}_{m} = \emptyset$);
\item $\mu\left(S^{(n)}_{m}\right) > \epsilon$.
\end{itemize}
Note that this uses the no-bottleneck assumption of $\mu$: if some neighborhood of $A^{m}$ has $\mu$-measure at least $\epsilon$, then we must be able to find a flow box nearby to $A^{m}$ and not containing $A^{m}$ with definite $\mu$-mass at least $\epsilon$. When $n$ is large, we can choose $S^{(n)}_{m}$ to be such a flow box. If the complement of $S^{(n)}_{m}$ in $A^{m}$ has $\mu$-mass less than $\epsilon$, it follows that the $\mu$-mass of $A^{m}$ is also less than $\epsilon$; thus if \ref{limit to 0} does not hold we can iterate this procedure and choose some $S^{(n+1)}_{m}$ satisfying the desired properties. 

It follows that every open neighborhood of $A^{m}$ has infinite $\mu$-mass, contradicting local finiteness. Therefore \ref{limit to 0} holds, as desired.

 It remains to consider contributions to $i(\mu, \nu_{n})$ coming from edges of $\nu_{n}$ with small hyperbolic length. We claim that given any $\epsilon>0$ there is $K>0$ so that when $n$ is large, the contribution to $i(\nu_{n}, \mu)$ coming from the complement of $GS(\Sigma)_{K} \times GS(\Sigma)_{K}$ is at most $\epsilon$. Note that this automatically holds if the hyperbolic length of any geodesic segment in the support of some $\nu_{n}$ is bounded uniformly from below. 

To prove the claim, let $\mathbb{P}T^{\infty}\tilde{\Sigma}$ denote the quotient of $T^{\infty}\tilde{\Sigma}$ by the antipodal map on each fiber. A point of $\mathbb{P}T^{\infty}\tilde{\Sigma}$ is specified by choosing some $p=(x,[v])$ in the projectivized unit tangent bundle $\mathbb{P}T\tilde{\Sigma}$ and some $\lambda \in [0, \infty]$. As in Section \ref{currents with corners}, we will use \textbf{0} to denote the $0$-section of $\mathbb{P}T^{\infty}\tilde{\Sigma}$. Let $\eta_{\lambda}$ be the map 
\[ \eta: \mathbb{P}T^{\infty}\tilde{\Sigma} \setminus \textbf{0}  \rightarrow  GS(\tilde{\Sigma})\]
which associates to, $p=(x, [v])$, the geodesic segment $\eta(p, \lambda)$ centered at $x$, in the direction of $[v]$, and with hyperbolic length $\lambda$. Then $\eta_{\lambda}$ descends to a map on $$\mathbb{P}T^{\infty}\Sigma \setminus \textbf{0} := \left(\mathbb{P}T^{\infty}\tilde{\Sigma} \setminus \textbf{0} \right)/\pi_{1}(\Sigma) $$ which we denote by
\[ \overline{\eta}_{\lambda}:  \mathbb{P}T^{\infty}\Sigma \setminus \textbf{0}  \longrightarrow GS(\Sigma). \]

Define the function \[\Phi_{\mu}: \mathbb{P}T^{\infty}(\Sigma) \setminus \textbf{0} \rightarrow [0, \infty] \]
by 
\[ \Phi_{\mu}(p, \lambda) = \mu \left(g \in G(\h^{2}): g \hspace{1 mm} \mbox{intersects } \widetilde{\overline{\eta}_{\lambda}(p)} \right),  \]
where $\widetilde{\overline{\eta}_{\lambda}(p)}$ is an arbitrary choice of lift of $\overline{\eta}_{\lambda}(p)$ to $\tilde{\Sigma}$. In words, $\Phi_{\mu}$ sends a geodesic segment on $\Sigma$ to the $\mu$-measure of the set of bi-infinite geodesics intersecting some lift of it to $\tilde{\Sigma}$. 

Finally, let 
\[ j_{\mu}: \mathbb{P}T\Sigma \rightarrow (0,\infty] \]
denote the supremum of the set of $\lambda \in (0, \infty]$ so that 
\[ \Phi_{\mu}(p, \lambda) < \frac{\epsilon}{2 \cdot M},\]
where $M$ is an upper bound for the total mass of each $\nu_{n}$. Then well-definedness and continuity of $j_{\mu}$ follows from the assumption that $\mu$ is without bottlenecks. Compactness of $\mathbb{P}T\Sigma$ implies there is an attained minimum value $K$; then as desired, the total contribution to $i(\nu_{n}, \mu)$ coming from all edges of $\nu_{n}$ with length less than $K$, is at most $\epsilon$. 

Therefore, the contribution to $i(\nu_{n}, \mu)$ coming from edges of $\nu_{n}$ with hyperbolic length going to $0$ as $n \rightarrow \infty$, goes to $0$ as well. This completes the argument under the assumption that $\nu$ is a graph current. 

It remains to consider the case that $\nu \in \widetilde{\mathcal{CGCC}}_{k}$ is a  geodesic current and that $(\nu_{n})$ is a sequence of graph currents converging to $\nu$. We note that Bonahon's original proof of continuity for the intersection form $i(\cdot, \cdot)$ on $\mathcal{GC}(\Sigma)$ simplifies greatly at points $(\mu, \nu)$ for which at least one of $\mu$, $\nu$ has no atoms. Our argument is inspired by this observation and we follow the outline of Bonahon's proof~\cite{Bonahon1}*{\S 4.2} in this setting.

Given $\epsilon>0$, we claim that it suffices to find $K>0$, $N \in \mathbb{N}$, and an open set $U$ containing \begin{enumerate} \item a neighborhood of the diagonal $\Delta$ in $GS(\Sigma) \times GS(\Sigma)$ and 
\item all pairs $(\gamma_{1}, \gamma_{2}) \in GS(\Sigma) \times GS(\Sigma)$ such that each $\gamma_{i}$ has hyperbolic length less than $K$, such that the contribution of $U$ to both $i(\nu_{n}, \mu)$ and $i(\nu, \mu)$ is at most $\epsilon$ for all $n \ge N$. 
\end{enumerate}

Indeed, the set $\mathcal{D} \setminus U$ is compact, and so it can be covered with finitely many sets of the form $B_{i} \times B'_{j}$ where:
\begin{itemize}
    \item $B_{i}$ (resp. $B'_{j}$) is $\nu$-admissible (resp. $\mu$-admissible).
    \item $B_{i}$ is transverse to $B'_{j}$, meaning that every geodesic segment in $B_{i}$ intersects every geodesic segment in $B'_{j}$.
    \item The total contribution to $i(\mu, \nu)$ coming from $\mathcal{D} \setminus U$ is within $\epsilon$ of $\sum_{i,j}\mu(B_{i}) \cdot \nu(B'_{j})$. 
\end{itemize}
The definition of weak$^{\ast}$-convergence then allows us to find $N$ so that for all $n\ge N$, the contribution to $i(\nu_{n}, \mu)$ coming from $\mathcal{D} \setminus U$ is within $\epsilon$ of $\sum_{i,j}\mu(B_{i}) \cdot \nu(B'_{j})$, whence it follows that $i(\nu_{n}, \mu)$ is within $4\epsilon$ of $i(\nu, \mu)$. 

It remains to construct $U$. Recall that above, we demonstrated that there is $K>0$ so that when $n$ is large, the contribution to $i(\nu_{n}, \mu)$ coming from the complement of $GS(\Sigma)_{K} \times GS(\Sigma)_{K}$ is at most $\epsilon/2$. Set $V$ equal to the complement of $GS(\Sigma)_{K} \times GS(\Sigma)_{K}$. 

Now, cover $GS(\Sigma)_{K} \times GS(\Sigma)_{K}$ with finitely many segmented flow boxes $\mathcal{B}_{1},..., \mathcal{B}_{r}$. Subdivide each flow box into smaller flow boxes $(\mathcal{B}'_{i})$ with disjoint interiors so that 
\[ \mu(\mathcal{B}'_{i}) < \frac{\epsilon}{ 2 \sum_{j} \nu(\mathcal{B}_{j}) }.  \]
Note that the existence of such a subdivision relies on the assumption that $\mu$ has no atoms. It follows that the total contribution to $i(\nu_{n}, \mu)$ of $\mathcal{B}'_{i} \times \mathcal{B}'_{i}$ is, when $n$ is large, very nearly at most
\[\frac{\epsilon}{2\sum_{j} \nu\left(\mathcal{B}_{j}\right) } \cdot \nu\left(\mathcal{B}'_{i}\right). \]
Therefore, the total contribution to either $i(\nu_{n}, \mu)$ or to $i(\nu, \mu)$ of the union, taken over each $\mathcal{B}'_{i} \times \mathcal{B}'_{i}$, will be at most $\epsilon/2$. Then let 
\[ U = V \cup \bigcup_{i} \mathcal{B'}_{i} \times \mathcal{B}'_{i}. \]
 
\end{proof}

\begin{remark} The assumption of no atoms is not always necessary for sequences $(\nu_{n})$ that approach a geodesic current. To see this, suppose that $(\nu_{n})$ is as in Example \ref{any curve} so that $\lim_{n \rightarrow \infty} \nu_{n} = \nu \in \mathcal{GC}(\Sigma)$ is a closed curve, and suppose $\mu \in \mathcal{GC}$ is another closed curve. Let $e$ be the edge of the domain graph $\Gamma$ which is sent--by the marking map of $\nu_{n}$-- to a loop traversing (the underlying closed curve equal to the support of) $\nu$ $n$ times. It follows that a lift of $e$ to $\tilde{\Sigma}$ will intersect $n \cdot i(\nu, \mu)$ geodesics in the support of $\mu$. Since $\nu_{n}$ is multiplied by $1/n$, the intersection converges to $i(\nu, \mu)$ in the limit, as desired. 
\end{remark}

The above discussion suggests that we can not expect the continuity of a function of the form $i_{\mu}: \mathcal{ML}(\Sigma) \rightarrow \mathbb{R}$ given by intersecting against some graph current $\mu$. However, the following lemma states that at the very least, this function is measurable with respect to Thurston measure:

\begin{lemma} Consider the function $i_{\gamma}: \mathcal{ML}(\Sigma) \rightarrow \mathbb{R}$ defined by $i_{\gamma}(\lambda) = i(\gamma, \lambda)$ where $\gamma$ is some graph current. Then $i_{\gamma}$ is measurable with respect to the Thurston measure. 
\end{lemma}

\begin{proof}  Fix $L$ and consider $i_{\gamma}^{-1}([0,L))$. Let $\lambda \in \mathcal{ML}(\Sigma)$ be in this pre-image. Assume first that no leaf of $\lambda$ intersects a vertex of $\gamma$. Then since the function sending a measured lamination to its geodesic representative on $\Sigma$ is continuous (say, for instance, with respect to the piecewise linear structure on $\mathcal{ML}(\Sigma)$ induced by train-track coordinates), we can then find a small neighborhood $U$ of $\lambda$ such that no leaf of any lamination in $\mu$ intersects a vertex of $\gamma$ either. It follows that $U$ can be chosen sufficiently small so that it lies entirely in $i_{\gamma}^{-1}([0,L))$. 

On the other hand, if some leaf of $\lambda$ intersects a vertex of $\gamma$, then arbitrary small perturbations of $\lambda$ may yield laminations that miss this vertex, but this will only reduce intersection number and so again we can find some neighborhood $U$ of $\lambda$ contained in $i_{\gamma}^{-1}([0,L))$. Therefore the pre-image of an open set $(L_{1}, L_{2})$ is the complement of an open set inside of another and is thus measurable. 
\end{proof} 

Finally, we observe that the vanishing of \textit{self}-intersection in a limit implies that the limit is a lamination: 

\begin{lemma} \label{limit 0 implies lamination} Let $(\mu_{j})$ be a sequence of marked harmonic $k$-currents converging to $\mu$ as $j \rightarrow \infty$. Suppose that 
\[ \lim_{j \rightarrow \infty} i(\mu_{j}, \mu_{j}) = 0.\]
Then $\mu \in \mathcal{ML}(\Sigma)$
\end{lemma}

\begin{proof} We first point out that $\mu$ can not be a graph current (unless it's the $0$ current in which case it is also a measured lamination). Indeed, if $\mu$ is a graph current, the weight vectors $(w_{j})$ have to converge to the weight vector $w$ of $\mu$, and if $w$ is not the $0$-vector (see Remark \ref{ghost edge annoyance} below for a minor caveat), there is a definite positive contribution to $i(\mu, \mu)$ coming from edges incident at the same vertex. Then $i(\mu_{j}, \mu_{j})$ receives very close to this same contribution when $j$ is large, and therefore the limiting intersection number can not be $0$. 

\begin{remark} \label{ghost edge annoyance} Another possibility is that, in the limit, there is only one edge incident to any given vertex with positive weight, and that all other edges incident to that vertex are ghost edges. In this situation, the limiting current is actually just the $0$ current because the harmonic representative will contract all non-ghost edges to points.
\end{remark}

Therefore, we can assume that $\mu \in \mathcal{GC}(\Sigma)$. In that case, in the event that $\mu \notin \mathcal{ML}(\Sigma)$, we can consider transverse $\mu$-admissible flow boxes $B_{1}, B_{2}$ (meaning that every geodesic segment in $B_{1}$ is transverse to every geodesic segment in $B_{2}$ and visa versa) which are both assigned positive mass by $\mu$. Then $\mu_{j}$ assigns definitely positive mass to each of $B_{1}, B_{2}$ for large $j$, contradicting $i(\mu_{j},\mu_{j}) \rightarrow 0$.

\end{proof}

\subsection{A compactness criterion}

We say that a geodesic current with corners is \textbf{filling} if it has positive intersection number with every simple closed geodesic. The following compactness criterion will be crucial for our main counting theorem:

\begin{Theorem} \label{compact} Let $\mu \in \mathcal{HC}$ be a filling current and let $R \geq 0$. Then the set 
\[ C_{R}(\mu) = \left\{ \nu \in \mathcal{HC}: i(\mu, \nu) \leq R \right\} \]
is pre-compact. 
\end{Theorem}

\begin{proof} We follow the general outline of Bonahon's original proof~\cite{Bonahon}, as exposited by Aramayona-Leininger~\cite{AL}*{Theorem 4.7}; for clarity, we use the same notation found there when possible. To begin, let $$\Omega = C_{c}(GS(\mathbb{H}^{2}), \mathbb{R}_{\geq 0})$$ be the space of compactly supported continuous and nonnegative real-valued functions on the space of geodesic segments. There is then a map $\Phi: \mathcal{HC} \rightarrow \mathbb{R}^{\Omega}$ defined by 
\[ \nu \mapsto \left( \int_{GS(\mathbb{H}^{2})} f d\nu \right)_{f \in \Omega}, \]
given some $\nu \in \mathcal{HC}$. We claim that when the codomain is equipped with the product topology, $\Phi$ is a topological embedding. To see this, first note that if $\nu, \nu' \in \mathcal{HC}$ are distinct harmonic currents, then there is some geodesic segment $\gamma$ so that $\nu(\gamma) \neq \nu'(\gamma)$. Then letting $\tilde{\gamma}$ denote the full lift of $\gamma$ to $\mathbb{H}^{2}$ and given some compact $A \subset \tilde{\gamma}$ and any $\epsilon>0$, there is $f \in \Omega$ with $|f- \nu(\gamma) \cdot \chi_{A}|_{L^{1}} <\epsilon$, where $\chi_{A}$ denotes the indicator function for $A$. Then for sufficiently small choice of $\epsilon$, we must have that $\Phi(\nu) \neq \Phi(\nu')$ in the $f$-component. This proves that $\Phi$ is injective. 

\paragraph*{\bf Continuity} For continuity, assume that $\left\{\nu_{n} \right\}_{n=1}^{\infty}$ is some sequence of harmonic currents satisfying $$\lim_{n \rightarrow \infty} \nu_{n} = \nu.$$ Then for any $f \in \Omega$, we must have 
\[ \lim_{n \rightarrow \infty} \int f d\nu_{n} = \int f d\nu\]
by definition of the weak$^{\ast}$ topology, whence it follows that $$\lim_{n \rightarrow \infty} \Phi(\nu_{n})= \Phi(\nu),$$ as desired. Finally, continuity of $\Phi^{-1}$ again follows immediately from the definition of the weak$^{\ast}$ topology. 

\paragraph*{\bf Compactness} It therefore suffices to prove compactness of $\Phi(C_{R}(\mu))$. Using the Tychonoff theorem, our goal then becomes to prove the compactness of $\pi_{f}(\Phi(C_{R}(\mu)))$, where $\pi_{f}$ denotes the projection of $\mathbb{R}^{\omega}$ on to the $f$-component.  Now, for each $\nu \in \mathcal{HC}$ and each $f \in \Omega$, 
\[ \int f d\nu \leq \mbox{max}(f) \cdot \nu(\mbox{support}(f)), \]
and thus it suffices to prove that 
\[ \left\{ \nu(\mbox{support}(f)) | \nu \in C_{R}(\mu) \right\} \]
is bounded in terms only of $f$. Since the support of $f$ is compact, this amounts to proving that for $K \subset GS(\mathbb{H}^{2})$ compact,
\[ \left\{ \nu(K) | \nu \in C_{R}(\mu) \right\} \]
is bounded in terms only of $K$. 

\paragraph*{\bf Bounded neighborhoods} For this, we claim that it suffices to prove that for every $\delta \in GS(\mathbb{H}^{2})$, there is a neighborhood $U_{\delta}$ such that $\left\{\nu(U_{\delta}) : \nu \in C_{R}(\mu) \right\}$ is bounded. Indeed, since $K$ is compact, if the previous claim is verified, we  can cover $K$ by finitely many such neighborhoods and apply the claim to each.  If $\delta$ intersects the lift of $\mu$ to $\mathbb{H}^{2}$, then let $\eta \in GS(\mathbb{H}^{2})$ be a specific segment in the support of $\mu$ intersecting $\delta$. There are then neighborhoods $U_{\eta}$ of $\eta$ and $U_{\delta}$ of $\delta$ so that $U_{\eta} \times U_{\delta} \subset \mathcal{D}$, and so that $U_{\eta} \times U_{\delta}$ projects homeomorphically into $GS(\Sigma) \times GS(\Sigma)$. It then follows that, as desired:
 
 \[ i(\mu, \nu) \geq (\nu \times \mu)(U_{\delta} \times U_{\eta})= \nu(U_{\delta}) \cdot \mu(U_{\eta}),  \]
 and since $\eta$ is in the support of $\mu$, $\mu(U_{\eta}) \neq 0$ and we obtain 
 \[ \nu(U_{\delta}) \leq \frac{R}{\mu(U_{\eta})}. \]
 
\paragraph*{\bf Complementary components} On the other hand, if $\delta$ does not intersect the lift of $\mu$ to $\mathbb{H}^{2}$, it lies in a complementary component which is necessarily simply connected since $\mu$ is filling. Then if some $\nu \in C_{R}$ assigns a very large mass to a small neighborhood of $\delta$, it follows that $\nu$ has an edge with a large $\nu$-mass and which lies in the same complementary component. Then, the balance equations imply that a large $\nu$-mass has to be visible along the boundary of that complementary component, and this will contribute to the intersection number between $\mu$ and $\nu$. More formally, letting $e_{0}$ denote the aforementioned edge of $\nu$ with large $\nu$-mass, assume that 
\[ \nu(e_{0}) \geq \frac{k^{k+1} \cdot 2D \cdot W \cdot R}{\ell(\delta)}, \]
where $D$ is an upper bound on the diameter of every complementary component of $\mu$, and where $W$ is an upper bound on the reciprocal of the weight of any edge of $\mu$. 
Since the length of $e_{0}$ must be close to the length of $\delta$, the weighted length of $e_{0}$ is bounded below by $k^{k} \cdot R$. It follows that at least one edge incident to an endpoint of $e_{0}$-- call it $e_{1}$--  must have weighted length at least $k^{k-1} \cdot R$. We now iterate this argument, replacing $e_{0}$ with $e_{1}$ and applying the balance condition at the other endpoint of $e_{1}$; since there are only at most $k$ edges of $\nu$ and since $\nu$ is harmonic (and therefore its support can not descend to a set of edges lying entirely in a disk of $\Sigma$), we eventually arrive at some edge $e_{n}$ of $\nu$ with weighted length at least $2D \cdot W \cdot R$ and which crosses edges of $\mu$. Now, let $w= w(e_{n})$ denote the weight of $e_{n}$; then 
\[ \ell(e_{n}) \ge \frac{2D\cdot R \cdot W}{w}.\]
It follows that $e_{n}$ intersects at least $(2R \cdot W)/w$ edges of $\mu$; each intersection contributes at least $w/W$ to the total intersection number $i(\mu, \nu)$ and so $i(\mu, \nu) \ge 2R> R$. 
\end{proof}

\begin{remark} \label{compactness for marked k currents} We note that Theorem \ref{compact} and Lemma \ref{intersection function continuous} implies that if $\mu \in \mathcal{GC}(\Sigma)$ is filling and without bottlenecks, then 
\[ \widetilde{C_{R}(\mu)}_{M, N} = \left\{ \nu \in \widetilde{\mathcal{HC}}_{k}: i(\mu, \nu) \leq R, \nu(e) \leq M \forall \hspace{1 mm}  \mbox{atoms } e \hspace{1 mm} \mbox{of } \nu, \frac{\max(\nu(e): e= \hspace{1 mm} \mbox{atom})}{\min(\nu(e): e= \hspace{1 mm} \mbox{atom})}  \leq N \right\}  \]
is compact, where $i(\mu, \nu)$ is defined to be the intersection between $\mu$ and the unmarked harmonic current obtained from $\nu$ by forgetting the marking. To see this, note first that $\widetilde{C_{R}(\mu)}_{M,N}$ is closed; this follows from the choice of topology on $\widetilde{\mathcal{HC}}_{k}$ and Lemma \ref{intersection function continuous}. Indeed, the set of marked harmonic $k$-currents with an (inclusive) upper bound on the mass of any atom is closed, and Lemma \ref{intersection function continuous} implies that the function $i(\mu, \cdot)$ is continuous on this set. Therefore, $\widetilde{C_{R}(\mu)}_{M, N}$ is a closed set contained in the product of three pre-compact sets: the closed ``half'' of the cube in $\mathbb{R}^{n{k}}$ with side lengths $M$ consisting of points where the ratio between any two components is at most $N$, $\overline{\mathcal{F}}(\Delta_{n_{k}})$, and (by Theorem \ref{compact}), the set of harmonic $k$-currents whose intersection with $\mu$ is at most $R$. 

\end{remark}



\subsection{Homotopical intersection} \label{homotopical instersection}

A major drawback of the intersection number $i( \cdot, \cdot)$ is that it is not $\mathcal{MCG}(\Sigma)$-invariant, as the following example demonstrates:

\begin{figure}[h!]
\includegraphics[width=0.4\textwidth]{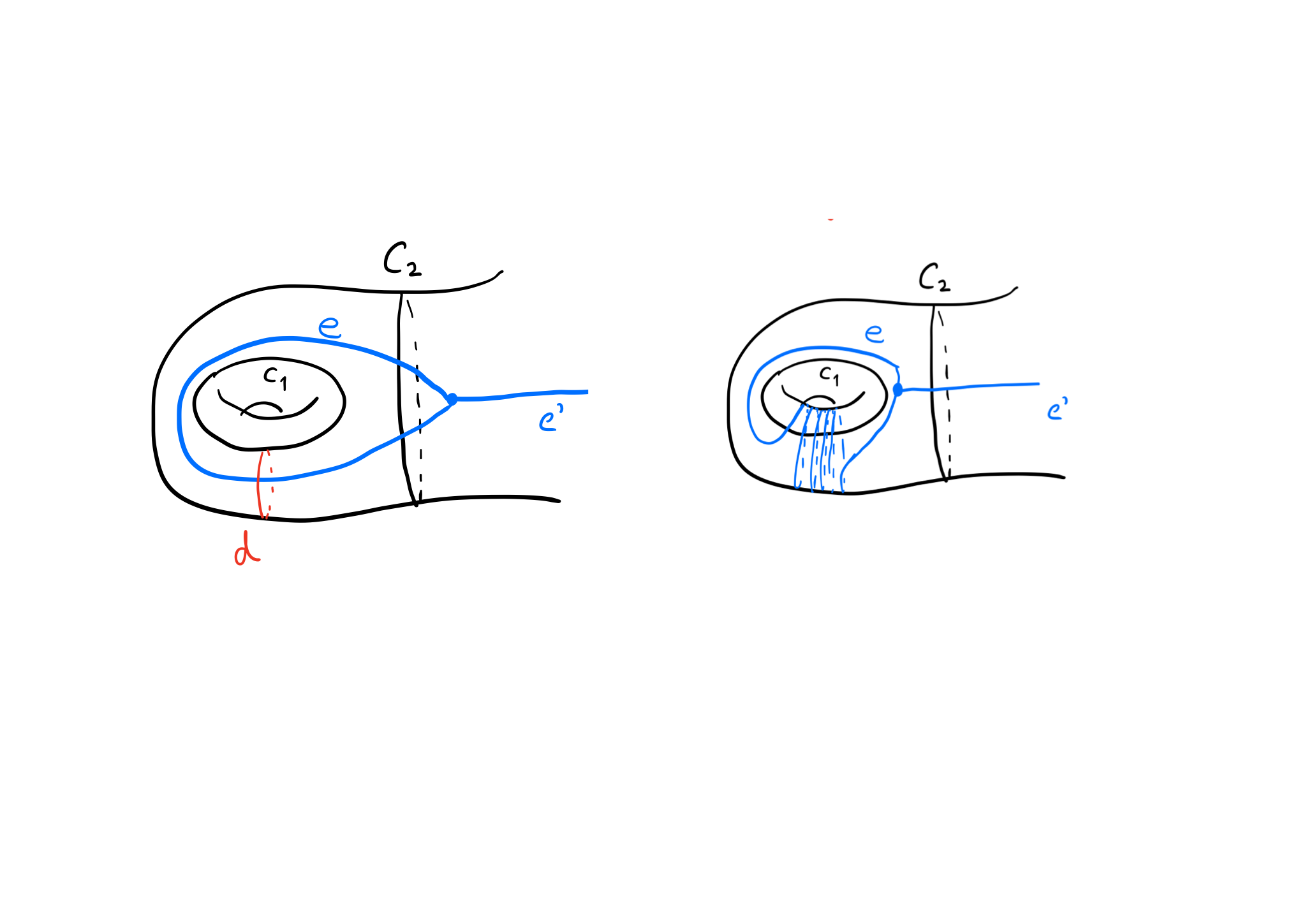}\caption{ If $w(e') >> w(e)$, the harmonic map may stretch $e$ over $c_2$ in order to minimize the length of $e'$. After twisting $e$ about $d$ sufficiently many ($n>>1$) times, the harmonic map will send $e$ very close to the geodesic representative for the twisted curve $T_d^n(c_1),$ so $$i(\mu, c_2) \neq i(\mu_n, c_2).$$}\label{fig:notmcginvariant}
\end{figure}

\begin{Example} \label{intersection not mcg invariant} Let $c_{1}, c_{2}$ be a pair of disjoint simple closed geodesics on $\Sigma$. Let $([\phi] :\Gamma \rightarrow \Sigma, \mu) \in \widetilde{\mathcal{HC}}$ be a marked harmonic current whose marking graph $\Gamma$ contains a self-loop $e$ that maps to a closed curve freely homotopic to $c_{1}$. Letting $e'$ be some edge incident to $e$, we can choose $\phi, \Gamma$, and the weighting on $\Gamma$ in such a way so that the image of $\Gamma$ intersects $c_{2}$ twice and is disjoint from $c_{1}$ (see figure). For example, this will happen if the weight on $e'$ is much larger than all other weights; in this case, the harmonic representative will make $e'$ very short and will send $e$ to a loop (inessentially) crossing $c_{2}$. 

Now, let $\mu_{n}$ be the marked harmonic current obtained from $\mu$ by composing the marking map with $T^{n}_{d}$, where $d$ is a simple closed curve disjoint from $c_{2}$ and intersecting $c_{1}$ once. For $n$ very large, the harmonic map will send $e$ to a loop that is very nearly the geodesic representative of $c_{1}$, and so for such $n$, the image of $e'$ will intersect $c_{2}$ and not the image of $e$. The weights on $e,e'$ can be chosen in such a way so that the above happens and so that $w(e') >> w(e)$ (in particular, $w(e') > 2w(e)$), and thus $i(\mu_{n}, c_{2}) \neq i(\mu, c_{2})$, see Figure~\ref{fig:notmcginvariant}. 
    
\end{Example}

For this reason, we need to define a different intersection form- the \textit{homotopical intersection number}, which will, by construction, be mapping class group invariant: 

\begin{Def} \label{homotopical} Let $c_{1}, c_{2} \in \widetilde{\mathcal{HC}}$. If $c_{1}, c_{2}$ are both classical geodesic currents, define their homotopical intersection number $h(c_{1}, c_{2})$ to be equal to their classical intersection number. If $c_{1}$ is a marked current with corners and $c_{2}$ is a classical current, define $h(c_{1}, c_{2})$ to be the minimum intersection number taken over the entire homotopy class of $c_{1}$. If both $c_{1}, c_{2}$ are marked currents with corners, define $h(c_{1}, c_{2})$ to be the double minimum intersection number, taken over all homotopy representatives of both $c_{1}, c_{2}$:

\[ h(c_{1}, c_{2}) =   \begin{cases} 
      i(c_{1}, c_{2}) & c_{1}, c_{2} \in \mathcal{GC}(\Sigma) \\
      \min_{c'_{1} \sim c_{1}} i(c'_{1}, c_{2}) & c_{1} \notin \mathcal{GC}(\Sigma), c_{2} \in \mathcal{GC}(\Sigma) \\
      \min_{c'_{1} \sim c_{1}, c'_{2}\sim c_{2}} i(c'_{1}, c'_{2}) & c_{1}, c_{2} \notin \mathcal{GC}(\Sigma)
   \end{cases}
\]
\end{Def}

Again, we emphasize that by construction, homotopical intersection number is invariant under the $\mathcal{MCG}(\Sigma)$ action, and homogeneous in both coordinates.

\subsection{Key properties so far established} \label{properties}

The point of this subsection is to summarize the properties of the space of marked harmonic $k$-currents (and the space of harmonic currents) that we will need and that we have so far verified: 

\begin{enumerate}
\item If $\mu$ is a marked harmonic $k$-current with corners that is filling, then 
 \[ \left\{ \nu \in \widetilde{\mathcal{HC}}_{k}: i(\mu, \nu) \leq R, \nu(e) \leq M \forall \hspace{1 mm}  \mbox{atoms } e \hspace{1 mm} \mbox{of } \nu, \frac{\max(\nu(e): e= \hspace{1 mm} \mbox{atom})}{\min(\nu(e): e= \hspace{1 mm} \mbox{atom})}  \leq N \right\} \] 
 is compact (Theorem \ref{compact} and Remark \ref{compactness for marked k currents}). 
 \item $\widetilde{\mathcal{HC}}_{k}$ is locally compact Hausdorff (Lemma \ref{Hausdorff}, Lemma \ref{intersection function continuous}, Remark \ref{compactness for marked k currents}).
\item The mapping class group acts on $\widetilde{\mathcal{HC}}$ by Borel transformations (Lemma \ref{mcg action} and Theorem \ref{Borel action}). 
\item The homotopical intersection number is homogeneous and invariant under the $\mathcal{MCG}(\Sigma)$-action on the space of marked harmonic currents (Definition \ref{homotopical}).  
\end{enumerate}

Before launching into the proof of the main counting result in the next section, we establish a few additional facts that we will need: 

\begin{lemma} \label{boundary zero} Let $c$ be a filling marked harmonic $k$-current with corners. Then 
\[ \mu_{Th}( \lambda \in \mathcal{GC}(\Sigma) : h(\lambda, c) = 1) = 0, \]
where $\mu_{Th}$ denotes the push-forward of the Thurston measure to the space of classical currents $\mathcal{GC}(\Sigma)$. 
\end{lemma} 

\begin{proof} We first claim that the pre-image of $(0,1]$ under $h( \cdot, c)$ in the set of classical currents is pre-compact. Indeed, if $h(c, \lambda) \leq 1$, then $i(c, \lambda)$ is bounded above in terms only of $c$: there is a homotopy taking the representative of $c$ minimizing intersection with $\lambda$ to the harmonic representative for $c$ which is expressible as a concatenation of finitely many moves in which a vertex of $c$ slides over strands of $\lambda$. The number of required moves is bounded by the number of edges in $c$, and each move can increase intersection by at most the difference between the smallest and largest atoms of $c$.

Therefore, this set $A$ has some finite Thurston measure. Note that $A \cap (L \cdot A)= \emptyset$ for all $L$, and that $\mu_{Th}(L \cdot A) = L^{6g-6} \mu_{Th}(A)$ by standard properties of the Thurston measure. Since the Thurston measure is in the Lebesgue class, it's Radon, which means that the measure of a compact set can be approximated from above by open subsets. But, if $\mu_{Th}(A) \neq 0$, the measure of any open subset containing it would have to be infinite because it will contain a set that can be disjointly partitioned into uncountably infinitely many dilates of $A$ by scalars all of size at least $1$.  
\end{proof}

\begin{remark} \label{measurable} The above proof assumes that $A$ is measurable at all. For this, we briefly note that the function $h( \cdot, c) : \mathcal{ML}(\Sigma) \rightarrow \mathbb{R}$ is measurable. Indeed, fix a measured lamination $\lambda_{1}$ and consider $\lambda_{2}$ very close by. Using train track coordinates, there should be some train track $\tau$ carrying both, and transverse measures induced on $\tau$ by $\lambda_{1}, \lambda_{2}$ are very close to one another. Thus, any homotopy of $c$ yields intersection numbers with $\lambda_{1}, \lambda_{2}$ that are close together, which implies that $h(c, \lambda_{2})$ is at most some small $\epsilon$ plus $h(c, \lambda_{1})$. Thus the function $h (\cdot, c)$ is upper semi-continuous and therefore measurable. A very similar argument shows that the set 
\[ U_{\lambda}(\epsilon):= \left\{ c \in \widetilde{\mathcal{HC}}_{k}(\Sigma): h(c, \lambda)< \epsilon \right\}\]
is open. We record this as a lemma for future reference: 

\begin{lemma} \label{open sets intersection lamination} The set 
\[ U_{\lambda}(\epsilon):= \left\{ c \in \widetilde{\mathcal{HC}}_{k}(\Sigma): h(c, \lambda)< \epsilon \right\}\]
it open, and so in particular $U_{\lambda}(\epsilon) \cap \mathcal{ML}(\Sigma)$ is open in the space of measured laminations.  
\end{lemma}
\end{remark}

In general, $h$ can be much less than $i$; the purpose of the next lemma is to demonstrate that when pairing a harmonic current with corners against the Liouville current, the two are comparable because being harmonic implies not being too far from length-minimizing:

\begin{lemma} \label{harmonic to length} Suppose $\nu \in \widetilde{\mathcal{HC}}_{k}$ is a harmonic graph $k$-current with all weights at most $M$ and so that the ratio of largest to smallest positive weights is at most $N$. Then 
\[ i(\nu, \mathcal{L}) \leq D \cdot h(\nu, \mathcal{L}), \]
where $D$ is a constant depending only on $k,M,$ and $N$. 
\end{lemma}

\begin{proof} The one-sentence idea is that $\nu$ minimizes a weighted $l^{2}$-norm; since weighted $l^{1}$ and $l^{2}$ norms are comparable in terms of the weight vector and the dimension, we get a contradiction if some other marked graph in the same homotopy class has a much smaller weighted $l^{1}$-norm than $\nu$. We spell this out formally as follows:

The energy of $\nu$ is the sum, taken over all edges $e$ of $\Gamma$ (the marking graph of $\nu$), of the product of $w(e)$ and the hyperbolic length of $\phi(e)$ where $\phi$ is the marking map of $\nu$:
\begin{equation} \label{energy}
\mathcal{E}(\nu) = \sum_{e \in E(\Gamma)} w(e) \cdot \ell^{2}(\phi(e)). 
\end{equation}
Let $m$ denote the smallest weight of $\nu$ and let $L$ denote the total unweighted hyperbolic length of $\nu$, i.e., 
\[ L = \sum_{e \in E(\Gamma)}  \ell_{\Sigma}(\phi(e)). \]
Then using the standard inequality relating $l^{1}$ and $l^{2}$ norms, \ref{energy} is at least
\[ m \cdot \frac{L^{2}}{k}. \]
Now, let $\nu'$ be a current with corners whose marking graph is $\Gamma$, whose marking map is homotopic to $\phi$, and so that $h(\nu, \mathcal{L}) = i(\nu', \mathcal{L})$. Let $L'$ be the unweighted hyperbolic length of $\nu'$ and suppose that $L' < L/(\sqrt{Nk})$. Then the energy of $\nu'$ is largest if all of its length is concentrated on the edge with largest weight, and therefore 
\[ \mathcal{E}(\nu') \leq M(L')^{2} < M \frac{L^{2}}{Nk} < \mathcal{E}(\nu), \]
 contradicting the assumption that $\nu$ is harmonic. Therefore, $L'$ is at least $L/(\sqrt{Nk})$, and the \textit{weighted} length of $\nu'$ is smallest when all of its unweighted length concentrates on the edge of smallest weight; thus
 \[ i(\nu', \mathcal{L}) \geq \frac{m \cdot L}{\sqrt{Nk}}.\]
 On the other hand, the weighted length of $\nu$ is largest if all of its unweighted length concentrates on the edge with largest weight, and so 
 \[ i(\nu, \mathcal{L}) \leq M \cdot L, \]
 whence it follows that 
 \[ i(\nu, \mathcal{L}) \leq \frac{M \sqrt{Nk}}{m} i(\nu', \mathcal{L}) = N \sqrt{Nk}\cdot i(\nu', \mathcal{L}). \]
\end{proof}

We will also need to know that the length of a classical current is uniformly comparable to its homotopical intersection number with a fixed filling harmonic current:

\begin{lemma} \label{replace with curve} Let $d$ be a marked filling harmonic current and $\gamma$ a filling curve. Then there is some $K \ge 1$ depending only on $d$ and $\gamma$ so that 
\[ \frac{1}{K}  \cdot h(\gamma, c) \leq h(c, d) \leq K  \cdot h(\gamma, c)\]
for any filling classical current $c$. 
\end{lemma}

\begin{proof} Since both $c$ and $\gamma$ are classical currents, $h(c, \gamma)= i(c, \gamma)$ by definition, and $i(c, \gamma)$ is uniformly comparable to $\ell_{\Sigma}(c)$, the length of $c$ on $\Sigma$, i.e., the intersection number between $c$ and the Liouville current for $\Sigma$ (and this comparability is in terms only of constants depending on $\gamma$). To bound $h(c,d)$ from above, note first that by definition, $h(c,d) \leq i(c,d)$, and $i(c, d)$ is uniformly comparable to $\ell_{\Sigma}(c)$ when $c$ is a (multiple of a) curve because each complementary region of the harmonic embedding for $d$ into $\Sigma$ has bounded diameter, and the number of intersection points per unit of length is bounded above in terms of the maximum degree (and in particular the number of edges) of $d$. Finally, $\ell_{\Sigma}(c)$ is uniformly comparable to $i(\gamma, c)= h(\gamma, c)$ in terms only of a constant depending on $\gamma$, by the same argument. This concludes the upper bound on $h(c,d)$ when $c$ is a curve multiple. If $c$ is not a curve multiple, consider a sequence $\left\{c_{j} \right\}$ of curves converging to $c$. By continuity of the intersection form on the space of classical currents, 
\[ h(\gamma, c_{j})= i(\gamma, c_{j}) \longrightarrow i(\gamma, c) =h(\gamma, c).\]
On the other hand, it may not be the case that $i(c_{j}, d) \rightarrow i(c,d)$, if there is an atom in the support of $c$ (but not in the support of any of the $c_{j}$'s) that intersects a vertex of $d$. But then a small homotopy of $d$ intersects the atom in the interior of an edge and so $h(c,d) \leq \lim_{j \rightarrow \infty} i(c_{j}, d)$. This concludes the proof of the upper bound.

For the lower bound, we note that there is $T>0$ and a filling multi-curve $\rho_{d}$, both depending only on $d$, such that that there is a representative of $\rho_{d}$ lying entirely on $d$ so that $\rho_{d}$ does not traverse any edge of $d$ more than $T$ times. Thus if $c$ is any classical current, 
\[ i(\rho_{d}, c) \leq \frac{1}{T} h(c,d). \]
This is because $i(\rho_{d}, c)$ can be calculated by taking the minimum of $i(\rho'_{d}, c)$, where $\rho'_{d}$ is any representative of $\rho_{d}$ up to homotopy (and this minimum is then achieved by the geodesic representative for $\rho_{d}$). On the other hand, there is $T'>0$ so that 
\[ i(\rho_{d}, c) \geq \frac{1}{T'} i(\gamma, c), \]
with $T'$ depending only on $\rho_{d}$ and $\gamma$. This concludes the proof of the lower bound.

\end{proof}

\begin{remark} \label{compactness with h} The proof of the upper bound of Lemma \ref{replace with curve} could be tweaked to show that $i(c,d) \leq K' \cdot i(\gamma, c)$ for some potentially larger $K'$. Indeed, one can choose $K= K'$ if $\lim_{j \rightarrow \infty} i(c_{j}, d) = i(c,d)$, but if not, the discrepancy is controlled by the $l^{1}$ norm of the vector of weights for the underlying graph of $d$. Indeed, if $c$ contains an atom that intersects a vertex, $i(c,d)$ could be larger than $\lim_{j \rightarrow \infty} i(c_{j}, d)$ because it includes mass coming from other edges incident to that vertex and which is missed by each curve in the sequence $\left\{c_{j} \right\}$, but then the difference is bounded above by the sum of all of the weights of other edges incident to that vertex. 

This observation, coupled with the lower bound of Lemma \ref{replace with curve}, implies there is some $D >0$ depending only on $d$ so that 
\[ h(c,d) \geq \frac{1}{D} i(c,d), \]
for $c$ any classical current.

\end{remark}


\section{The main counting result} \label{main theorem}

Before launching into the proof, state the main counting theorem using the terminology developed in previous sections and in a way that implies Theorem \ref{main theorem!}:

\begin{Theorem} \label{main} Let $c \in \widetilde{\mathcal{HC}}_{k} \setminus \mathcal{N}_{k}$ be a marked harmonic current corresponding to a triangulation with no ghost edges. Then 

\[ \lim_{L \rightarrow \infty} \frac{ \# (\phi \in \mathcal{MCG}(\Sigma) : \ell_{\Sigma}(\phi(c)) \leq L ) }{L^{6g-6}} = \frac{n_{c} m_{\Sigma}}{\textbf{m}}, \]
where $\ell_{\Sigma}(\phi(c))$ denotes the minimum length representative of $\phi(c)$, 
\[ n_{c} = \mu_{Th}(\lambda \in \mathcal{ML}(\Sigma) : h(c, \sigma) \leq 1), m_{\Sigma} = \mu_{Th}( \lambda \in \mathcal{ML}(\Sigma) : \ell_{\Sigma}(\lambda) \leq 1 ), \]
and
\[ \textbf{m} = \int_{\mathcal{M}}m_{\Sigma} \mbox{dvol}_{WP}. \] 
\end{Theorem}

We follow the strategy of Rafi-Souto~\cite{RafiSouto} to prove Theorem \ref{main}. In some places, their arguments will apply in our setting without alteration. There will however be some important subtleties and departures and we will make clear when those occur. 

First, for $k\in \mathbb{N}$ equal to the number of edges of the domain graph $G$ of $c$, define the measure $\nu^{L_{c}}$ on $\widetilde{\mathcal{HC}}_{k}$ via
\[ \nu^{L}_{c} = \frac{1}{L^{6g-6}} \sum_{\phi \in \mathcal{MCG}(\Sigma)} \delta_{\frac{1}{L} \cdot \phi(c)}. \]

Let $\widetilde{\mathcal{HC}}_{k}(M,N)$ refer to the set of all marked harmonic $k$-currents that are either geodesic currents, or simple marked graph currents with all weights at most $M$ and so that the maximum ratio between any two positive weights is at most $N$.  Then there is $N,M$ so that $\phi(c) \in \widetilde{\mathcal{HC}}_{k}(M,N)$ for all $\phi \in \mathcal{MCG}(\Sigma)$, and so that $t \cdot \phi(c) \in \widetilde{\mathcal{HC}}_{k}(M,N)$ where $0  \leq t \leq 1$ is a scalar.

Our goal will be to show that the family of measures $(\nu^{L}_{c})$ has a well-defined limit on $\widetilde{\mathcal{HC}}_{k}(M,N)$ as $L \rightarrow \infty$, and to identify this limit. 

We first show that there are accumulation points of the family, and that any such measure is automatically positive:

\begin{lemma}  \label{accumulation points} The set of measures $(\nu^{L}_{c})_{L}$ is precompact in the space of Radon measures on $\widetilde{\mathcal{HC}}_{k}(M,N)$. Moreover, each accumulation point is locally finite and positive.
\end{lemma}

\begin{proof} 

Fix a filling curve $\gamma$ as in Lemma \ref{replace with curve} and let $K$ be as in the statement of that lemma. As in Rafi-Souto~\cite{RafiSouto}, we will show that

\[ \limsup_{L \rightarrow \infty} \nu^{L}_{c}\left( \left\{ \lambda \in \widetilde{\mathcal{HC}}_{k}(M,N) : \ell_{\Sigma}(\lambda) \leq T \right\} \right) < \infty, \]

for all $T>0$. To see why this implies the lemma, note first that the above set is compact in $\widetilde{\mathcal{HC}}_{k}$ by Theorem \ref{compact} (and in particular Remark \ref{compactness for marked k currents}) since the Liouville current is filling and without bottlenecks. Since the space of probability measures is compact in the weak-$^{\ast}$ topology on a compact space, establishing the above allows one to identify each $\nu^{L}_{c}$ with a probability measure (after a rescaling) on a compact space. One then obtains accumulation points by exhausting $\widetilde{\mathcal{HC}}_{k}(M,N)$ with larger and larger compact sets associated to a sequence $T_{1}, T_{2}...$ of length thresholds going to infinity. We now compute: 

\[\nu^{L}_{c}\left( \left\{ \lambda \in \widetilde{\mathcal{HC}}_{k}(M, N) : \ell_{\Sigma}(\lambda) \leq T \right\} \right) \leq \nu^{L}_{c}\left( \left\{ \lambda \in \widetilde{\mathcal{HC}}_{k}(M,N) : h(\lambda, \mathcal{L})  \leq T \right\} \right),\]
since $i(\cdot, \cdot) \ge h(\cdot, \cdot).$
In turn, we have 
\[\nu^{L}_{c}\left( \left\{ \lambda \in \widetilde{\mathcal{HC}}_{k}(M,N) : h(\lambda, \mathcal{L})  \leq T \right\} \right) = \frac{1}{L^{6g-6}} \cdot \# \left( \phi \in \mathcal{MCG}(\Sigma): h \left( \frac{1}{L} \cdot \phi(c), \mathcal{L} \right) \leq  T \right) \]
by definition. By homogeneity of $h$, the right hand side is 
\[ \frac{1}{L^{6g-6}} \cdot  \# \left( \phi \in \mathcal{MCG}(\Sigma): h(\mathcal{L}_{\Sigma}, \phi(c)) \leq  T \cdot L \right), \]
 Since $h$ is $\mathcal{MCG}$-invariant, this is equal to 
\[ \frac{1}{L^{6g-6}} \cdot  \# \left( \phi \in \mathcal{MCG}(\Sigma): h(\phi^{-1}\mathcal{L}, c) \leq T \cdot L \right). \]
Applying Lemma \ref{replace with curve} gives that this is at most 
\begin{align*} \frac{1}{L^{6g-6}} \cdot  \# \left( \phi \in \mathcal{MCG}(\Sigma): h(\phi^{-1}\mathcal{L}, \gamma) \leq K \cdot T \cdot L \right) &=  \frac{1}{L^{6g-6}} \cdot  \# \left( \phi \in \mathcal{MCG}(\Sigma): h(\mathcal{L}, \phi(\gamma)) \leq K \cdot T \cdot L \right) \\
& = \frac{1}{L^{6g-6}} \cdot  \# \left( \phi \in \mathcal{MCG}(\Sigma): i(\mathcal{L}, \phi(\gamma)) \leq K \cdot T \cdot L \right) \\
& =  \frac{1}{L^{6g-6}} \cdot  \# \left( \phi \in \mathcal{MCG}(\Sigma): \ell_{\Sigma}(\phi(\gamma)) \leq K \cdot T \cdot L \right). \end{align*}

In~\cite{RafiSouto}*{\S 3}, the limit of the last quantity as $L \rightarrow \infty$ is shown to exist and depend only on $\Sigma, K, T$, and $\gamma$. Thus, for each $T$, the function $L \rightarrow \nu_{c}^{L} \left( \left\{ \lambda \in \widetilde{\mathcal{HC}}_{k}(M,N) : \ell_{\Sigma}(\lambda) \leq T \right\} \right)$ is bounded, as desired. Since the space of probability measures is compact, the accumulation points are Radon and thus locally finite. 

To show that accumulation points are positive, we first use Lemma \ref{harmonic to length} to establish 
\[ \nu^{L}_{c}\left( \left\{ \lambda \in \widetilde{\mathcal{HC}}_{k}(M,N) : \ell_{\Sigma}(\lambda) \leq T \right\} \right) \ge  \nu^{L}_{c}\left( \left\{ \lambda \in \widetilde{\mathcal{HC}}_{k}(M,N) : h(\lambda, \mathcal{L}) \leq T/D \right\} \right). \]

From here, a completely symmetric calculation using the lower bound in Lemma \ref{replace with curve} gives

\[ \nu^{L}_{c} \left( \left\{ \lambda \in \widetilde{\mathcal{HC}}_{k}(M,N): \ell_{\Sigma}(\lambda) \leq T \right\} \right) \geq \frac{1}{L^{6g-6}} \cdot \# \left(\phi \in \mathcal{MCG}(\Sigma): \ell_{\Sigma}(\phi(\gamma)) \leq \frac{L}{C\cdot D} \right), \]
which again, by the arguments in~\cite{RafiSouto}*{\S 3}, approaches a positive limit as $L \rightarrow \infty$.

\end{proof}

\begin{remark} \label{subsequence} We remark that Lemma \ref{accumulation points} proves more than the existence of accumulation points: the family $(\nu^{L}_{c})_{L}$ is pre-compact in the space of Radon measures on $\widetilde{\mathcal{HC}}_{k}(M,N)$ and thus for every sequence $(L_{n})$ with $L_{n} \rightarrow \infty$, there is a subsequence $(L_{n_{i}})$ so that $\lim_{i \rightarrow \infty}\nu^{L_{n_{i}}}_{c}$ exists. 
\end{remark}

Eventually, we will want to show that an accumulation point $\nu_{c}$ is $\mathcal{MCG}$-invariant (see Section \ref{establishing invariance}). However, a priori, $\nu_{c}$ may contain points of $\mathcal{N}_{k}$ in its support, and so it is not altogether clear how to make sense of $\mathcal{MCG}$-invariance. In the next lemma, we address this by showing that any accumulation point $\nu$ is supported on $\mathcal{ML}(\Sigma)$. We can therefore interpret any such $\nu$ as a measure on $\widetilde{\mathcal{HC}}_{k}(M,N) \setminus \mathcal{N}_{k}$ where the action is well-defined. 

\begin{lemma} \label{lamination} Any accumulation point $\nu$ of the family $(\nu_{c}^{L})_{L}$ is supported on the space of measured laminations in $\widetilde{\mathcal{HC}}_{k}(M,N)$. 
\end{lemma}

\begin{proof} We will prove that if $U \subset \widetilde{\mathcal{HC}}_{k}(M,N)$ is a small open neighborhood of some $c \notin \mathcal{ML}$ and disjoint from $\mathcal{ML}$, then $\nu_{c}(U)= 0$ for $\nu_{c}$ any limiting measure of the family $\left\{\nu^{L}_{c} \right\}$. By Lemma \ref{limit 0 implies lamination} and since $\mathcal{ML}$ is closed, we can assume there is $\epsilon >0$ so that for all $x \in U$, $i(x,x) \ge \epsilon$. Indeed, if every open neighborhood of $c$ contains a measured lamination in its closure, $c$ is a limit point of $\mathcal{ML}$. 

The $\nu_{c}^{L}$-measure of the subset of $U$ consisting of currents with length at most $T$ is given by
\[ \nu_{c}^{L}\left( j \in U: \ell_{\Sigma}(j) \leq T \right) = \frac{1}{L^{6g-6}} \cdot \# \left( \phi \in \mathcal{MCG}(\Sigma): \phi(c) \in U, \ell_{\Sigma}(\phi(c)) \leq T\cdot L \right). \]
Consider the integral of self-intersection with respect to $\nu^{L}_{c}$:
\begin{align*} \int_{\left\{ \lambda \in \widetilde{\mathcal{HC}}_{k}(M,N) : \ell_{\Sigma}(\lambda) \leq T \right\} } i(\lambda, \lambda) d\nu^{L}_{c}(\lambda) &= \frac{1}{L^{6g-6}}\sum_{\phi \in \mathcal{MCG}(\Sigma), \ell_{\Sigma}(\phi(c)) \leq T\cdot L} i \left( \frac{\phi(c)}{L}, \frac{\phi(c)}{L} \right)\\ &=  \frac{1}{L^{6g-6}} \sum_{\phi \in \mathcal{MCG}(\Sigma), \ell_{\Sigma}(\phi(c)) \leq T \cdot L} \frac{i(c, c)}{L^{2}}, \end{align*}
 Here, we have used homogeneity of $i(\cdot, \cdot)$ and the fact due to de Verdiere~\cite{deVerdiere} that harmonic triangulations are always embeddings.   

This is in turn equal to 
\[ \frac{ \#(\phi \in \mathcal{MCG}(\Sigma): \ell_{\Sigma}(\phi(c)) \leq L \cdot T)}{L^{6g-6}} \cdot \frac{ i(c,c)}{L^{2}} = \nu^{L}_{c}\left( \left\{ \lambda \in \widetilde{\mathcal{HC}}_{k}(M): \ell_{\Sigma}(\lambda) \leq T \right\} \right) \cdot \frac{ i(c,c)}{L^{2}}. \] 

The proof of Lemma \ref{accumulation points} implies that $\nu^{L}_{c}\left( \left\{ \lambda \in \widetilde{\mathcal{HC}}_{k}(M): \ell_{\Sigma}(\lambda) \leq T \right\} \right)$ is bounded both above and below independently of $L$. Dividing by this quantity gives the average of self-intersection with respect to $\nu^{L}_{c}$, and this average goes to $0$ as $L \rightarrow \infty$ owing to the $L^{2}$ term in the denominator. 

Since the $\nu^{L}_{c}$-average self-intersection goes to $0$ as $L \rightarrow \infty$, it follows that a vanishingly small proportion of the total $\nu^{L}_{c}$ mass can lie in $U$. Since the total $\nu^{L}_{c}$ mass is bounded independently of $L$, this implies that
\[ \liminf_{L \rightarrow \infty} \nu^{L}_{c}(\left\{\lambda \in U : \ell_{\Sigma}(\lambda)\leq T \right\} ) = 0,\]
whence it follows that $\nu_{c}(U)= 0$ by the second part of Proposition \ref{Portmanteau}, and applying the entire argument to a sequence $T_{1}< T_{2}< ...$ of length thresholds going to infinity.

\end{proof}
 Before moving on, we establish a notational abbreviation: henceforth, by $\widetilde{\mathcal{HC}}$, we will mean $\widetilde{\mathcal{HC}}_{k}(M,N)$.

\subsection{Mapping class group invariance of the limiting measure} \label{establishing invariance}

Since the $\mathcal{MCG}(\Sigma)$-action on $\widetilde{\mathcal{HC}}$ may not be continuous, it is not automatically the case that $\nu_{c}$ is $\mathcal{MCG}(\Sigma)$-invariant even though each $\nu^{L}$ clearly is. Indeed, as the following simple example demonstrates -- shown to us by Anush Tserunyan~\cite{Tserunyan}-- if a group $G$ admits a Borel action on some space $X$, it needn't be true that the set of $G$-invariant measures is closed in the weak$^{\ast}$-topology:

\begin{Example} \label{Anush's example} Let $X = \left\{1/n: n \in \mathbb{N} \right\} \cup \left\{0 \right\}$, equipped with the usual subspace topology, and let $\mu_{n}$ be the atomic measure assigning mass $1$ to $\frac{1}{n}$. Finally, let $T: X \rightarrow X$ be the identity for each $1/n \neq 1$, and define $T(1)= 0, T(0)= 1$. Note that $\mu_{n}$ is $T$-invariant for each $n$. On the other hand, $\lim_{n \rightarrow \infty} \mu_{n}$ is the atomic measure assigning mass $1$ to $0$, and this is not $T$-invariant. 
\end{Example}

To verify $\mathcal{MCG}(\Sigma)$-invariance, we will use Proposition \ref{Portmanteau}. First, we notice how Proposition \ref{Portmanteau} interacts with Example \ref{Anush's example}: if $U \subset X$ is any open neighborhood of $0$, then the limiting measure gives non-zero mass to the topological boundary of $T(U)$. This motivates the following lemma: 

\begin{lemma} \label{general criterion for invariance} Let $T: X \rightarrow X$ be Borel and suppose that $\lim_{n \rightarrow \infty} \mu_{n} = \mu$ in the sense of weak$^{\ast}$-convergence, and so that $\mu_{i}$ is a Radon probability measure and $T$-invariant for each $i$. Suppose further that $A \subset X$ is open and that $T(A)$ can be outer-approximated by open sets $\left\{U_{i} \right\}$ so that $\mu(\partial U_{i})= \mu(\partial A) = 0$. Then 
\[ \mu(A) = \mu(T(A)). \]
\end{lemma}

\begin{proof} Since $\mu(\partial A)= 0$, Proposition \ref{Portmanteau} implies that 
\[ \mu(A) = \lim_{n \rightarrow \infty} \mu_{n}(A). \]
Since each $\mu_{n}$ is $T$-invariant, this in turn is equal to $\lim_{n \rightarrow \infty} \mu_{n}(T(A))$. Since $\mu_{i}$ is Radon, this can be rewritten as the double limit 
\[ \lim_{i \rightarrow \infty} \lim_{n \rightarrow \infty} \mu_{n}(U_{i}). \]
Finally, since $\mu(\partial U_{i})= 0$, again by Proposition \ref{Portmanteau}, this is equal to 
\[ \lim_{i \rightarrow \infty} \mu(U_{i}) = \mu(T(A)), \]
where the last equality is because $\mu$ is Radon. 
\end{proof}

Therefore, to prove $\mathcal{MCG}$-invariance of $\nu_{c}$, it will suffice to find an abundance of open sets whose images under the mapping class group action can be outer approximated by sets whose topological boundaries are $\nu_{c}$-null. Towards this goal, given $\lambda \in \mathcal{ML}(\Sigma)$ a minimal filling lamination, we focus on sets of the form 
\[U_{\lambda}(\epsilon):= \left\{ j \in \widetilde{\mathcal{HC}}: h(j, \lambda)< \epsilon \right\}. \]
The intersection of all such sets with $\mathcal{ML}(\Sigma)$ form a base for the standard topology, and moreover, the collection of these open sets is preserved by the $\mathcal{MCG}(\Sigma)$ action. Using Lemma \ref{open sets intersection lamination}, the topological boundary $\partial U_{\lambda}(\epsilon)$ can be described as 
\[ \partial U_{\lambda}(\epsilon)= \left\{ j \in \widetilde{\mathcal{HC}}: h(j, \lambda)= \epsilon \right\}. \]

\begin{lemma} \label{zero to boundaries} For any $L_{n} \rightarrow \infty$ such that $\lim_{n \rightarrow \infty} \nu^{L_{n}}_{c} = \nu_{c}$ exists, one has 
\[ \nu_{c} \left( \partial U_{\lambda}(\epsilon) \right) = 0 \]
for every $\epsilon>0$ and every minimal filling lamination $\lambda$. 
\end{lemma} 

\begin{proof} Our argument is inspired by~\cite{RafiSouto}*{Lemma 4.4} and our strategy follows theirs very closely. First observe that by Lemma \ref{replace with curve}, there is some $K$ and some filling curve $\eta$ on $\Sigma$ so that 
\[ \frac{1}{K} \cdot \ell_{\Sigma}(\lambda) \leq h(c, \lambda) \leq K \ell_{\Sigma}(\lambda), \]
and 
\[ \frac{1}{K} \cdot i(\lambda, \eta) \leq h(c, \lambda) \leq K \cdot i(\lambda, \eta)\]
for all classical currents $\lambda$ (simply apply Lemma \ref{replace with curve} twice-- once for the Liouville current and once for $\eta$-- and choose the larger of the two associated constants).

It suffices to prove that for all $T >0$, 
\[ \lim_{\delta \rightarrow 0} \limsup_{L \rightarrow \infty}\nu^{L}_{c} \left( \left\{ j \in \widetilde{\mathcal{HC}} : |i(j, \lambda) - \epsilon| \le \delta \cdot T  \hspace{2 mm} \mbox{and} \hspace{2 mm} \ell_{\Sigma}(j) \leq T \right\} \right) = 0. \]
Here we are using Proposition \ref{Portmanteau}, Lemma \ref{intersection function continuous} and the fact that $\lambda$ is without bottlenecks (which implies that the set of harmonic currents whose intersection with $\lambda$ lies in $(\epsilon- \delta \cdot T, \epsilon + \delta \cdot T)$ is open) to argue that if the above holds, the $\nu_{c}$ measure of the set of currents $j \in \widetilde{\mathcal{HC}}$ satisfying $i(\lambda, j) = \epsilon$, is null. Then we use Lemma \ref{lamination} to replace $i$ with $h$ and obtain the desired result.

Fixing $T, \epsilon>0$, one has that 
\begin{align*} & \nu^{L}_{c} \left( \left\{ j \in \widetilde{\mathcal{HC}}_{k} : |i(j, \lambda) - \epsilon| \leq \delta \cdot T  \hspace{2 mm} \mbox{and} \hspace{2 mm} \ell_{\Sigma}(j) \leq T \right\} \right)  
\\ &= \frac{1}{L^{6g-6}} \cdot \#\left( \phi \in \mathcal{MCG}(\Sigma) : |i(\phi(c), \lambda)- \epsilon \cdot L| \leq \delta \cdot T \cdot L \hspace{2 mm} \mbox{and} \hspace{2 mm} \ell_{\Sigma}(\phi(c)) \leq T \cdot L \right) \\ &\leq \frac{1}{L^{6g-6}} \cdot \#\left( \phi \in \mathcal{MCG}(\Sigma) : |i(\phi(\eta), \lambda)- \epsilon \cdot L \cdot K| \leq \delta \cdot T \cdot L \cdot K \hspace{2 mm} \mbox{and} \hspace{2 mm} \ell_{\Sigma}(\phi(\eta)) \leq T \cdot L \cdot K \right) \\ &= \nu^{L}_{\eta} \left( \left\{ \rho \in \widetilde{\mathcal{HC}}: |i(\rho, \lambda) - \epsilon \cdot K| \leq \delta \cdot T \cdot K \hspace{2 mm} \mbox{and} \hspace{2 mm} \ell_{\Sigma}(\rho) \leq T \cdot K \right\} \right) \\
&= \nu^{L}_{\eta} \left( \left\{ \rho \in \mathcal{C}(\Sigma) : |i(\rho, \lambda)- \epsilon \cdot K| \leq \delta \cdot T \cdot K \hspace{2 mm} \mbox{and} \hspace{2 mm} \ell_{\Sigma}(\rho) \leq T \cdot K \right\} \right), \end{align*}

By~\cite{ErlandssonSouto}*{Proposition 4.1}, we have a $t>0$ such that $$\lim_{L \rightarrow \infty} \nu_L(\eta) = t\mu_{Th}(\eta).$$ Thus, using Proposition \ref{Portmanteau} and basic properties of the Thurston measure yields
\begin{align*} & \limsup_{L \rightarrow \infty} \nu^{L}_{c}  \left( \left\{ \rho \in \widetilde{\mathcal{HC}} : |h(\rho, \lambda) - \epsilon| \leq \delta \cdot T  \hspace{2 mm} \mbox{and} \hspace{2 mm} \ell_{\Sigma}(\rho) \leq T  \right\} \right) \\ & \leq t \cdot \mu_{Th} \left( \left\{ \rho \in \mathcal{ML} : |i(\lambda, \rho) - \epsilon \cdot K| \leq \delta \cdot T \cdot K \hspace{2 mm} \mbox{and} \hspace{2 mm} \ell_{\Sigma}(\rho) \leq T \cdot K \right\} \right).\end{align*}
Since, by Lemma \ref{boundary zero}, the Thurston measure assigns $0$ to every set of laminations intersecting a given lamination a fixed amount, the lemma follows by letting $\delta \rightarrow 0$. 
\end{proof}

We are now ready to prove that $\nu_{c}$ must be $\mathcal{MCG}$-invariant. Since the collection of sets of the form $U_{\lambda}(\epsilon)$, when intersected with $\mathcal{ML}$, form a base for the topology on $\mathcal{ML}(\Sigma)$, and $\nu_{c}$ is supported on $\mathcal{ML}$, it suffices to prove that 
\[ \nu_{c}(U_{\lambda}(\epsilon)) = \nu_{c}(f(U_{\lambda}(\epsilon)) \]
for any $f \in \mathcal{MCG}(\Sigma)$. And indeed, this follows from Lemma \ref{general criterion for invariance}, Lemma \ref{zero to boundaries}, and the fact that $f(U_{\lambda}(\epsilon))$ is open. 

\begin{remark}
Applying Proposition \ref{Portmanteau} requires the measures in our sequence to be probability measures. Note that it suffices to prove invariance for all $U_{\lambda}(\epsilon)$ for all $\epsilon$ sufficiently small. Hence we can assume that any $\mu \in U_{\lambda}(\epsilon)$ has bounded length and mass (depending on the length and mass of $\lambda$). Thus $U_{\lambda}(\epsilon)$ lies entirely in 
\[ \left\{ \nu \in \widetilde{\mathcal{HC}}_{k}(M,N): \ell_{\Sigma}(\nu) \leq T \right\}\]
for some $T, M, N$. And by the proof of Lemma \ref{accumulation points}, $\nu^{L}_{c}$ has bounded total mass on this set independent of $L$. 
\end{remark}

\subsection{Identifying the limiting measure}

At this stage, we have shown that any accumulation point $\nu_{c}$ is $\mathcal{MCG}(\Sigma)$-invariant, positive, and supported on the space of measured laminations. Lindenstrauss-Mirzakhani~\cite{LindenstraussMirzakhani} characterize the Thurston measure class as the unique $\mathcal{MCG}$-invariant class assigning $0$ to every set of laminations disjoint from a given simple closed curve. In the next lemma, we therefore verify this last property in order to conclude that $\nu_{c}$ must be a multiple of $\mu_{Th}$:

\begin{lemma} \label{zero to curve missing sets} For any $L_{n} \rightarrow \infty$ such that $\lim_{n \rightarrow \infty} \nu^{L_{n}}_{c} = \nu_{c}$ exists, one has 
\[ \nu_{c} \left( \left\{ \lambda \in \mathcal{ML}(\Sigma): i(\lambda, \gamma) = 0  \right\} \right) = 0 \]
for every simple closed curve $\gamma$. 
\end{lemma} 

\begin{proof} The proof is almost identical to the proof of Lemma \ref{zero to boundaries}, although perhaps technically a little simpler. We include it for the sake of completeness. 

Again, by Lemma \ref{replace with curve}, there is some $K$ and some filling curve $\eta$ on $\Sigma$ so that 
\[ \frac{1}{K} \cdot \ell_{\Sigma}(\lambda) \leq h(c, \lambda) \leq K \ell_{\Sigma}(\lambda), \]
and 
\[ \frac{1}{K} \cdot i(\lambda, \eta) \leq h(c, \lambda) \leq K \cdot i(\lambda, \eta)\]
for all classical currents $\lambda$ (simply apply Lemma \ref{replace with curve} twice-- once for the Liouville current and once for $\eta$-- and choose the larger of the two associated constants). Fixing $T>0$ and some small $\epsilon$, one has that 
\begin{align*} & \nu^{L}_{c} \left( \left\{ \lambda \in \widetilde{\mathcal{HC}} : h(\lambda, \gamma) \leq \epsilon \cdot T  \hspace{2 mm} \mbox{and} \hspace{2 mm} \ell_{\Sigma}(\lambda) \leq T \right\} \right)  
\\ &= \frac{1}{L^{6g-6}} \cdot \#\left( \phi \in \mathcal{MCG}(\Sigma) : h(\phi(c), \gamma) \leq \epsilon \cdot T \cdot L \hspace{2 mm} \mbox{and} \hspace{2 mm} \ell_{\Sigma}(\phi(c)) \leq T \cdot L \right) \\ &\leq \frac{1}{L^{6g-6}} \cdot \#\left( \phi \in \mathcal{MCG}(\Sigma) : h(\phi(\eta), \gamma) \leq \epsilon \cdot T \cdot L \cdot K \hspace{2 mm} \mbox{and} \hspace{2 mm} \ell_{\Sigma}(\phi(\eta)) \leq T \cdot L \cdot K \right) \\ &= \nu^{L}_{\eta} \left( \left\{ \lambda \in \widetilde{\mathcal{HC}}: h(\lambda, \gamma) \leq \epsilon \cdot T \cdot K \hspace{2 mm} \mbox{and} \hspace{2 mm} \ell_{\Sigma}(\lambda) \leq T \cdot K \right\} \right) \\
&= \nu^{L}_{\eta} \left( \left\{ \lambda \in \mathcal{C}(\Sigma) : i(\lambda, \gamma) \leq \epsilon \cdot T \cdot K \hspace{2 mm} \mbox{and} \hspace{2 mm} \ell_{\Sigma}(\lambda) \leq T \cdot K \right\} \right), \end{align*}

where the last equality follows from the fact that the mapping class group orbit of $\eta$ lies entirely in $\mathcal{C}(\Sigma)$. By~\cite{ErlandssonSouto}*{Proposition 4.1}$\lim_{L \rightarrow \infty}\nu^{L}_{\eta}$ is some multiple $t$ of the Thurston measure $\mu_{Th}$. Hence 
\[ \lim_{\epsilon \rightarrow 0} \limsup_{L \rightarrow \infty}  \nu^{L}_{c} \left( \left\{ \lambda \in \widetilde{\mathcal{HC}} : h(\lambda, \gamma) \leq \epsilon \cdot T  \hspace{2 mm} \mbox{and} \hspace{2 mm} \ell_{\Sigma}(\lambda) \leq T \right\} \right) \leq \lim_{\epsilon \rightarrow 0} t \cdot \mu_{Th} \left( \left\{ \lambda \in \mathcal{ML} : i(\lambda, \gamma) \leq \epsilon \cdot T \cdot K \hspace{2 mm} \mbox{and} \hspace{2 mm} \ell_{\Sigma}(\lambda) \leq T \cdot K \right\} \right), \]
for some $t>0$. Finally, this is in turn $0$ because $\mu_{Th}$ assigns measure $0$ to any set of laminations disjiont from a given simple closed curve. 
\end{proof}

It now follows that if $\nu_{c} = \lim_{n \rightarrow \infty} \nu^{L_{n}}_{c}$ for any convergent sequence $\left\{\nu^{L_{n}}_{n} \right\}_{n}$ of measures in the family, $\nu_{c} = t \cdot \mu_{Th}$ for some $t$ potentially dependent on the chosen sequence. 

It remains to show that $t$ does not depend on the particular sequence, and to determine its value. Lemma \ref{accumulation points} and Remark \ref{subsequence} imply that for any $ \left\{ L_{n} \right\}$ with $L_{n} \rightarrow \infty$, there is a subsequence $\left\{ L_{n_{i}} \right\}$ so that the limit $\lim_{i \rightarrow \infty} \nu^{L_{n_{i}}}_{c} = \nu_{c}$ exists. Thus there is some $t$ so that $\nu_{c} = t \cdot \mu_{Th}$. Following Rafi-Souto, we will show that 
\[ t=  \frac{n_{c}}{\textbf{m}}, \]
which in particular proves that it is independent of the particular sequence. We require the following lemma, following ~\cite{ErlandssonSouto}*{Corollary 4.4}: 

\begin{lemma} \label{ES combined}  Suppose $\eta \in \mathcal{C}(\Sigma)$ and $c \in \widetilde{\mathcal{HC}}$ be filling. Then for any multi-curve $\gamma$,   
\[ \lim_{L \rightarrow \infty} \frac{ \# \left( \phi \in \mathcal{MCG}(\Sigma) : i(\eta, \phi(\gamma)) \leq L \right)}{\# \left( \phi \in \mathcal{MCG}(\Sigma) : h(c, \phi(\gamma)) \leq L \right)} = \frac{ \mu_{Th} \left( \left\{ \lambda \in \mathcal{ML}(\Sigma) : i(\eta, \lambda) \leq 1 \right\} \right)}{ \mu_{Th} \left( \left\{ \lambda \in \mathcal{ML}(\Sigma) : h(c, \lambda) \leq 1 \right\} \right)} .\]
\end{lemma}

\begin{proof}  It suffices to show that every sequence has a subsequence for which the claim follows. Given a sequence $\left\{ L_{n} \right\}$, choose some subsequence $\left\{L_{n_{i}} \right\}$ for which the limit $\nu = \lim_{i \rightarrow \infty} \nu^{L_{n_{i}}}_{c}$ exists. Thus there is $t$ so that $\nu_{c} = t \cdot \mu_{Th}$. 

We then compute:

\[ \lim_{i \rightarrow \infty} \frac{1}{L_{n_{i}}^{6g-6}} \cdot \# \left(\phi \in \mathcal{MCG}(\Sigma) : h(c, \phi(\gamma)) \leq L_{n_{i}} \right) \]
\[= \lim_{i \rightarrow \infty} \nu^{L_{n_{i}}}_{\gamma} \left( \left\{ \rho \in \widetilde{\mathcal{HC}}: h(c, \rho) \leq 1 \right\} \right)\]
by definition of the measure $\nu^{L_{n_{i}}}_{\gamma}$. Since the $\mathcal{MCG}$-orbit of $\gamma$ lives in $\mathcal{GC}(\Sigma)$, this is in turn equal to 
\[ \lim_{i \rightarrow \infty} \nu^{L_{n_{i}}}_{\gamma} \left( \left\{ \rho \in \mathcal{GC}(\Sigma) : h(c, \rho) \leq 1 \right\} \right). \]
Now, Lemma \ref{boundary zero} and Remark \ref{measurable} implies that 
\[ \mu_{Th}\left( \left\{ \rho \in \mathcal{GC}(\Sigma): h(c, \rho) = 1 \right\} \right) = 0,  \]
and therefore, using Proposition \ref{Portmanteau}, we can identify the above limit with 
\[ t \cdot \mu_{Th} \left( \left\{ \rho \in \mathcal{GC}(\Sigma) : h(c, \rho) \leq 1 \right\} \right) =  t \cdot \mu_{Th} \left( \left\{ \rho \in \mathcal{ML}(\Sigma) : h(c, \rho) \leq 1 \right\} \right). \]

The same argument applied to $\eta$ in place of $c$ yields 
\[  \lim_{i \rightarrow \infty} \frac{1}{L_{n_{i}}^{6g-6}} \cdot \# \left(\phi \in \mathcal{MCG}(\Sigma) : i(\eta, \phi(\gamma)) \leq L_{n_{i}} \right) = t \cdot \mu_{Th} \left( \left\{ \rho \in \mathcal{ML}(\Sigma) : i(\eta, \rho) \leq 1 \right\} \right). \]
The lemma now follows by taking the quotient of these expressions. 

\end{proof}

Applying Lemma \ref{ES combined} to the Liouville current at $\Sigma$ and to $c$, one obtains 

\[ \lim_{L \rightarrow \infty} \frac{ \# \left( \phi \in \mathcal{MCG}(\Sigma): h(c, \phi(\gamma)) \leq L \right)}{\# \left( \phi \in \mathcal{MCG}(\Sigma): \ell_{\Sigma}(\phi(\gamma)) \leq L \right)} = \frac{m_{c}}{m_{\Sigma}}. \]

Moreover, by~\cite{RafiSouto}*{Theorem 3.1}, we have
\[ \lim_{L \rightarrow \infty} \frac{ \# \left(\phi \in \mathcal{MCG}(\Sigma) : \ell_{\Sigma}(\phi(\gamma)) \leq L \right)}{L^{6g-6}} = \frac{m_{\Sigma} m_{\gamma}}{\textbf{m}}. \]

Combining this with Lemma \ref{ES combined} yields an analog of~\cite{RafiSouto}*{Corollary 3.5}: 

\[ \lim_{L \rightarrow \infty} \frac{1}{L^{6g-6}} \cdot \#(\phi \in \mathcal{MCG}(\Sigma) : h(c, \phi(\gamma)) \leq L) = \frac{m_{\gamma} m_{c}}{\textbf{m}}. \] 

We can now compute the value of $t$, as follows: 

\begin{align*} \nu_{c} \left( \left\{ \lambda \in \widetilde{\mathcal{HC}} : h(\lambda, \gamma) \leq 1 \right\} \right) &= \lim_{i \rightarrow \infty} \nu_{c}^{L_{n_{i}}} ( \left\{ \lambda \in \widetilde{\mathcal{HC}} : h(\lambda, \gamma) \leq 1 \right\}) \\
&= \lim_{i \rightarrow \infty} \frac{1}{L_{n_{i}}^{6g-6}} \cdot \#( \phi \in \mathcal{MCG}: h(\phi(c), \gamma) \leq L_{n_{i}} ) \\
&= \lim_{i \rightarrow \infty} \frac{1}{L_{n_{i}}^{6g-6}} \cdot \#( \phi \in \mathcal{MCG}: h(c, \phi(\gamma)) \leq L_{n_{i}} )\\
&= \frac{m_{\gamma}m_{c}}{\textbf{m}}. \end{align*}

On the other hand, $\nu_{c} = t \cdot \mu_{Th}$, and thus 
\[ \nu_{c} \left( \left\{ \lambda \in \widetilde{\mathcal{HC}} : h(\lambda, \gamma) \leq 1 \right\} \right) = t \cdot m_{\gamma}. \]
It follows that 
\[ t= \frac{m_{c}}{\textbf{m}}. \]

We can now complete the proof of Theorem \ref{main}, following~\cite{RafiSouto}, more-or-less verbatim: 

\begin{align*} \#(\phi \in \mathcal{MCG}(\Sigma) : \ell_{\Sigma}(\phi(c)) \leq L) &= \left( \sum_{\phi \in \mathcal{MCG}(\Sigma)} \delta_{\phi(c)} \right) \left( \left\{ \eta \in \widetilde{\mathcal{HC}}: \ell_{\Sigma}(\eta) \leq L \right\} \right) \\ &=  \left( \sum_{\phi \in \mathcal{MCG}(\Sigma)} \delta_{\frac{1}{L} \cdot \phi(c)} \right) \left( \left\{ \eta \in \widetilde{\mathcal{HC}}: \ell_{\Sigma}(\eta) \leq 1 \right\} \right) \\ &= L^{6g-6} \cdot \nu_{c}^{L} \left( \left\{ \eta \in \widetilde{\mathcal{HC}}: \ell_{\Sigma}(\eta) \leq 1 \right\} \right).\end{align*}
Above, we have shown that 
\[ \lim_{L \rightarrow \infty} \nu^{L}_{c} = \frac{m_{c}}{\textbf{m}} \cdot \mu_{Th}, \]
and therefore 
\[ \lim_{L \rightarrow \infty} \frac{ \#(\phi \in \mathcal{MCG}(\Sigma) : \ell_{\Sigma}(\phi(c)) \leq L)}{L^{6g-6}} = \frac{m_{c}}{\textbf{m}} \mu_{Th}\left( \left\{ \eta \in \widetilde{\mathcal{HC}}: \ell_{\Sigma}(\eta) \leq 1 \right\} \right) = \frac{m_{c} m_{\Sigma}}{\textbf{m}}. \]
Similarly, 
\[ \#(\phi \in \mathcal{MCG}(\Sigma) : h(\phi(c), \mathcal{L}) \leq L) = L^{6g-6} \cdot \nu^{L}_{c}\left(\left\{\eta \in \widetilde{\mathcal{HC}}: h(\eta, \mathcal{L}) \leq 1 \right\} \right), \]
and since $\ell_{\Sigma}(\eta) = h(\eta, \mathcal{L})$ for all $\eta \in \mathcal{ML}$, we again have that 
\[ \lim_{L \rightarrow \infty} \frac{ \#(\phi \in \mathcal{MCG}(\Sigma) : h(\phi(c), \mathcal{L}) \leq L)}{L^{6g-6}} =  \frac{m_{c} m_{\Sigma}}{\textbf{m}}.\]


\begin{bibdiv}
\begin{biblist}

\bib{AL}{article}{
   author={Aramayona, Javier},
   author={Leininger, Christopher J.},
   title={Hyperbolic structures on surfaces and geodesic currents},
   conference={
      title={Algorithmic and geometric topics around free groups and
      automorphisms},
   },
   book={
      series={Adv. Courses Math. CRM Barcelona},
      publisher={Birkh\"{a}user/Springer, Cham},
   },
   isbn={978-3-319-60939-3},
   isbn={978-3-319-60940-9},
   date={2017},
   pages={111--149},
   review={\MR{3752040}},
}

\bib{AH}{article}{ 
author={Arana-Herrera, Francisco}, 
title={Normalization of Thurston measures on the space of measured geodesic laminations}, 
note={In preparation.},


}

\bib{Billingsley}{book}{
   author={Billingsley, Patrick},
   title={Convergence of probability measures},
   series={Wiley Series in Probability and Statistics: Probability and
   Statistics},
   edition={2},
   note={A Wiley-Interscience Publication},
   publisher={John Wiley \& Sons, Inc., New York},
   date={1999},
   pages={x+277},
   isbn={0-471-19745-9},
   review={\MR{1700749}},
   doi={10.1002/9780470316962},
}

\bib{BirmanSeries}{article}{
   author={Birman, Joan S.},
   author={Series, Caroline},
   title={Geodesics with bounded intersection number on surfaces are
   sparsely distributed},
   journal={Topology},
   volume={24},
   date={1985},
   number={2},
   pages={217--225},
   issn={0040-9383},
   review={\MR{0793185}},
   doi={10.1016/0040-9383(85)90056-4},
   
   
}

\bib{Bonahon1}{article}{
   author={Bonahon, Francis},
   title={Bouts des vari\'{e}t\'{e}s hyperboliques de dimension $3$},
   language={French},
   journal={Ann. of Math. (2)},
   volume={124},
   date={1986},
   number={1},
   pages={71--158},
   issn={0003-486X},
   review={\MR{0847953}},
   doi={10.2307/1971388},
}
\bib{Bonahon}{article}{
   author={Bonahon, Francis},
   title={The geometry of Teichm\"{u}ller space via geodesic currents},
   journal={Invent. Math.},
   volume={92},
   date={1988},
   number={1},
   pages={139--162},
   issn={0020-9910},
   review={\MR{0931208}},
   doi={10.1007/BF01393996},
}

\bib{Chen}{article}{
   author={Chen, Jingyi},
   title={On energy minimizing mappings between and into singular spaces},
   journal={Duke Math. J.},
   volume={79},
   date={1995},
   number={1},
   pages={77--99},
   issn={0012-7094},
   review={\MR{1340295}},
   doi={10.1215/S0012-7094-95-07903-4},
}

\bib{deVerdiere}{article}{
   author={Colin de Verdi\`ere, Yves},
   title={Comment rendre g\'{e}od\'{e}sique une triangulation d'une
   surface?},
   language={French},
   journal={Enseign. Math. (2)},
   volume={37},
   date={1991},
   number={3-4},
   pages={201--212},
   issn={0013-8584},
   review={\MR{1151746}},
}

\bib{ErlandssonSouto}{article}{
   author={Erlandsson, Viveka},
   author={Souto, Juan},
   title={Counting curves in hyperbolic surfaces},
   journal={Geom. Funct. Anal.},
   volume={26},
   date={2016},
   number={3},
   pages={729--777},
   issn={1016-443X},
   review={\MR{3540452}},
   doi={10.1007/s00039-016-0374-7},
}

\bib{ErlandssonSoutobook}{book}{
   author={Erlandsson, Viveka},
   author={Souto, Juan},
   title={Mirzakhani's curve counting and geodesic currents},
   series={Progress in Mathematics},
   volume={345},
   publisher={Birkh\"{a}user/Springer, Cham},
   date={[2022] \copyright 2022},
   pages={xii+226},
   isbn={978-3-031-08704-2},
   isbn={978-3-031-08705-9},
   review={\MR{4501245}},
   doi={10.1007/978-3-031-08705-9},
}

\bib{ESTripods}{article}{
author={Erlandsson, Viveka},
   author={Souto, Juan},
   title ={Counting geodesics of given commutator length},
      journal={ArXiv preprint},
      date={2023},
doi={10.48550/arXiv.2304.10274}
}

\bib{FarbMargalit}{book}{
   author={Farb, Benson},
   author={Margalit, Dan},
   title={A primer on mapping class groups},
   series={Princeton Mathematical Series},
   volume={49},
   publisher={Princeton University Press, Princeton, NJ},
   date={2012},
   pages={xiv+472},
   isbn={978-0-691-14794-9},
   review={\MR{2850125}},
}

\bib{Folland}{book}{
   author={Folland, Gerald B.},
   title={Real analysis},
   series={Pure and Applied Mathematics (New York)},
   edition={2},
   note={Modern techniques and their applications;
   A Wiley-Interscience Publication},
   publisher={John Wiley \& Sons, Inc., New York},
   date={1999},
   pages={xvi+386},
   isbn={0-471-31716-0},
   review={\MR{1681462}},
}

\bib{Huber}{article}{
   author={Huber, Heinz},
   title={Zur analytischen Theorie hyperbolischen Raumformen und
   Bewegungsgruppen},
   language={German},
   journal={Math. Ann.},
   volume={138},
   date={1959},
   pages={1--26},
   issn={0025-5831},
   review={\MR{0109212}},
   doi={10.1007/BF01369663},
}

\bib{KajigayaTanaka}{article}{
   author={Kajigaya, Toru},
   author={Tanaka, Ryokichi},
   title={Uniformizing surfaces via discrete harmonic maps},
   language={English, with English and French summaries},
   journal={Ann. H. Lebesgue},
   volume={4},
   date={2021},
   pages={1767--1807},
   review={\MR{4353977}},
   doi={10.5802/ahl.116},
}

\bib{LindenstraussMirzakhani}{article}{
   author={Lindenstrauss, Elon},
   author={Mirzakhani, Maryam},
   title={Ergodic theory of the space of measured laminations},
   journal={Int. Math. Res. Not. IMRN},
   date={2008},
   number={4},
   pages={Art. ID rnm126, 49},
   issn={1073-7928},
   review={\MR{2424174}},
   doi={10.1093/imrn/rnm126},
}

\bib{Margulis}{article}{
   author={Margulis, G. A.},
   title={Certain applications of ergodic theory to the investigation of manifolds of negative curvature.},
   language={Russian},
   journal={Funkcional. Anal. i Prilo\v{z}en.},
   date={1969},
   number={no. 4},
   pages={89--90},
   issn={0374-1990},
   review={\MR{0257933}},
}

\bib{Masur85}{article}{
   author={Masur, Howard},
   title={Ergodic actions of the mapping class group},
   journal={Proc. Amer. Math. Soc.},
   volume={94},
   date={1985},
   number={3},
   pages={455--459},
   issn={0002-9939},
   review={\MR{0787893}},
   doi={10.2307/2045234},
}

\bib{Mirzakhani}{article}{
   author={Mirzakhani, Maryam},
   title={Growth of the number of simple closed geodesics on hyperbolic
   surfaces},
   journal={Ann. of Math. (2)},
   volume={168},
   date={2008},
   number={1},
   pages={97--125},
   issn={0003-486X},
   review={\MR{2415399}},
   doi={10.4007/annals.2008.168.97},
}

\bib{Mirzakhani2}{article}{
   author={Mirzakhani, Maryam},
   title ={Counting Mapping Class group orbits on hyperbolic surfaces},
      journal={ArXiv preprint},
date={2016},
   doi={10.48550/arXiv.1601.03342},

}

\bib{Monin}{article}{
   author={Monin, Leonid},
   author={Telpukhovskiy, Vanya},
   title={On normalizations of Thurston measure on the space of measured
   laminations},
   journal={Topology Appl.},
   volume={267},
   date={2019},
   pages={106878, 12},
   issn={0166-8641},
   review={\MR{4008568}},
   doi={10.1016/j.topol.2019.106878},
}

\bib{PennerHarer}{book}{
   author={Penner, R. C.},
   author={Harer, J. L.},
   title={Combinatorics of train tracks},
   series={Annals of Mathematics Studies},
   volume={125},
   publisher={Princeton University Press, Princeton, NJ},
   date={1992},
   pages={xii+216},
   isbn={0-691-08764-4},
   isbn={0-691-02531-2},
   review={\MR{1144770}},
   doi={10.1515/9781400882458},
}

\bib{RafiSouto}{article}{
   author={Rafi, Kasra},
   author={Souto, Juan},
   title={Geodesic currents and counting problems},
   journal={Geom. Funct. Anal.},
   volume={29},
   date={2019},
   number={3},
   pages={871--889},
   issn={1016-443X},
   review={\MR{3962881}},
   doi={10.1007/s00039-019-00502-7},
}

\bib{Sapir1}{article}{
   author={Sapir, Jenya},
   title={Bounds on the number of non-simple closed geodesics on a surface},
   journal={Int. Math. Res. Not. IMRN},
   date={2016},
   number={24},
   pages={7499--7545},
   issn={1073-7928},
   review={\MR{3632090}},
   doi={10.1093/imrn/rnw032},
}

\bib{Sapir2}{article}{
   author={Sapir, Jenya},
   title={Lower bound on the number of non-simple closed geodesics on
   surfaces},
   journal={Geom. Dedicata},
   volume={184},
   date={2016},
   pages={1--25},
   issn={0046-5755},
   review={\MR{3547779}},
   doi={10.1007/s10711-016-0155-3},
}

\bib{Sunada}{book}{
   author={Sunada, Toshikazu},
   title={Topological crystallography},
   series={Surveys and Tutorials in the Applied Mathematical Sciences},
   volume={6},
   note={With a view towards discrete geometric analysis},
   publisher={Springer, Tokyo},
   date={2013},
   pages={xii+229},
   isbn={978-4-431-54176-9},
   isbn={978-4-431-54177-6},
   review={\MR{3014418}},
   doi={10.1007/978-4-431-54177-6},
}

\bib{Taylor}{book}{
   author={Taylor, Angus E.},
   title={General theory of functions and integration},
   edition={2},
   publisher={Dover Publications, Inc., New York},
   date={1985},
   pages={x+437},
   isbn={0-486-64988-1},
   review={\MR{0824243}},
}

\bib{Tserunyan}{article}{
author={Tserunyan, Anush},
title={personal communication},
date={2023}
}

\bib{Tuzhilin}{article}{
author={Tuzhilin, Alexey},
   title ={Who Invented the Gromov-Hausdorff Distance?},
      journal={ArXiv preprint},
      date={2016},
doi={10.48550/arXiv.1612.00728}
}


\end{biblist}
\end{bibdiv}
\end{document}